\documentclass[a4paper,12pt,titlepage,onecolumn,oneside]{book}
\usepackage[utf8]{inputenc}
\usepackage{amsmath,amssymb,graphicx,upgreek,appendix}
\usepackage{setspace,subfigure,psfrag,subfigure}
\usepackage[vlined,boxruled,longend]{algorithm2e}
\SetKwInOut{Input}{input}
\SetKwInOut{Output}{output}
\SetKw{KwDownto}{down to}
\IncMargin{1.2em}
\LinesNumbered
\usepackage{latexsym}
\usepackage{mathrsfs}
\usepackage{cite}
\usepackage{amssymb}
\usepackage{bm}
\usepackage{fullpage}
\usepackage{amsthm}
\usepackage{amsmath}
\usepackage{amssymb}
\usepackage[linktocpage]{hyperref}
\usepackage{enumerate}
\usepackage{amsmath, amssymb}
\usepackage{psfrag}
\usepackage{enumerate}
\usepackage{pdfpages}
\usepackage[section]{placeins} 
\numberwithin{equation}{chapter}

\newtheorem{definition}{Definition}
\newtheorem{fact}{Fact}

\newtheorem{theorem}{Theorem}
\newtheorem{lemma}{Lemma}

\newtheorem{corollary}[theorem]{Corollary}
\newtheorem{proposition}{Proposition}
\newtheorem{example}{Example}
\newtheorem{remark}{Remark}

\newcommand{\such}{\text{ such that }}
\newcommand{\eqdef}{\triangleq}

\newcommand{\cB}{\mathcal{B}}

\newcommand{\cF}{\mathcal{F}}

\newcommand{\mathset}[1]{\left\{#1\right\}}

\newcommand{\abs}[1]{\left|#1\right|}

\newcommand{\parenv}[1]{\left( #1 \right)}
\newcommand{\sparenv}[1]{\left[ #1 \right]}

\newcommand{\tends}[1]{\operatorname*{\longrightarrow}\limits_{#1}}
\newcommand{\norm}[1]{\left\lVert#1\right\rVert}

\DeclareMathOperator{\fin}{Fin}

\DeclareMathOperator{\topent}{h}
\DeclareMathOperator{\Irre}{Irr}
\DeclareMathOperator{\Pf}{Pf}
\DeclareMathOperator{\sgn}{sgn}
\DeclareMathOperator{\PM}{PM}
\DeclareMathOperator{\PC}{PC}
\DeclareMathOperator{\ima}{Img}
\DeclareMathOperator{\Int}{Int}
\DeclareMathOperator{\fix}{Fix}
\DeclareMathOperator{\Is}{Iso}
\DeclareMathOperator{\Linspan}{Span} 
\DeclareMathOperator{\Aut}{Aut} 
\DeclareMathOperator{\SL}{SL} 
\DeclareMathOperator{\GL}{GL} 
\renewcommand{\Bbb}{\mathbb}

\usepackage{xcolor}
\newcommand{\C}{{\Bbb C}}
\newcommand{\N}{{\Bbb N}}
\newcommand{\R}{{\Bbb R}}
\newcommand{\Z}{{\Bbb Z}}
\newcommand{\Pro}{{\Bbb P}}
\newcommand{\E}{{\Bbb E}}

\newcommand*{\fullref}[1]{\hyperref[{#1}]{\autoref*{#1}}}

\title{	
	\large{BEN--GURION UNIVERSITY OF THE NEGEV}\\[8pt]
	\large{FACULTY OF ENGINEERING SCIENCES}\\[8pt]
	\footnotesize{SCHOOL OF ELECTRICAL AND COMPUTER ENGINEERING}\\[72pt]
	\Large{PERMUTATIONS WITH RESTRICTED MOVEMENT}\\[60pt]
	\footnotesize{THESIS SUBMITTED IN PARTIAL FULFILLMENT OF THE REQUIREMENTS FOR THE M.Sc. DEGREE}\\[24pt]}	
\author{
	\footnotesize{By: Dor Elimelech}\\[24pt]\\[40pt]}
\date{September 2019}

\begin{document}
\maketitle

\begin{center}
	\textbf{Abstract}
\end{center}
We study restricted permutations of sets which have a geometrical structure.  The study of restricted permutations is motivated by their application in coding for flash memories, and their relevance in different applications of networking technologies and various channels.  We generalize the model of $\Z^d$-permutations with restricted movement suggested by Schmidt and Strasser in 2016, to restricted permutations of graphs, and study the new model in a symbolic dynamical approach. We show a correspondence between restricted permutations and perfect matchings. We use the theory of perfect matchings for investigating several two-dimensional cases, in which we compute the exact entropy and propose a polynomial-time algorithm for counting admissible patterns. We prove that the entropy of $\Z^d$-permutations restricted by a set with full affine dimension depends only on the size of the set. We use this result in order to compute the entropy for a class of two-dimensional cases. We discuss the global and local admissibility of patterns, in the context of restricted $\Z^d$-permutations. Finally, we review the related models of injective and surjective restricted functions.

\newpage
\begin{flushleft}
	\Huge\textbf{Index terms}
\end{flushleft}
Permutations, SFT, perfect matchings, entropy.

\newpage
\begin{flushleft}
	\textbf{Acknowledgments}
\end{flushleft}
I would like to evince my sincerest gratitude to my supervisors, Prof. Moshe Schwartz and Prof. Tom Meyerovitch, who kindly accepted me as a master student and guided me throughout the learning process of this thesis. I would like to thank them for challenging me when I needed challenging, and supporting me when I needed supporting. Their input and patience has been invaluable in helping me to learn how to do research, and to navigate some of the more emotionally challenging aspects of this project.

\tableofcontents

\newpage
\begin{flushleft}
	\Huge\textbf{Notation}
\end{flushleft}
\begin{itemize}
\item $\left[n \right]$ - the set of integers $\left\{ 0,1,2,\dots,n-1\right\} $.

\item $\left[m,n\right]$ the set of integers $\left\{ m,m+1,m+2,\dots,n\right\} $. 

\item $[n] \eqdef [n_1]\times [n_2] \times \cdots \times [n_d]\subset \mathbb{Z}^d$ - the cylinder set defined by a multi-index $n=\parenv{n_1,n_2,\dots,n_d}\in \mathbb{N}^d$.

\item $\fin(S)$ - the set finite subsets of a set $S$.

\item $a\bmod b=\left(a_{1}\bmod b_{1},a_{2}\bmod b_{2},\dots,a_{d}\bmod b_{d}\right)$ the modulus of $a\in \Z^d$ from $b\in \N^d$.

\item  $\sigma_n:\Z^d \to \Z^d$ the shift by $n$ operation: $\sigma_n(m)\eqdef n+m$.

\item  $A+B$ - the sum set:  $A+B\eqdef \mathset{a+b: a\in A, b\in B}$. 

\item $S(A)$ - the set of permutations of $A$.
\end{itemize}


\listoffigures


\mainmatter

\begin{spacing}{1.5}

\chapter{Introduction}
\label{CHA:introduction}
	In the last decade, permutations have received increased attention in the field of communication \cite{JiaMatSchBru08,KovPop14,KobVuk13,WalWeb08,LanSchYaa17, Kra94,ShaZeh91,PerCufWei08, KloLinTsaTze10,ShiTsa10, TamSch10,ZhoSchJiaBru15,TamSch12, SchTam11, FarSchBru16a,WanMazWor15,KarSch17 , YehSch12b, YehSch16, LanSchYaa17}. This is mainly due to their recent application in coding for flash memories, and their relevance in different applications of networking technologies and various channels.

Flash memories are non-volatile storage devices that are both electrically programmable and electrically erasable. The wide use of flesh memories is motivated by their high storage density and relative low cost. One of the most significant disadvantages of flash memories is the asymmetry between cell programming (charge placement) and cell erasing (charge removal). While cell programing is relatively a simple and fast operation, cell erasing is a difficult and complicated task.

The rank modulation coding scheme was proposed \cite{JiaMatSchBru09} in order to overcome this problem.  In the rank modulation coding scheme, information is stored in the form of permutations. More precisely, information is stored in the permutation suggested by sorting a group of cells by their relative charge values, instead of in the charge values themselves. The study of permutations and their  use in coding, which appears in the literature as early as the works  \cite{Sle65, ChaKur69}, was reignited by their latest use in the rank modulation scheme.

Permutations also play an important role in communication of information in the presence of synchronization errors, often modelled by permutation channels. Under this setup, a vector of symbols is transmitted in some order, but due to synchronization errors, the symbols received are
not necessarily in the order in which they were transmitted. Such  models were studied in \cite{KovPop14,KobVuk13,WalWeb08,LanSchYaa17} (permutation channels), \cite{Kra94,ShaZeh91} (the
bit-shift magnetic recording channel), and \cite{PerCufWei08} (the trapdoor channel).

So far, the research of permutations in the context of coding was focused on one-dimensional permutations. That is, permutations of the integer set $\mathset{1,2,\dots,n}$. Such a model overlooks  aspects concerning the positioning of objects in space. However, memory cells in flash technology can be ordered in two-dimensional or three-dimensional geometric structures. Motivated by this observation, we investigate permutations of sets with non-trivial geometric form.

In typical settings, permutation codes are considered  in the framework of metric spaces. One popular  metric for permutation spaces is the $\ell_\infty$ metric, also known by the name infinity metric.  Spaces of permutations with infinity metric have been used for error-correction \cite{KloLinTsaTze10,ShiTsa10, TamSch10,ZhoSchJiaBru15},  code relabeling \cite{TamSch12} anticodes \cite{SchTam11}, covering codes \cite{FarSchBru16a,WanMazWor15,KarSch17}, snake-in-the-box codes \cite{YehSch12b, YehSch16}, and codes of the limited permutation channel \cite{LanSchYaa17}.

Balls in a metric spaces and their parameters are a key elements in coding theory, as many coding-theoretic problems may be viewed as packing or covering of a metric space by balls. The extensive use of the infinity metric and the significance of balls for coding motivated the study of $\ell_\infty$-ball sizes \cite{ShiTsa11,Sch09,SchVon17, Klo08, Klo09}. It was already observed in \cite{Klo08} that the size of $\ell_\infty$-balls does not depend in the center of the ball and therefore it is sufficient to study the balls centred in the identity permutation. Permutations inside a ball centred in the identity permutation were called in \cite{Klo08, Klo09} by the name permutation with limited displacement, since they are exactly the permutations which satisfy $\abs{\pi(n)-n}\leq r$, where $r$ stands for the radius of the ball.

In this work, we generalize the concept of limited permutations by considering spaces of permutations of vertex sets of directed graphs, we label such permutations as restricted permutations. This generalization allows us to explore permutations of general sets, with restrictions that take into account geometrical structure, which can not be modelled in the standard settings of metric space. Given a graph $G=(V,E)$, we consider permutations of $V$ that respect the graph structure of $G$.  That is, permutations satisfying $(v,\pi(v))\in E$ for any vertex $v\in V$. We say that such permutations are restricted by $G$, and denote the set of such  permutations by $\Omega(G)$.

In Chapter \ref{CHA:preliminaries}, we show that under some assumptions, such permutation spaces can be interpreted as a topological dynamical systems. We discuss the important specific case of permutations of $\Z^d$ with movement restricted by some finite set. That is, permutations satisfying $\pi(n)-n\in A$ for all $n\in \Z^d$, where $A\subseteq \Z^d$ is some finite set. The concept of restricted $\Z^d$ permutations was introduced by Schmidt and Strasser in \cite{SchStr17}. They have shown that such permutations spaces are shifts of finite type (SFT). They have studied their topological and dynamical properties.

SFTs are mathematical structures from the field of symbolic dynamics,  which are often used in order to describe and study constrained coding problems. One-dimensional constrained  codes over permutations space  were studied in \cite{BuzYaa16,SalDol13}. In our work, we focus on multidimensional SFTs defined by restricted permutations of $\Z^d$. We study their topological entropy, which in the terminology of constrained coding, is called the capacity of the constrained system.

In Chapter \ref{CHA:PM}, we find a natural correspondence between restricted permutations of graphs and perfect matchings. Theorem \ref{th:GenCan} says that for any graph, there exists a bijection between restricted permutations and perfect matchings of a certain canonically derived bipartite graph. It also proved that this bijection respects the underlying dynamical structure (whenever there is such). In Theorem \ref{th:PerToPM} we describe a correspondence between restricted permutations of a bipartite graph, and  pairs of perfect matchings of  the original graph. We use those connections in order to find the exact topological entropy in a couple of two-dimensional cases. We do that by appealing to the theory of perfect matching of $\Z^2$-periodic planar graphs \cite{TemFis60,Fis66,ChoKenPro01,Fis61,Kas67,Kas61,Ken00,KenOkoShe06,Kas63}.

Chapter \ref{CHA:Entropy} is devoted to the study of entropy of $\Z^d$-permutations restricted by some finite set. We prove an important invariance property of the entropy under affine transformations. This property is later used in the proof of Theorem \ref{th:AffEqEnt}, where we present an exact expression for the entropy of permutations restricted by sets which consists of three elements. 	


\chapter{Preliminaries}
\label{CHA:preliminaries}
A topological dynamical system is a pair $(X,H)$, where $X$ is a topological compact space (which is usually also a metric space), and $H$ is a semigroup, acting on $X$ by continuous transformations. In this work we investigate topological dynmical systems defined by permutations of graphs.  

\begin{definition}
Let $G=(V,E)$ be a directed graph. A permutation, $\pi\in S(V)$, is said to be restricted by $G$ if for all $v\in V$, 
\[(v,\pi(v))\in E.\] 
We define $\Omega(G)$ to be the set of all permutations restricted by $G$. Formally,
\[ \Omega(G)\eqdef \mathset{\pi\in S(V) : \pi \text{ is restricted by }G}.\] 
Similarly, for an undirected graph $G=(V,E)$, $\pi\in S(V)$ is said to be restricted by $G$ if for all $v\in V$, 
\[\mathset{v,\pi(v)}\in E,\]  
and $\Omega(G)$ is defined to be the set of all permutations restricted by $G$.
\end{definition}

We observe that restricted permutations of an undirected graph can be equivalently defined by restricted permutations of directed graphs. If $G=(V,E)$ is an undirected graph then $\pi\in S(V)$ is restricted by $G$ if and only if it restricted by the directed graph $ G'=(V,E')$, where 
\[  E'\eqdef \mathset{(v,u)\in V\times V : \mathset{v,u}\in E}.\]


We focus on the case where $G=(V,E)$ is a countable directed graph, which is also locally finite (that is, any vertex has finite degree). Let $H$  be a group, acting on $G$ by graph isomorphisms. That is, there is a homomorphism $H\to \Is(G)$, where $\Is(G)$ is the group of graph isomorphisms of $G$. We will show that with the right settings, the set of $G$-restricted permutations is a compact topological space with $H$ acting on it continuously.

We consider $\Omega(G)$ as a topological space, with the pointwise convergence topology (where $G$ has the discrete topology). We claim that $\Omega(G)\subseteq V^V$ is a compact  space. For any vertex $v\in V$, denote the set of neighbours of $v$ in $G$ by $N(v)$, that is
\[ N(v)\eqdef \mathset{u\in V :(v,u)\in E}.\] 
We consider the set of functions $f:V\to V$ for which $(v,f(v))\in E$ for all $v\in V$. When we think of this set as a subset of $V^V$ (the set of all function from $V$ to $V$), we note that it is exactly the set $\prod_{v\in V}N(v)\subseteq V^V$. For any $v$, $N(v)$ is compact as a finite set.  Thus, by Tychonoff's Theorem \cite{Bou95}, $\prod_{v\in V}N(v)\subseteq V^V$ is compact as a product of compact spaces.

From the definition of restricted permutations, it immediately follows that $\Omega(G)\subseteq \prod_{v\in V}N(v) $. Hence, in order to show that $\Omega(G)$ is compact, it is sufficient to show that it is closed. Indeed, let $f\in \prod_{v\in V}N(v) \setminus \Omega(G)$, that is, $f$ is not a permutation of $V$. If there exists distinct $v_1,v_2\in V$ such that $f(v_1)= f(v_2)=u$, then the cylinder set 
\[ U\eqdef \mathset{g\in \prod_{v\in V}N(v) : g(v_1)=g(v_2)=u},\]
which is open by the definition of product topology, contains $f$. Clearly $U$ separates between $f$ and and $\Omega(G)$, as it contains only non injective functions. If $f$ is injective, then it is not onto $V$. In that case, there exists $v'\in V$ such that $f(v)\neq v'$ for any $v\in N(v')$. It is easy to verify that the set  
\[ \mathset{g\in \prod_{v\in V}N(v) : \forall v\in N(v'), g(v)\neq v'},\] 
which is also a cylinder set, separates between $f$ and and $\Omega(G)$. This shows that $\prod_{v\in V}N(v) \setminus \Omega(G)$ is open and completes the proof of the claim.

We consider the group action of $H$ on $\Omega(G)$ by conjugation, induced by the action of $H$ on $G$. That is, for $\pi\in\Omega(G)$ and $h\in H$, the action of $h$ on $\pi$ is defined to be
\[ \pi^h(v)\eqdef h(\pi(h^{-1}(x))). \]  
We claim that the action $\pi\to\pi^h$ is a continuous group action on $\Omega(G)$.

Note that $\pi^h$ is a permutation of $V$ as a composition of permutations. Recall that $H$ acts on $G$ by a graph isomorphism. Therefore
\begin{align*}
 (x,(\pi^h)(x))\in E &\iff (h^{-1}(x),h^{-1}((\pi^h)(x))\in E\\
&\iff (h^{-1}(x),\pi(h^{-1}(x))\in E.
\end{align*}
We recall that $\pi$ is restricted by $G$ and therefore $(h^{-1}(x),\pi(h^{-1}(x))\in E$ for any $x$. Thus, $\pi^h\in \Omega(A)$. In order to prove that it is indeed a group action, we need to show that $\pi^{hg}=\parenv{\pi^{g}}^h$ for all $\pi\in \Omega(G)$. Indeed, 
\begin{align*}
(\pi^{hg})(x)=(h\cdot g) (\pi (g^{-1}h^{-1}(x))=h(g(\pi(g^{-1}(h^{-1}(x))))=h(\pi^{g}(h^{-1}(x)))=\parenv{\pi^{g}}^h(x).
\end{align*}

It remains to show that the action is continuous. That is, for any $h\in H$, the action of $h$ on $\Omega(G)$ is a continuous map. Let $(\pi_n)_n$ be a pointwise convergent sequence in $\Omega(G)$, Let $\pi$ be the limit. Since $V$ is equipped with the discrete topology, for any $v\in V$ the sequence $\parenv{\pi_n(h^{-1}v)}_n$ is constant and equals $\pi(h^{-1}v)$ for sufficiently large $n$. Thus, the sequence $\parenv{h\pi_n(h^{-1}v)}_n$ is constant and equals $h\pi(h^{-1}v)=\parenv{\pi^{h}} (v)$ for sufficiently large $n$. That is, $(\pi^{h}_n)_n$ pointwise converge to $\pi^{h}$.

\begin{example}
\label{ex:Honeycomb}
Consider the undirected graph $L_H=(V_H,E_H)$ where
\[ V_H = \mathset{v + m\cdot (\sqrt{3},0)+n\cdot \parenv{\frac{\sqrt{3}}{2},\frac{3}{2}}:m,n\in \Z, v\in \mathset{\parenv{\frac{\sqrt{3}}{2},\frac{1}{2}},(\sqrt{3},1)}}\]   
and any vertex $v$ is connected to its three closest neighbours in $V_H$. That is, a vertex of the form $v=\parenv{\frac{\sqrt{3}}{2},\frac{1}{2}}+ m\cdot (\sqrt{3},0)+n\cdot \parenv{\frac{\sqrt{3}}{2},\frac{3}{2}}$ is connected in $E_H$ to $v+u$ where $u\in\mathset{\parenv{\pm \frac{\sqrt{3}}{2},\frac{1}{2}}, (0,-1) }$.

This graph is the well known two-dimensional honeycomb lattice (see Figure \ref{fig:Honey_basic}). We have $\Z^2$ acting on $L_H$ by translations of the fundamental domain. By this we mean that $n=(n_1,n_2)\in \Z^2$ acts on a vertex $v\in V_H$ by 
 \[ n(v)\eqdef v + n_1\cdot (\sqrt{3},0)+n_2 \cdot \parenv{\frac{\sqrt{3}}{2},\frac{3}{2}}, \] 
see Figure \ref{fig:Honey_fondemnetal}. We note that $L_H$ is a bipartite graph.

\begin{figure}
    \centering
    \subfigure[]
    {
        \includegraphics[scale=.13]{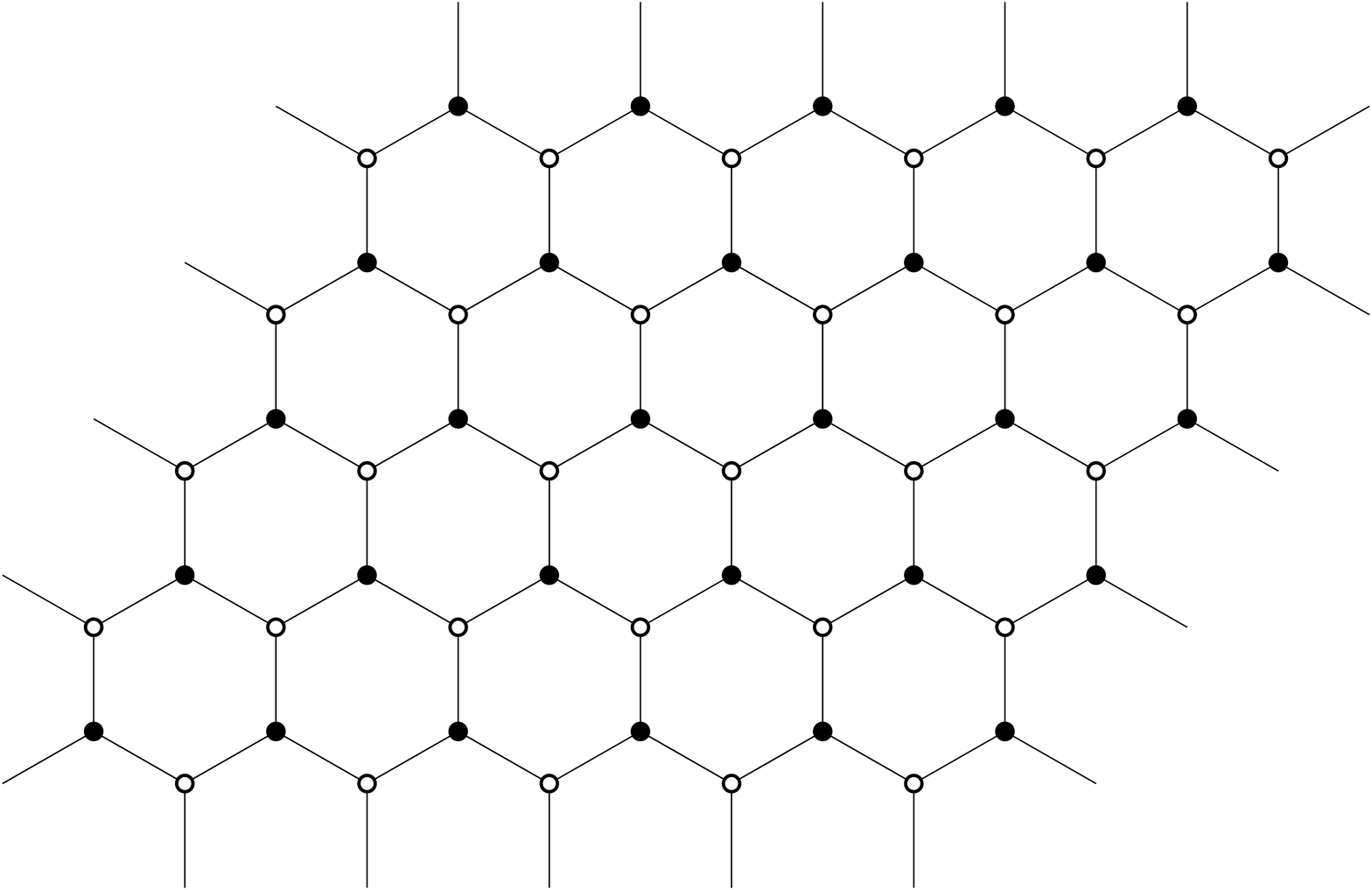}
        \label{fig:Honey_basic}
    }
    \subfigure[]
    {
        \includegraphics[scale=.37]{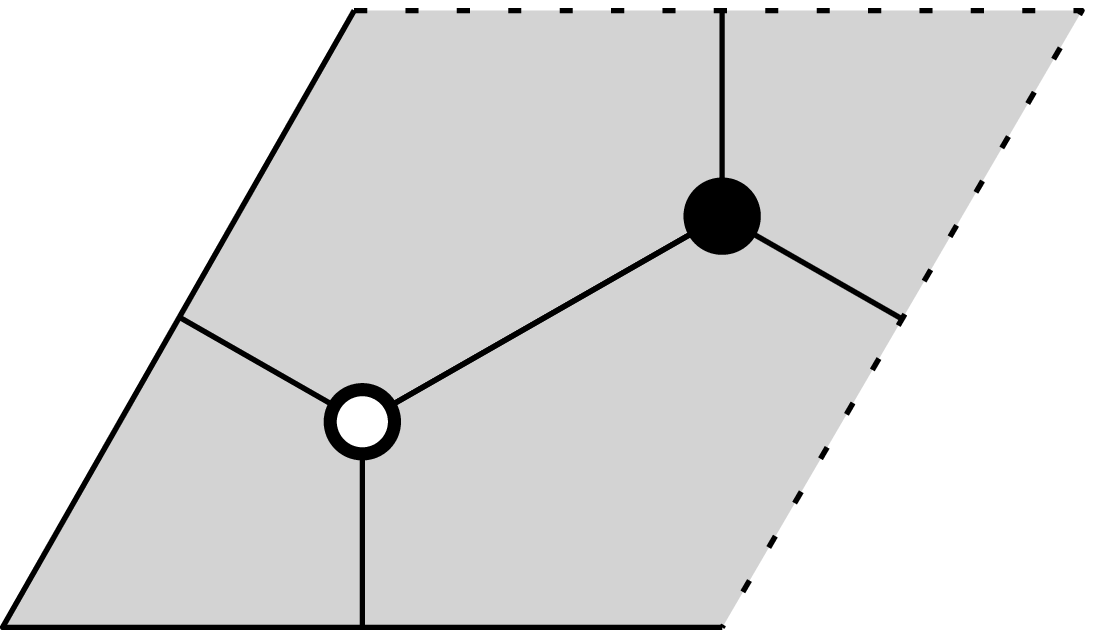}
        \label{fig:Honey_fondemnetal}
    }
    \caption
    {
        (a) The two-dimensional honeycomb lattice.
        (b) The fundamental domain.
    }
    \label{fig:Honey}
\end{figure}
\end{example}

\begin{example}
\label{ex:GenToZd1}
Let $\Gamma$ be a countable discrete group and $A\subseteq \Gamma$ be a finite non-empty set.  Consider the graph $G=(\Gamma,E)$, where 
\[ E \eqdef \mathset{(\gamma,\gamma a) : \gamma \in \Gamma, a\in A }.\] 
We have $\Gamma$ acting on $G$ by multiplication from the left. That is, $\alpha \in \Gamma$ acts on $\gamma\in \Gamma$ by $\alpha \cdot \gamma$. Note that for any $\gamma_1, \gamma_2\in A$ we have 
\begin{align*}
(\gamma_1,\gamma_2)\in E &\iff \exists a\in A \such \gamma_1=\gamma_2  a \\
&\iff  \exists a\in A \such \alpha \gamma_1=\alpha \gamma_2  a \\
& \iff (\alpha(\gamma_1),\alpha(\gamma_2))\in E.
\end{align*} 
This shows that $\Gamma$ acts by graph isomorphisms. Since $\Gamma$ is countable and $A$ is finite, $G$ is a locally finite graph. 
\end{example}

Consider the case when we choose the group from Example \ref{ex:GenToZd1} to be $\Gamma=\Z^d$ for some $d\in \N$, and we take $A\subseteq \Z^d$ to be some finite set. In that case we have the graph  $G_A\eqdef(\Z^d,E_A)$, where 
\[ E_A\eqdef \mathset{(n,m)\in \Z^d\times \Z^d : m-n\in A},\] 
and $\Z^d$ acting on $(\Z^d,E_A)$ by translations. That is, $n$ acts on $m$ by $\sigma_{n}(m)\eqdef m+n$.

If a permutation of $\Z^d$ is restricted by $G_A$, we say that it is restricted by the set $A$.
\begin{example}
\label{ex:GenToZd}
Let $d=2$ consider the sets $A_+\eqdef \mathset{(\pm 1,0),(0,\pm 1)}\subseteq \Z^2,$ and $ A_L\eqdef \mathset{(0,0),(1,0),(0, 1)}\subseteq \Z^2$.
For a permutation $\pi \in \Omega(G_{A_L})$, the orbit of an element is  $\parenv{\pi^{\circ n}(m)}_{n\in \Z}$, where $\pi^{\circ n}$ is the composition of $\pi$ - $n$ times for positive $n$ and the composition of $\pi^{-1}$ - $n$ times for negative $n$. We note that the orbit of any point is either a single point or a bi-infinite sequence.

We can represent each infinite orbit of $\pi$ by a polygonal path in $\Z^2$, moving either north or east at each step. We can characterize $\pi$ by the configuration of non-intersecting polygonal paths in $\Z^2$ defined by its bi-infinite orbits. On the other hand, any set of such polygonal paths defines an element in $\Omega(G_{A_L})$ (see Figure \ref{fig:Paths_L}). This case was first revisited by Schmidt and Strasser in \cite{SchStr17}.

In a similar fashion, we can represent a permutation in $\Omega(G_{A_+})$ by its orbits. In that case, orbits can be infinite, or finite with size grater then one. Each permutation in $\Omega(G_{A_+})$ is correspond to a covering of $\Z^2$ by substitutions (orbits of size 2) and  polygonal paths moving north, south, east or west at each step (see Figure \ref{fig:Paths_+}).
In Figure \ref{fig:Graph_PL} we exhibit the directed graph associated with $G_{A_+}$ and $G_{A_L}$.

\begin{figure}
    \centering
    \subfigure[]
    {
        \includegraphics[scale=.4]{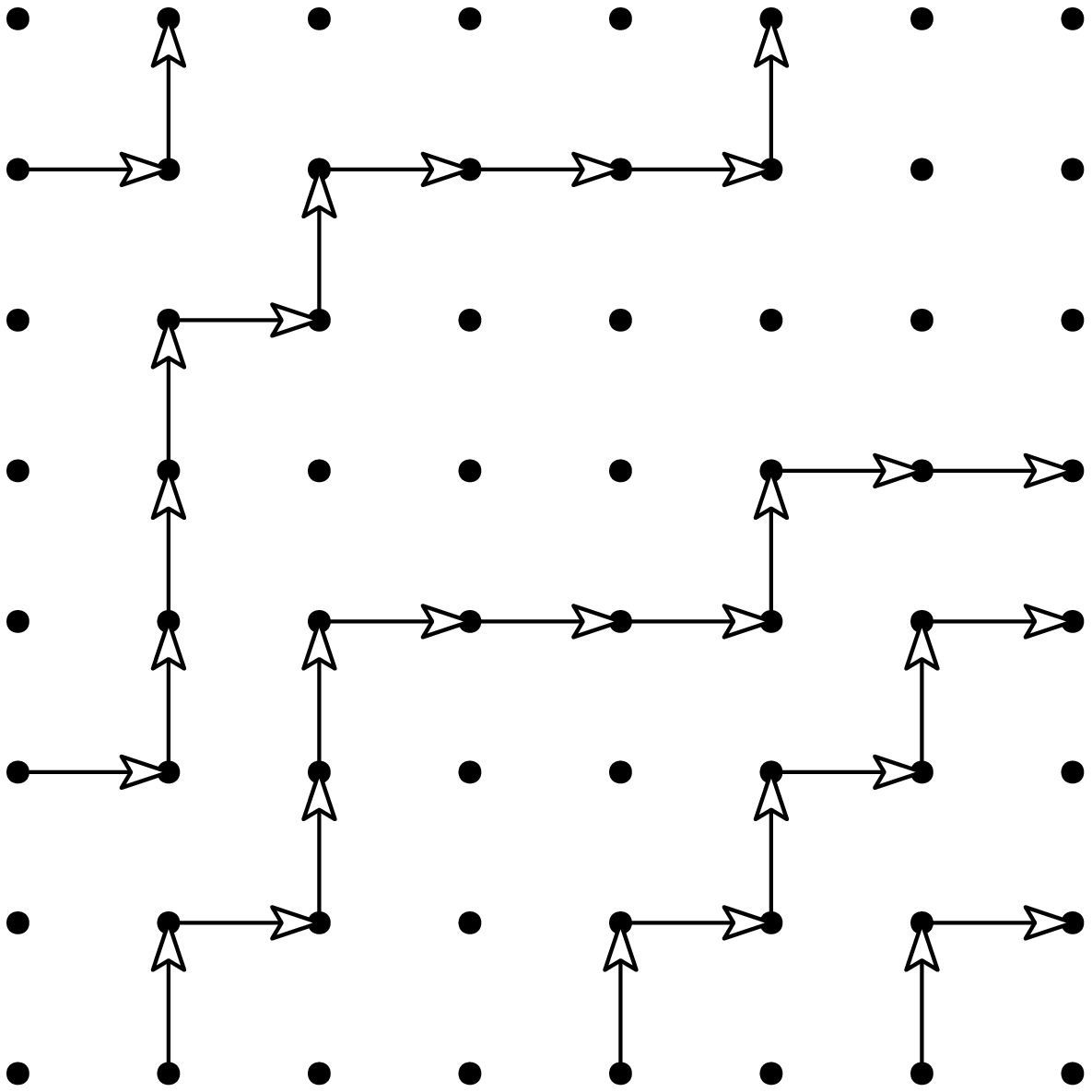}
        \label{fig:Paths_L}
    }  \hspace{7mm}
    \subfigure[]
    {
        \includegraphics[scale=.4]{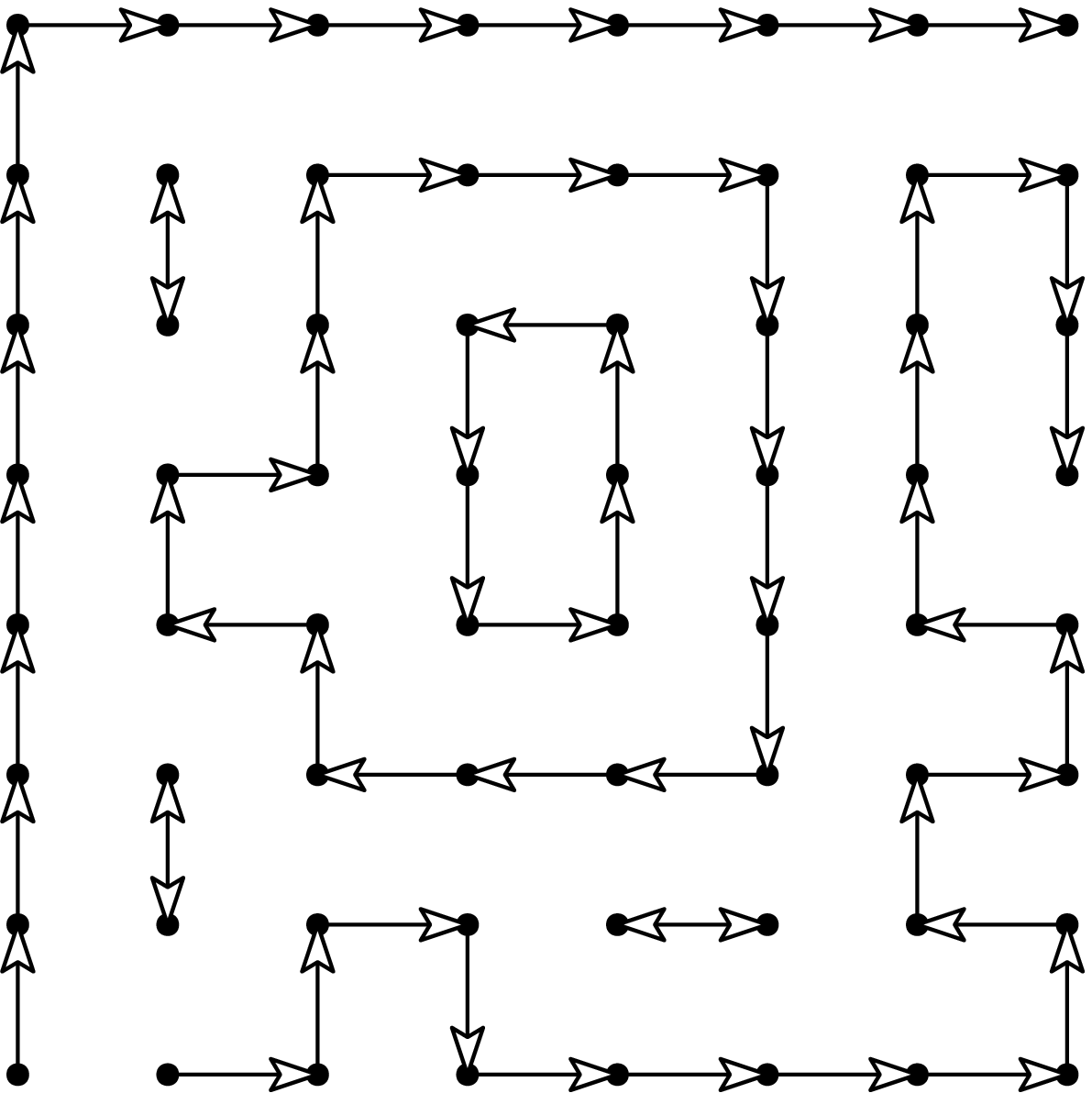}
        \label{fig:Paths_+}
    }
    \caption
    {
        (a) Paths configuration corresponding to an elements in $\Omega(G_{A_L})$.
        (b) Paths configuration corresponding to an elements in $\Omega(G_{A_+})$.
    }
    \label{fig:Paths}
\end{figure}

\begin{figure}
 \centering
  \includegraphics[width=150mm, scale=0.65]{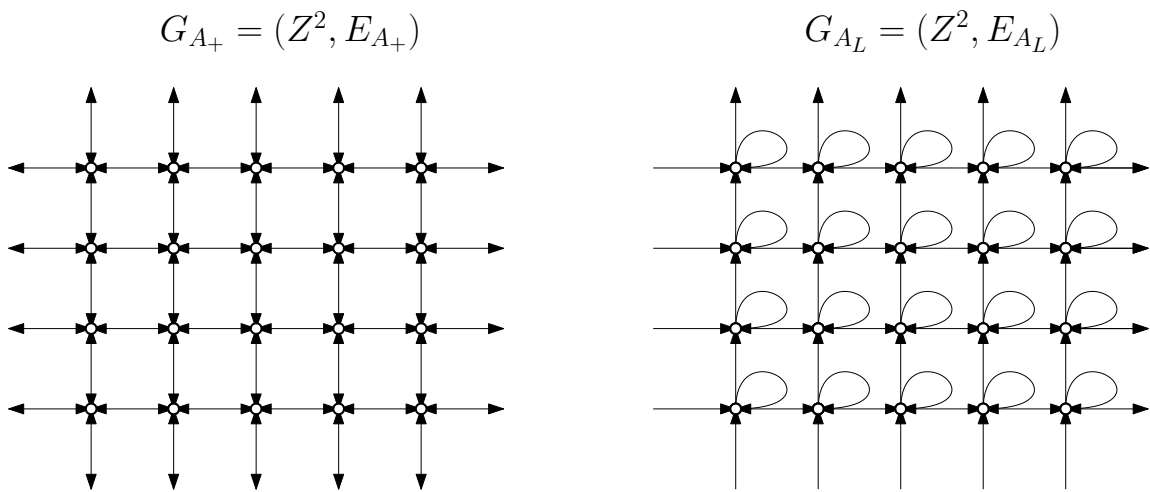}
  \caption{The graphs corresponding to $A_+$ and $A_L$.}
  \label{fig:Graph_PL}
\end{figure}
\end{example}

An important special case of dynamical systems is a subshift of finite type (SFT). Given a finite set, $\Sigma$, and an integer $d\in\mathbb{N}$, we consider the set $\Sigma^{\Z^d}$. An SFT, $\Omega \subseteq \Sigma^{\Z^d}$, is a subset of $\Sigma^{\Z^d}$, which is defined by a set of forbidden patterns. That is, there exists a finite set of forbidden patterns, $F\subseteq \bigcup_{B\in \fin(\Z^d)} \Sigma^B$, such that  
\[ \Omega=\mathset{\omega\in \Sigma^{\Z^d} : \forall n\in \Z^d \text{ and } B\in\fin(\Z^d), (\omega\circ \sigma_n)(B)\notin F },\]
where $\omega\left(B\right)$ is the restriction of $\omega$ to the
coordinates contained in the set $A$ and $\fin(\Z^d)$ denotes the set of all finite subsets of $\Z^d$. Throughout this  work, by abuse of notation, for $\omega \in \Sigma^{\Z^d}$ and $n\in \Z^d$, we will denote the composition $\omega \circ \sigma_n$ by $\sigma_n(\omega)$.

For an SFT, $\Omega\subseteq \Sigma^{\Z^d}$,  $\Z^d$ acts on $\Omega$ by translations. That is, $n\in \Z^d$ acts on an element $\omega\in \Omega$ by $n\omega=\omega \circ \sigma_{-n}$. For a multi-index $n\in \N^d$, the set of $[n_1]\times \cdots \times [n_d]$ patterns appearing in elements of $\Omega$ is denoted by $B_n(\Omega)$. Formally,
\[B_n(\Omega) \eqdef \mathset{P\in \Sigma^{[n]}  : \exists \omega\in \Omega \such \omega([n])=P}. \]

The topological entropy of an SFT is defined to be 
\[ \topent (\Omega) \eqdef \limsup_{n\to \infty} \frac{\log\abs{B_n(\Omega)}}{\abs{[n]}},\]
where we define that a sequence  $(n_k)_{k=1}^{\infty}\subseteq \N^d$ converge to $\infty$ if $n_k(i)\tends{k} \infty$ for all $1\leq i \leq d$. 

\begin{remark}
Entropy is defined and studied in a much more general settings of topological dynamical systems and sofic groups actions. Despite that, in this work we interested in SFTs, and therefore we will use the equivalent definition of entropy for SFTs, presented above. See \cite{KerHan11} for a detailed discussion and a general definition.
\end{remark}

\begin{fact} 
\label{prop:fekete} (\cite{HocMey10}, Section 2.2) The limit defining the topological entropy exists and \[ \topent(\Omega)=\inf_{n\in \N^d} \frac{\log \abs{B_n(\Omega)}}{\abs{[n]}}.\]
\end{fact}

Two SFT'S, $\Omega_1,\Omega_2\subseteq \Z^{d}$, are said to be conjugated if there exists a homeomorphism $\Phi:\Omega_1\subseteq \Z^{d_1}\to \Omega_2\subseteq \Z^{d_2}$ that commutes with the action of $\Z^d$. Such a map is called a conjugacy. 

\begin{fact} 
\label{fact:conj} (\cite{KerHan11}, Chpater 1) If $\Omega_1\subseteq \Sigma_1^{\Z^{d}}$ and $\Omega_2\subseteq \Sigma_2^{\Z^{d}}$ are conjugated, then $\topent(\Omega_1)=\topent(\Omega_2)$.
\end{fact}

The model of $\Z^d$ permutations restricted by some finite set, presented in Example \ref{ex:GenToZd}, which will be the main focus of this work, was introduced by Schmidt and Strasser in \cite{SchStr17}. A permutation $\pi\in S(\Z^d)$ which is restricted by some finite $A\subseteq \Z^d$ can be identified with an element $\omega_\pi \in A^{\Z^d}$, where $\omega_\pi(n)=\pi(n)-n\in A$. This identification induces an embedding of $\Omega(G_A)$ in $A^{\Z^d}$, which we denote by $\Omega(A)$. Formally,
\[ \Omega(A)\eqdef \mathset{\omega_\pi : \pi\in \Omega(G_A)}.\]

From now on, we will use this notation in order to describe $\Z^d$-restricted permutations.
With this embedding, the action of $\Z^d$ on $\Omega(G_A)$ translates to a shift operation in $\Omega(A)$. To see that, we compute
\begin{align*}
\omega_{\pi^m}(n)&=(\pi^m)(n)-n=\pi^{\sigma_m}(n)-n=(m+\pi(n-m))-n\\
&=\pi(n-m)-(n-m)=\omega_\pi(\sigma_{-m}(n))=\omega_\pi\circ \sigma_{-m}(n) =(m\omega_\pi)(n). 
\end{align*}

In their work \cite{SchStr17}, Schmidt and Strasser have shown that $\Omega(A)$ (with the shift operation) is an SFT for any finite $A\subseteq \Z^d$. They investigated the dynamical properties of such SFTs and their entropy, in general, and in some specific examples. We will focus on studying the entropy, mostly in the two-dimensional cases. 

\begin{definition}
Given a finite restricting set $A\subseteq \Z^d$ and $n=(n_1,n_2,\dots,n_d)\in \N^d$, a function $f:[n]\to \Z^d$ is said to be a permutation of the $n_1\times \cdots \times n_d$ discrete torus if $\tilde{f}:[n]\to [n]$ defined by
\[ \tilde{f}(m)=f(m) \bmod n,\]
is a permutation of $[n]$. If $f$ is restricted by $A$,  we say that $f$ is a restricted permutation (by $A$) of the torus.
\end{definition}

\begin{definition}
Let $\Omega\subseteq \Sigma^{\Z^d}$ be a d-dimensional SFT over some finite alphabet $\Sigma$. For  a subgroup $\Gamma\subseteq \Z^d$ of finite index, we denote the set of $\Gamma$ periodic points by
\[ \fix_\Gamma(\Omega)\eqdef \mathset{\omega\in \Omega : \omega \circ \sigma_n=\omega \text{ for all }n\in \Gamma}.  \]
\end{definition}
Given a finite restricting set $A\subseteq \Z^d$ and $n=(n_1,n_2,\dots,n_d)\in \N^d$, consider the group 
\[ \Gamma_n\eqdef n_1\Z\times n_2 \Z \times \cdots \times n_d\Z \subseteq \Z^d.\]

We observe that elements in $\fix_{\Gamma_n}(\Omega(A))$ correspond to restricted permutations of the $n_1\times\cdots\times n_d$ discrete torus, in the usual manner. 
We identify $\omega\in \fix_{\Gamma_n}(\Omega(A))$ with the function defined by the restriction of $\omega$ to $[n]$, denoted by $f_\omega$, which is, a restricted permutation of the torus.  That is, $\tilde{f}_\omega\eqdef f_\omega \bmod n$  is a permutations of $[n]$. 

\begin{definition}
The periodic entropy of an SFT $\Omega\subseteq \Sigma^{\Z^d}$ is defined to be 
\[ \topent_p(\Omega)\eqdef \limsup_{n\to \infty}\frac{\log \abs{\fix_{\Gamma_n}(\Omega)}}{\abs{[n]}}. \] 
\end{definition}

\begin{fact} \label{prop:GenPerEnt}
(\cite{LinMar85}, Proposition 4.1.15, Theorem 4.3.6)  For an SFT $\Omega\subseteq \Sigma^{Z^d}$, 
\[ \topent_p(\Omega)\leq \topent(\Omega).\]
Furthermore, if $d=1$, equality holds.
\end{fact}

\begin{remark}
 The inequality from Fact \ref{fact:conj} holds in the more general settings of shift spaces, in any dimension.  While equality holds in the one dimensional case, it can fail badly for general $\Z^d$ shift spaces when $d>1$, since there exists $\Z^d$ shift spaces with positive entropy and no periodic points (see \cite{HocMey10}, Section 9).
\end{remark}

Let $n\in \N^d$ and a permutation $f\in S([n])$, that is, a closed permutation of the $n_1\times\cdots \times n_d$ array. We observe that $f$ is also a toral permutation, as $\tilde{f}=f \bmod n$  is a permutations of $[n]$, since $f=\tilde{f}$. Thus, for some finite $A\subseteq \Z^d$, denoting
\[ B_n^f(A)\eqdef \mathset{f\in S([n]):\forall m\in [n], f(m)-m\in A },\] 
we have that $B_n^f(A)$ is a subset of toral permutations, restricted by $A$.
We conclude that $\abs{B_n^f(A)} \leq \abs{\fix_{\Gamma_n}(\Omega(A))}. $
Given a finite restricting set $A\subseteq \Z^d$ we define the closed entropy of $\Omega(A)$ to be 
\[ \topent_c(\Omega(A))\eqdef \limsup_{n\to\infty} \frac{\log\abs{B_n^f(A)}}{\abs{[n]}}.\]
Followed by this observation and Fact \ref{prop:GenPerEnt}, we have
 \[\topent_c(\Omega(A)) \leq \topent_p(\Omega(A))\leq \topent(\Omega(A)).\]

We now have three entropy-like quantities associated to permutations restricted by a fixed finite subset of $\Z^d$: periodic permutations (permutations of a torus), closed permutations and general permutations of $\Z^d$. In the next chapters, we will further study the relations between them.


\chapter{Restricted Permutations and Perfect Matchings}
\label{CHA:PM}
	A perfect matching of an undirected graph, $G=(V,E)$, is a subset of edges not containing self loops, $M\subseteq{E}$, in which every vertex is covered by exactly one edge. That is, for every vertex $v\in V$ there exists a unique edge $e_v \in M$ (which is not a self loop), for which $v\in e_v$. We denote the set of perfect matchings of a graph by $\PM(G)$. If $W:E\to \C$ is a weighting function on the edges, it naturally induces a score function on perfect matchings by
\[ W(M)\eqdef \prod_{e\in M }W(e).\] 
The weighted perfect matchings of $G$ with respect to $W$ is defined to be 
\[ \PM(G,W)\eqdef \sum_{M\in \PM(G)} W(M )=\sum_{M\in \PM(G)} \prod_{e\in M} W(e).\] 
We note that for the constant function $W \equiv 1$, $\PM(G,1)$ is just the number of perfect matchings of $G$.

 In \cite{Kas61, Kas63}, Kasteleyn presented an ingenious method for computing the weighted perfect matching of finite planar graphs. This method was used by Kasteleyn himself in order to compute the exponential growth rate of the number of perfect matchings of the two-dimensional square lattice. In 2006, Kenyon, Okounkov, and Sheffield \cite{KenOkoShe06} computed the exponential growth rate of perfect matchings (and much more than that) of $\Z^2$-periodic bipartite planar graphs. In their work, they were also using Kasteleyn's method.

In this chapter we show two different characterizations of restricted permutations by perfect matchings (Theorem \ref{th:GenCan},  Theorem \ref{th:PerToPM}). We use the results on perfect matchings $\Z^2$-periodic bipartite planar graphs in order to compute the topological entropy of restricted permutations in a couple of two-dimensional cases (see Sections \ref{A_L} and \ref{A_+}). We show a use of Kasteleyn's method and present a polynomial-time algorithm for computing the exact number of $n \times n$-possible patterns in one specific case. Finally, we show a natural generalization of this algorithm (see Section \ref{A_L}).

	\section{General Correspondence}
	\label{PM_char}
Let $G=(V,E)$ be a directed graph. Consider the undirected graph $G'=(V',E')$ defined by 
\[ V'\eqdef \mathset{I,O}\times V, \ \text{ and }\ E'\eqdef \mathset{\mathset{(O,v),(I,u) } ~:~ (v,u)\in E}. \]
Edges in $G'$ will be used to encode functions from $V$ to $V$ which are restricted by the original graph $G$. An edge of the form $\mathset{(v,O),(u,I)}$ will represent a mapping of $v$ to $u$. In Theorem \ref{th:GenCan} we will show that perfect matchings of $G'$ correspond to restricted permutations of $G$.

Assume that a group $H$ is acting on $G$ by graph isomorphisms, one can define an action of $H$ on $G'$ by 
\[ h(a,v)=(a,hv), \ a\in \mathset{I,O}, \ v\in V.\]
Unsurprisingly, this action is a group action on $G'$ and each element $h\in H$ acts on $G'$ by graph isomorphism as for $(O,v),(I,u)\in G'$, 
\begin{align*}
\mathset{(O,v),(I,u)}\in E' &\iff (v,u)\in E\\
& \iff (hv,hu)\in E\\
& \iff \underset{\mathset{h(O,v),h(I,u)}}{\underbrace{\mathset{(O,hv),(I,hu)}}}\in E'.
\end{align*}
This shows that $H$ acts on $G'$ by graph isomorphisms.

\begin{example}
Let $G=G_{A_L}$ be the graph described in Example \ref{ex:GenToZd}. We recall that $\Omega(G_{A_L})$, also denoted by $\Omega(A_L)$, is the set of $\Z^2$-permutations restricted by the set $A_L\eqdef \mathset {(0,0),(0, 1), (1,0)}$. In that case, the graph $G'$ consists of two copies $\Z^2$ with edges between vertices whose difference are in $A_L$ (see Figure \ref{fig:CanonL}).
\begin{figure}
 \centering
  \includegraphics[width=120mm, scale=0.5]{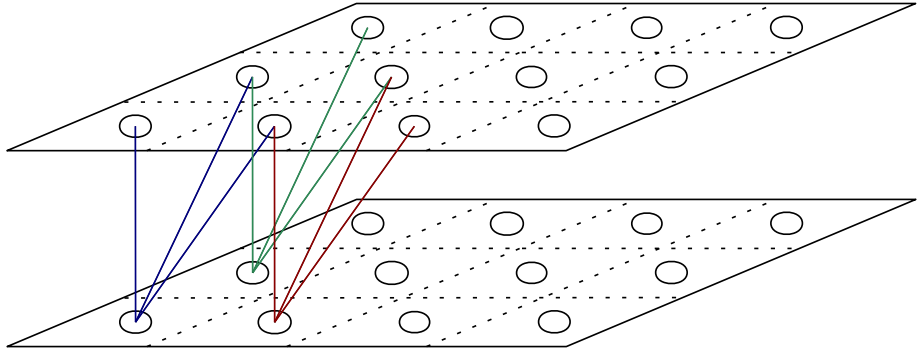}
  \caption{The graph $G_{A_L}'$}
  \label{fig:CanonL}
\end{figure} 
\end{example}

\begin{theorem}
\label{th:GenCan}
There is a bijection, $\Psi$, between elements of $\Omega(G)$ and  $\PM(G')$.
If  a group $H$ acts on $G$ by graph isomorphisms, then the action of $H$ on $G'$ induces a group action of $H$ on $\PM(G')$ such that following diagram commutes
\begin{align*}
\includegraphics[keepaspectratio=true,scale=1.1]{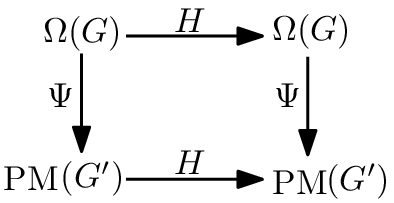}
\end{align*}
\end{theorem}

\begin{proof}
Consider the function $\Psi:\Omega(G)\to 2^{E'}$ defined by
\[ \Psi(\pi)\eqdef \mathset{\mathset{(O,v),(I,\pi(v))}~:~v\in V}.\]
Since $\pi$ is restricted by $G$, for any $v\in V$, $(v,\pi(v))\in E$. Thus, by the definition of $E'$, $\mathset{(O,v),(I,\pi(v))}\in E'$. This shows that $\Psi(\pi)\subseteq E'$.

We now show that  $\Psi(\pi)$ is a perfect matching of $G'$. Let $x$ be a vertex in $V'$. If $x$ is of the form $(O,v)$, $v\in V$, by the definition of $\Psi(\pi)$, $\mathset{(O,v),(I,\pi(v))}\in E'$ is the unique edge in $\Psi(\pi)$ containing $(O,v)$. If $x$ is of the form $(I,u)$, $u\in V$, we have that $\mathset{(O,\pi^{-1}(u)),(I,u)}\in E'$. Assume to the contrary the there exists another edge containing $(I,u)$ in $\Psi(\pi)$. That is, $\mathset{(O,v'),(I,u)}\in \Psi(\pi)$, $v'\neq \pi^{-1}(u)$. From the definition of $\Psi(\pi)$ it follows that $\pi(v')=u$, this is a contradiction as $\pi$ is injective. 

For a perfect matching $M\in \PM(G')$ and $v\in V$ let $M(v)\in V$ be the unique vertex such that $\mathset{(O,v),(I,M(v))}\in M$. Consider the function $\Phi:\PM(G')\to \Omega(G)$ defined by \[ \Phi(M)(v)=M(v),\] we show that $\Phi$ is well defined and it is the inverse function of $\Psi$. This will show that $\Psi$ is a bijection. From the Definition of $ \Phi(M)$, for any $v\in V$, $\mathset{(O,v),(I,M(v))}\in M\subset E'$, which implies that $(v,M(v))\in E$ and indeed the function $\Phi(\pi)$ is restricted by $G$. 

Now we show that $\Phi(M)$ is a permutation. Let $v_1,v_2$ be two distinct vertices in $V$. If $ \Phi(M)(v_1)= \Phi(M)(v_2)=u$, from the definition of $\Phi(M)$ we have $\mathset{(O,v_1),(I,u)}\in M$ and $\mathset{(O,v_2),(I,u)}\in M$. This is a contradiction since $M$ is a perfect matching of $G'$. This shows that $\Phi(M)$ is injective.

Let $u\in V$ and let $v$ be the unique vertex such that $M(v)=u$ (such exists since $M$ is a perfect matching and $G'$ is bipartite). Clearly, $\Psi(M)(v)=u$. This shows that $\Phi(M)$ is surjective. 

Let $v\in V$, from the definition of $\Psi$, $\Phi(\Psi(\pi))(v)$ is the unique vertex $u\in V$ such that $\mathset{(O,v),(I,u)}\in \Psi(\pi)$. From the definition of $\Psi(\pi)$, $u=\pi(v)$. This show that $\Phi \circ \Psi $ is the identity on $\Omega(G)$. Similarly, we have that $\Psi \circ \Phi $ is the identity on $\PM(G')$.

For the second part of the proof, let $H$ be a group, acting on $G$ by graph isomorphisms. 
For any $h\in H$, the isomorphic action of $h$ on $G'$ defines a map $h:E'\to E'$ by 
\[h(\mathset{(O,v),(I,u)})\eqdef \mathset{h(O,v),h(I,u)}\in E'.\]
It is easy to verify that this function maps perfect matchings of $G'$ to other perfect matchings of $G'$. It remains to show that the diagram commutes. Indeed, 
\begin{align*}
\Psi(\pi^{h})&=\mathset{\mathset{(O,v),(I,(\pi^{h})(v))}~:~v\in V}\\
&=\mathset{\mathset{(O,v),(I,h(\pi(h^{-1}v)))}~:~v\in V}\\
&=\mathset{\mathset{(O,h(h^{-1}v),(I,h(\pi(h^{-1}v)))}~:~v\in V}\\
&=\mathset{h\parenv{\mathset{(O,h^{-1}v),(I,\pi(h^{-1}v))}}~:~v\in V}\\
&=h\parenv{\mathset{\mathset{(O,u),(I,\pi(u))}~:~u\in h^{-1}V}}\\
\underset{h^{-1}V=V}{\Rsh} &=h \underset{\Psi(\pi)}{\underbrace{\parenv{\mathset{\mathset{(O,u),(I,\pi(u))}~:~u\in V}}}}=h(\Psi(\pi)).
\end{align*}
\end{proof}

	\subsection{Permutations of $\Z^2$ Restricted by $A_L$ }
		\label{A_L}
Permutations of $\Z^2$ restricted by the set $A_L=\mathset{(0,0),(0,1),(1,0)}$ were first studied by Schmidt and Strasser in \cite{SchStr17}. They proved that the topological entropy and the periodic entropy are equal in that case and speculated that it is around $\log(1.38)$. In this part of the work, we will show a connection between permutations of $\Z^2$ restricted by $A_L$ and perfect matchings of the honeycomb lattice. We will use this connection in order to derive an exact expression for the topological entropy (and periodic entropy) of $\Omega(A_L)$. In the second part we show a polynomial-time algorithm for computing the exact number of patterns in $B_{n,n}(A_L)$, and discuss a natural generalization of this algorithm. 

\subsubsection{The Honeycomb Lattice}
By Theorem \ref{th:GenCan}, we can  (bijectively) encode elements from $\Omega(A_L)$ by perfect matchings of the graph $G_{A_L}'$. If we draw the $G_{A_L}'$ on the plane, we may see that it is in fact the well known honeycomb lattice, $L_H$.
The honeycomb lattice is a $\Z^2$-periodic bipartite planar graph (see Figure \ref{fig:Honey_basic}, where different colors of vertices represent the two disjoint and independent sets). By this, we mean that it can be embedded in the plane so that translations of the fundamental domain in $\Z^2$ act by color-preserving isomorphisms of $L_{H}$ -- isomorphisms which map black vertices to black vertices and white to white.

For $n\in \N$, let $L_{H,n}$ be the quotient of $L_{H}$ by the action of $n\Z^2$, which is a finite bipartite on the $n\times n$ torus (see Figure \ref{fig:Honey_quotient}). A perfect matching of $L_{H,n}$ corresponds to a permutation of the $n\times n$ torus, restricted by $A$, in same manner as in Theorem \ref{th:GenCan}. Thus, 
\[ \abs{\fix_{n\Z^2}}=\abs{\PM(L_{H,n})}.\]

Kenyon, Okounkov and Sheffield \cite{KenOkoShe06} found an exact expression for the exponential growth rate of the number of toral perfect matchings of $\Z^2$-periodic bipartite planar graph. We use the following result which is a direct application of their work.
\begin{figure}
 \centering
  \includegraphics[width=120mm, scale=0.25]{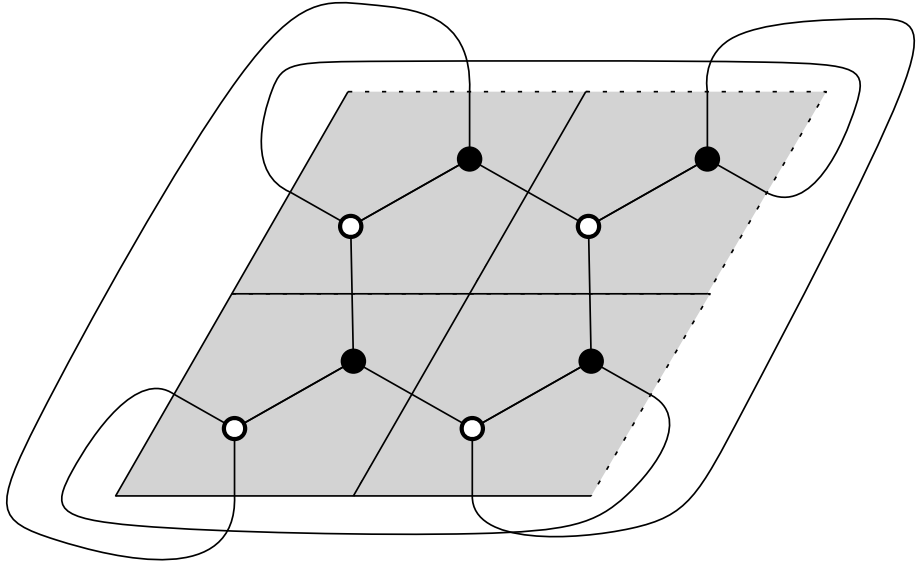}
  \caption{The quotient of $L_{H}$ by the action of $2\Z^2$ }
  \label{fig:Honey_quotient}
\end{figure}
\begin{proposition} \cite{KenOkoShe06, Ken00}
\label{prop:HoneyEnt}
\[ \lim_{n\to \infty} \frac{\log\abs{\PM(L_{H,n})}}{n^2}=\frac{1}{4\pi^2}\intop_0^{2\pi}\intop_0^{2\pi}\log\abs{1+e^{ix}+e^{iy}}dxdy.\] 
\end{proposition} 
The connection between the periodic entropy and the topological entropy of $\Omega(A_L)$ was investigated by Schmidt and Strasser in their first work on restricted movement. They  proved the following proposition:
\begin{proposition} \cite{SchStr17}
\label{prop:PerFreeL}
\[ \limsup_{n\to \infty}\frac{\abs{\fix_{n\Z^2}(\Omega(A))}}{n^2}=\lim_{n\to \infty}\frac{\abs{\fix_{4n^3\Z^2}(\Omega(A))}}{(4n^3)^2}=\topent(\Omega(A)).\]
\end{proposition}
\begin{remark} 
The proof of Proposition \ref{prop:PerFreeL} presented in \cite{SchStr17} by Schmidt and Strasser involves arguments regarding forming periodic points using reflections of polygonal patterns. Although using different machinery, the idea behind their proof is conceptually similar to the principle of reflection positivity, used by Meyerovitch and Chandigotia in \cite{MeyCha19} in order to explain that the topological entropy and the periodic entropy of the square lattice dimer model are equal. This suggests that the principle of reflection positivity may be used in order to prove that periodic entropy and topological entropy are equal in the more general case of perfect matchings of bipartite planar $\Z^2$-periodic graphs. 
\end{remark} 
We combine the results presented above with the observation about the correspondence between perfect matchings of the honeycomb lattice to $A_L$-restricted permutations to obtain:
\begin{theorem}
\label{th:SolL}
\[ \lim_{n\to\infty}\frac{\log\abs{\fix_{n\Z^2}(\Omega(A))}}{n^2}  =\topent_p(\Omega(A_L))=\topent(\Omega(A_L))=\frac{1}{4\pi^2}\intop_0^{2\pi}\intop_0^{2\pi}\log\abs{1+e^{ix}+e^{iy}}dxdy.\] 
\end{theorem}
\begin{proof}
From Proposition \ref{prop:GenPerEnt}, we know that $\topent_p(\Omega(A))\leq \topent(\Omega(A))$. 
On the other hand, by Proposition \ref{prop:PerFreeL},
\begin{align*}
 \topent_p(\Omega(A))&=\limsup_{n_1,n_2\to\infty}\frac{\log\abs{\fix_{\Gamma_{(n_1,n_2)}}(\Omega(A))}}{n_1 n_2}\\
 &\geq \lim_{n\to\infty} \frac{\log\abs{\fix_{4n^3\Z^2}(\Omega(A))}}{(4n^3)^2}=\topent(\Omega(A)).
\end{align*}
This shows that 
\[\topent(\Omega(A))=\topent_p(\Omega(A))=\lim_{n\to \infty}\frac{\abs{\fix_{4n^3\Z^2}(\Omega(A))}}{(4n^3)^2}. \] 
Using Proposition \ref{prop:HoneyEnt} and the equivalence between perfect matchings of $L_{H,n}$ and periodic restricted permutations, we conclude
\begin{align*}
\frac{1}{4\pi^2}\intop_0^{2\pi}\intop_0^{2\pi}\log\abs{1+e^{ix}+e^{iy}}dxdy&=\frac{\log\abs{\PM(L_{H,n})}}{n^2}\\
&=\frac{\log\abs{\fix_{n\Z^2}(\Omega(A))}}{n^2}\\
&=\lim_{n\to \infty}\frac{\abs{\fix_{4n^3\Z^2}(\Omega(A))}}{(4n^3)^2}=\topent_p(\Omega(A)).
\end{align*}
This completes the proof of the theorem.
\end{proof}
\begin{remark}
Theorem \ref{th:SolL} provides a complete solution to the question raised in the work by Schmidt and Strasser \cite{SchStr17}, whether it is true that $ \lim_{n\to\infty} \frac{\log\abs{\fix_{n\Z^2}(\Omega(A))}}{n^2}$ exists (and equal to $\topent(\Omega(A))$). 
\end{remark}

\subsubsection{Counting Patterns in Polynomial-Time}
We saw that there exists a natural correspondence between permutations restricted by $A_L$ and perfect matchings of the honeycomb lattice. unfortunately, this correspondence does not translate to a matching between elements in $B_{n,n}(A)$  and perfect matching of finite subsets of the honeycomb lattice. Thus, we do not have a canonical way to use the powerful tools known for counting perfect matchings, as we desire. In this part of the work we will see a construction that allows us to find a correspondence between  elements from $B_{n,n}(A)$ and perfect matchings of a some slightly different graph. Finally, we will obtain a generalization of this idea to a more general case.

\begin{definition}
An orientation of an undirected graph with no parallel edges $G=(V,E)$ is an assignment of a direction to each of the edges of the graph. Formally, an orientation of $G=(V,E)$ is a directed graph $G'=(V',E')$ such that $V'=V$, $E=\mathset{\mathset{u,v}:(u,v)\in E'}$ and $(u,v)\in E'$ implies $(v,u)\notin E'$. 
\end{definition}

\begin{definition}
Let $G=(V,E)$ be an undirected graph and consider two perfect matchings in the graph $M,M'\in \PM(G)$. Denote by $M\oplus M'$ the symmetric difference operation between sets. A Pfaffian orientation of
$G$ is an orientation such that for any two perfect matchings in the graph $M,M'\in \PM(G)$, any
cycle in $M\oplus M'$ and any traversal direction of the cycle there is an odd number of edges oriented in agreement with it. If there exists a Pfafian orientation for $G$, it is said to be Pfafian orientable.  
\end{definition}
\begin{proposition} \label{prop:PaffOr} \cite{Kas67} 
Let $G=(V,E)$ be a planar. An orientation of $G$ for which every clockwise walk on a face of the graph has an odd number of edges agreeing is a Pfaffian orientation. Furthermore, there is a polynomial-time algorithm which finds such an orientation (in particular, such an orientation always exists).
\end{proposition}
\begin{definition}
Let $n\in \mathbb{N}$ and $A$ be a skew-symmetric $2n\times 2n$ matrix with complex values (namely, satisfying $A=-A^*$). The Pfaffian of A is defined by the equation
\begin{align*}
\Pf (A) \eqdef\frac{1}{2^n\cdot n!} \sum_{\sigma \in S(2n)}  \sgn(\sigma)\prod_{i=1}^n a_{\sigma(2i-1),\sigma(2i)}.
\end{align*} 
\end{definition}
If a matrix $A$ is skew-symmetric, the Pfaffian can be calculated by the formula
\begin{align*}
\parenv{\Pf(A)}^2=\det (A)
\end{align*}
 (\cite{Tes00}, Appendix A.1), and since determinants are efficiently computable, Pfaffians are efficiently computable as well (up to their sign).
\begin{proposition}
\label{prop:Kas}
\cite{Kas61} Let $G=(V,E)$ be a finite  undirected graph, $W:E\to \mathbb{C}$ be a weighting function and $G'=(V,E')$ a Pfaffian orientation of $G$. Then the weighted number of perfect matchings of $G$ can be calculated by 
\begin{align*}
\PM(G,W)=\pm \Pf(A_G),
\end{align*}
where $A_G$ is the $\abs{V}\times \abs{V}$ adjacency matrix defined by 
\begin{align*}
A_G (i,j)=\begin{cases}
W\parenv{\mathset{i,j}} & \text{If } (i,j)\in E'\\
-W\parenv{\mathset{i,j}} & \text{If } (j,i)\in E'\\
0 & \text{Otherwise}.
\end{cases}
\end{align*}  
\end{proposition}

The idea for counting the number of $n\times n$ patterns in $B_{n,n}(\Omega(A_L))$ is to construct a graph $G_n$, and a weighting function $W$, such that $\PM(G_n,W)=\abs{B_{n,n}(\Omega(A_L))}$. We take the $2n\times 2n$ quotient graph, $L_{H,2n}$, described in the previous part of this section, and we change it in the boundary. To each vertex in the boundary, we connect a gadget (or two in some cases). Finally, we will set the weights such that $\PM(G_n,W)=\abs{B_{n,n}(\Omega(A_L))}$.
For $n\in\mathbb{N}$, consider the graph $G_n=(V_n,E_n)$, where $V_n$ is the vertex set of size $4n^2+4n-1$ defined by the union of
\begin{align*}
P_n\eqdef \bigcup_{k=1}^{n^2} \mathset{I_k,O_k}
\end{align*}
and
\begin{align*}
U_n\eqdef \bigcup_{k=1}^{4n-1} \mathset{T_{k,1},T_{k,2},T_{k,3},T_{k,4}}.
\end{align*}
The set of edges is the union of 3 types of edges; edges between vertices of $P_n$, edges between vertices of $U_n$, and crossing edges between $P_n$ and $U_n$:
\begin{align*}
&E_{P_1}\eqdef\mathset{ \mathset{I_k,O_k} : 1 \leq k \leq n^2}\\
&E_{P_2}\eqdef \mathset{\mathset{O_k,I_{k+1}}:1\leq k \leq k+1 \leq n^2}\\
&E_{P_3}\eqdef\mathset{\mathset{O_k,I_{k-n}}:1\leq k-n \leq k \leq n^2}\\
&E_P\eqdef\bigcup_{i=1,2,3}E_{P_i}.
\end{align*}
\begin{align*}
&E_{U_1}\eqdef\mathset{ \mathset{T_{k,i},T_{k,j}} : 1 \leq k \leq 4n-1, 1\leq i<j \leq 4}\\
&E_{U_2}\eqdef \mathset{\mathset{T_{k,3},T_{k+1,2}}:1\leq k < k+1 \leq 4n-1}
E_S=\bigcup_{i=1,2}E_{U_i}.
\end{align*}
\begin{align*}
&E_{P,U,1}\eqdef\bigcup_{1\leq k \leq n}\mathset{ \mathset{O_k,T_{k,1}},\mathset{O_{n\cdot k},T_{n+1+k,1}}, \mathset{I_{n^2+1-k},T_{2n+k,1}}}\\
&E_{P,U,2}\eqdef \bigcup_{1\leq k\leq n-1}\mathset{\mathset{I_{n_2-n-k\cdot n},T_{3n+k}}}
E_{P,S}=\bigcup_{i=1,2}E_{P,U,i}.
\end{align*}
\begin{align*}
E_n\eqdef E_P\cup E_S \cup E_{P,S}.
\end{align*}
The graph described above consists of two types of gadgets, connected to each other. The first gadget is the complete graph $K_2$ with two vertices  $\mathset{I_k,O_k}$, we will call it the $IO$ gadget (see figure  \ref{fig:L-gad}). We have $n^2$ such $IO$ gadgets, representing the vertices of the $n\times n$ square lattice. In this part, we will not enumerate those vertices by coordinates of $\mathbb{Z}^2$ as it will easier for later analysis to enumerate them with consecutive natural numbers. Thus, we enumerate the $IO$ gadgets by $\mathset{IO_1,IO_2,\dots, IO_{n^2}}$ and arrange them in the square lattice as follow  
\begin{align*}
\begin{matrix}
IO_1 &\ IO_2&\  \cdots &\ IO_n\\
IO_{n+1} &\ IO_{n+2} &\ \cdots &\ IO_{2n}\\
\vdots &\ \vdots &\ \ddots &\ \vdots \\
IO_{n^2-n+1} &\ IO_{n^2-n+2} &\ \cdots &\ IO_{n^2}.
\end{matrix}
\end{align*}
The $O$-vertex of any $IO$ gadget is connected to the to its right and upper $IO$ neighbour gadgets via its $I$-vertex . The second gadget is the complete graph with 4 vertices - $K_4$. We will call this gadget $T$ (see figure  \ref{fig:L-gad}). We arrange the $4n-1$ $T$-gadgets that we have in a chain, wrapped around the $n\times n$ square lattice, ordered as follows:
 \begin{align*}
\begin{matrix}
 \ &\ T_1 &\ T_2 &\ \cdots &\ T_n &\ \ \\
T_{4n-1} &\ IO_1 &\ IO_2&\  \cdots &\ IO_n &\ T_{n+1}\\
T_{4n-2} &\ IO_{n+1} &\ IO_{n+2} &\ \cdots &\ IO_{2n} &\ T_{n+2}\\
\vdots &\  \vdots &\ \vdots &\ \ddots &\  \vdots &\ \vdots \\
T_{3n+1} &\ IO_{n^2-2n+1} &\ IO_{n^2-2n+2} &\ \cdots &\ IO_{n^2-n} &\ T_{2n-1}\\
\ &\ IO_{n^2-n+1} &\ IO_{n^2-n+2} &\ \cdots &\ IO_{n^2} &\  T_{2n}\\
\ &\ T_{3n} &\ T_{3n-1} &\ \cdots &\ T_{2n+1} &\  \\
\end{matrix}
\end{align*}
 The $T$-gadgets are connected between them by an edge $\mathset{T_{k,3},T_{k+1,2}}$ forming the chain. Finally, we connect each of the $IO$ gadgets in the boundary of the lattice to the its neighbouring $T$ gadget in the lattice. The connection is by an edge, connecting the vertex from the $IO$ gadget which is on the boundary of the square (it might be $I$ or $O$) with  $T_{k,1}$ ($T_k$ is the lattice neighbouring in the chain). Any $IO$ gadget in the boundary is connected to exactly one $T$ gadget, except for three cases: \begin{itemize}
\item The upper left corner gadget, $IO_1$ is is connected to $T_1$ by $O_1$ and to $T_{4n-1}$ by $I_1$. 
\item The upper right corner gadget, $IO_n$ is is connected to $T_n$ and  to $T_{n+1}$ by $O_n$.
\item The lower right corner gadget, $IO_{n^2}$ is is connected to $T_{2n}$ by $O_{n_2}$ and to $T_{2n+1}$ by $I_{n^2}$.
\end{itemize}
See Figure \ref{fig:L-PM} for an depiction of $G_4$.
\begin{figure}
 \centering
  \includegraphics[width=70mm, scale=0.5]{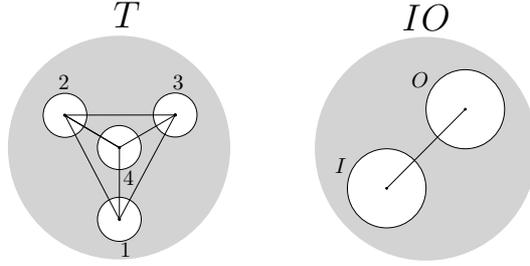}
  \caption{The $T$ and $IO$ gadgets. }
  \label{fig:L-gad}
\end{figure}

\begin{figure}
 \centering
  \includegraphics[width=80mm, scale=0.5]{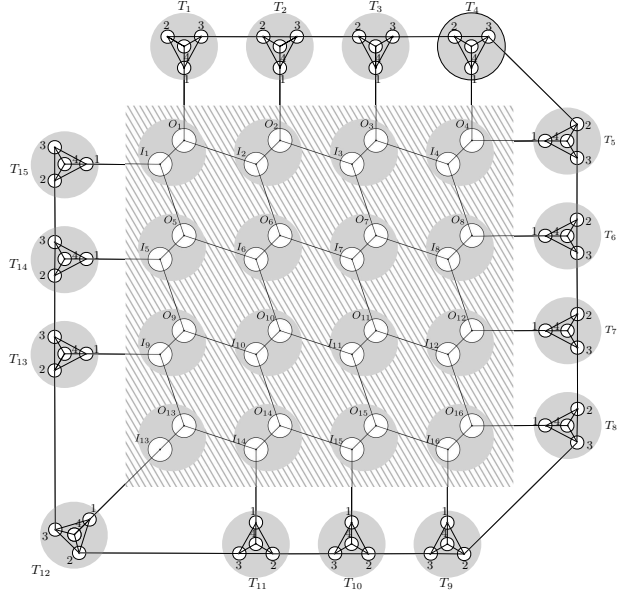}
  \caption{The graph $G_4$. }
  \label{fig:L-PM}
\end{figure}
\begin{definition}
Let $G=(V,E)$ be an undirected graph and $V'\subseteq V$ be a subset of vertices. A set of edges $C\subseteq E$ is said to be a perfect cover of $V'$ if the following are satisfied
\begin{itemize}
\item Any vertex $v\in V'$ has an edge $e_v\in C$ such that $v\in e_v$.
\item No two different edges in $C$ share a vertex.
\item For any edge, $e \in C$, the intersection $e\cap V'$ is non empty.  
\end{itemize} 
Denote the set of all perfect covers of $V'$ in $G$ by $\PC(V',G)$.
\end{definition}
Given an undirected graph $G=(V,E)$ we will sometimes identify subsets of edges with assignments of zeros and ones to the edges in $E$. That is, given $C\subseteq E$, we will identify it with its indicator function $\mathbb{I}_C\in \mathset{0,1}^{E}$. By an abuse of notation, we denote $C(e)\eqdef \mathbb{I}_C(e)$ for $e\in E$. With this identification, if $V'\subseteq V$ is a subset of vertices , $C\in \PC(V',G)$  if and only if  the following are satisfied :
\begin{itemize}
\item For each $v\in V'$ ,
\[ \abs{\mathset{e\in E_n:v\in e \text{ and } C(e)=1}}=1.\]

\item For any edge, $e\in E$,  $C(e)=1$ implies $e\cap V'\neq \emptyset$.
\end{itemize} 

\begin{lemma}
\label{lem:K4}
Let $n,k\in\mathbb{N}$ such that $1\leq k \leq 4n - 1$. For any $C\in \PC(T_k,G_n)$ and $C' \in \PC(V_n\setminus T_k,G_n)$:
\begin{enumerate}
\item The number of edges in $C$ connecting it to other gadgets is even. 
\item If $C'$ has no edges meeting $T_k$ then there are exactly three elements in $\PM(G_n)$ containing $C'$. If $C'$ has exactly one edge that intersects $T_k$, then there is exactly one element in $\PM(G_n)$ containing $C'$.
\end{enumerate} 
\end{lemma}
\begin{proof}
For convenience, in this proof we will identify subsets edges with assignments of zeros and ones to the edges as described above.
\begin{enumerate}
\item Assume to the contrary that  the number of edges in $C$ connecting it to other gadgets, denoted by $l$, is odd. Recalling that $T_k$ contains $4$ vertices, we have that the number of vertices in $T_k$ connected with each other by edges from $C$ is exactly $4-l$, which is also an odd number. Since $C$ is a perfect cover, any vertex is connected by edges in $C$ to exactly one other vertex, meaning that the number of vertices in $T_k$ connected with each other by edges from $C$ is even (they come in pairs). This is a contradiction.  
\item Let $C' \in \PC(V_n\setminus T_k,G_n)$, denote the edges connecting between $T_k$ and $V\setminus T_k$  by $e_{o,1},e_{o,2},$ and $e_{o,3}$.
\begin{itemize}
\item If $\abs{\mathset{e_{o,1},e_{o,2},e_{o,3}}\cap C'}=2$, without loss of generality, $C'(e_{o,1})=C'(e_{o,2})=1$ and $C'(e_{o,3})=0$. We will now extend $C'$ to a perfect matching of $G_n$, and we will see that there is only one way to do it. As in the previous part of the proof, all the other edges connected to $T_{k,1}$ and $T_{k,2}$ must be assigned with $0$ in $C'$ (otherwise $T_{k,1}$ or $T_{k,2}$ would be covered with more than one edge), meaning that $T_{k,4}$ must be connected in with $T_{k,3}$ in $C'$, that is $C'(e_{3,4})=1$. Now each vertex in $T_k$ is covered by exactly one edge from $C'$ and the remaining edges must be assigned with $0$. It is easy to see now that the extended $C'$ is indeed a perfect matching of $G_n$ as any vertex is covered by exactly one edge from $C'$.
\item   If $\abs{\mathset{e_{o,1},e_{o,2},e_{o,3}}\cap C'}=0$ we have three options to set the edge covering $T_{k,4}$ each defines another extension of $C'$. If we set $C'(e_{1,4})=1$, since $T_{k,2}$ and $T_{k,3}$ are not covered by edges coming outside of $T_k$, we must have $C'(e_{2,3})=1$ which force us to assign $0$ to the remaining edges. The other cases where we set $C'(e_{2,4})=1$ and $C'(e_{3,4})=1$ are proven the same.
\end{itemize}
See a visualization of all of the described cases in Figure \ref{fig:K4lemma2}.
\begin{figure}
 \centering
  \includegraphics[width=160mm, scale=0.5]{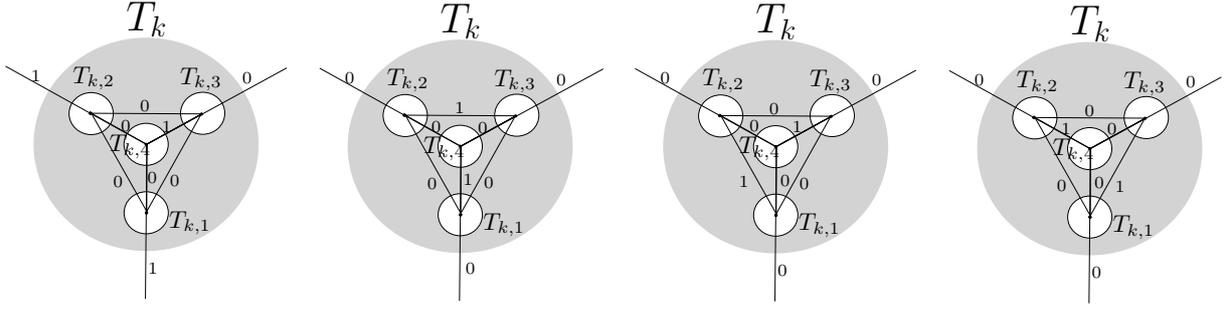}
  \caption{The possible extensions of a perfect covering to $T_k$. }
  \label{fig:K4lemma2}
\end{figure}

\end{enumerate}
\end{proof}

\begin{lemma}
\label{lem:PerToPC}
Perfect covers of $P_n$ in $G_n$ correspond bijectively with elements in $B_{n,n}(\Omega(A_L))$. That is,
\begin{align*}
\abs{\PC(P_n,G_n)}=\abs{B_{n,n}(\Omega(A_L))}.
\end{align*}
\end{lemma}
\begin{proof}
The proof of this lemma follows the same idea as the proof of Theorem \ref{th:GenCan}. Given a perfect cover, $C\in \PC(P_n,G_n) $ we construct a function $\pi_C$ from the $n \times n$ square lattice (call it $Q_n$), restricted by $A_L$, which can be extended to a restricted permutation of $\mathbb{Z}^2$.  Recall that in the construction $G_n$, we had $n^2$  $IO$ gadgets positioned on the lattice points in $[n]\times [n]$, enumerated by $1,2,\dots,n^2$. For an index $j\in [n]\times [n]$, there exists a unique index,  $1\leq J \leq n^2$, such that the gadget $IO_J$ sits in the coordinate $j\in [n]\times [n]$. We identify the index $j\in [n]\times [n]$ with its matching index $J\in [1,n^2]$. Let $j$ be a vertex in the $n\times n$ square lattice, The function $\pi_C(j)$ will be determined by the unique edge in $C$ which covers the vertex $O_{J}$, denoted by $e_J$. 
\begin{itemize}
\item If $e_J$ connects $O_J$ to  the $I$ vertex of its right neighbour in $G_n$, $J+1$, (or to the $T$ gadget on its right in case that $j$ is on the right edge of the square) we set $\pi_C(j)=j+ (1,0)$.
\item If $e_J$ connects $O_J$ to  the $I$ vertex of its neighbour from above in $G_n$, $J-n$, (or to the $T$ gadget above it in case that $j$ is on the upper edge of the square) we set $\pi_C(j)=j+ (0,1)$.
\item Otherwise, $O_J$ is connected to $I_J$ and we set $\pi_C(j)=j$.
\end{itemize}
See Figure \ref{fig:PerToPC} for a visualization.

By Proposition \ref{prop:compli2}, in order to show that $\pi_C$ can be extended to a restricted permutation $\mathbb{Z}^2$, it is sufficient to show that the function defined is injective and that its image covers the $(n-1)\times (n-1)$ square obtained when we remove the lower and left edges of the $n\times n$ square, that is, $Q_{n}^{int}\eqdef [1,n-1]\times[1,n-1]$.

Indeed, assume to the contrary that there exist two distinct lattice points, $j_1$ and $j_2$, such that $\pi_C(j_1)=\pi_C(j_2)=j_3$. From the construction of $\pi_C$ it follows that $O_{J_1}$ and $O_{J_2}$ are connected to $I_{J_3}$, by edges from $C$. That is a contradiction, as $C$ is a perfect cover. For an index $j$ in $Q_{n}^{int}$, note that $I_J$ is covered by exactly one edge in $C$, as $C$ is a perfect cover of $G_n$, denote it by $\mathset{I_j,V}$. Recall that $j\in Q_{n}^{int}$ so by the construction of $G_n$, it can only be connected to $O$-type vertices of the $n\times n$ square. Therefore $V=O_I$ for some $I$ and $\pi_C(i)=j$.

In order to complete the proof it remains to show that the map $C\to\pi_C$ is a bijection. Let $\pi\in B_{n,n}(\Omega(A_l))$ be a function defined on $Q_{n}$. For any index $1\leq J \leq n^2$, let $j\in [n]\times [n]$ be its matching index (as described at the beginning of the proof). Denote $j_\pi\eqdef \pi(j)$ and let $J_\pi$ be the index identified with $j_\pi$.  
By Proposition \ref{prop:compli2}, $\pi$ is injective and its image covers $Q_{n}^{int}$. For an index on the upper and right edges of $Q_{n}$,  $i$,  such that $\pi(i)\notin Q_n$ let $k(i)$ the index of the $T$-gadget placed in $\pi(i)$. For an index $j\in Q_n\setminus \ima(\pi)$, not covered by the image of $\pi$, let $k(j)$ be the index of the unique $T$-gadget connected to $I_j$ (such exists as $j$ can only be in $Q_n\setminus Q_{n}^{int}$). 
First define $C_{\pi,P}$ and $C_{\pi,P,U}$ by
\begin{align*}
C_{\pi,P}\eqdef\bigcup_{j\in Q_n \text{ s.t. } \pi(j)\in Q_n}\mathset{O_{J},I_{J\pi}}
\end{align*} 
and
\begin{align*}
C_{\pi,P,U}\eqdef\parenv{\bigcup_{j\in Q_n \text{ s.t. } \pi(j)\notin Q_n}\mathset{O_{J},T_{k(j),1}}}\bigcup \parenv{\bigcup_{i\in Q_n \setminus\ima(\pi(Q_{n}))}\mathset{I_{I},T_{k(i),1}}}.
\end{align*} 
Finally, define 
\begin{align*}
C_\pi\eqdef C_{\pi,P}\cup C_{\pi,P,U}.
\end{align*}
It is easy to verify that $C$ is a perfect cover of $P_n$ as $\pi$ is injective and covers $Q_{n}^{int}$ with its image, and that 
\begin{align*}
C_{\pi_C}=C,
\end{align*}
meaning that the map $C\to\pi_C$ is a bijection. 
\end{proof}
\begin{figure}
 \centering
  \includegraphics[width=160mm, scale=0.5]{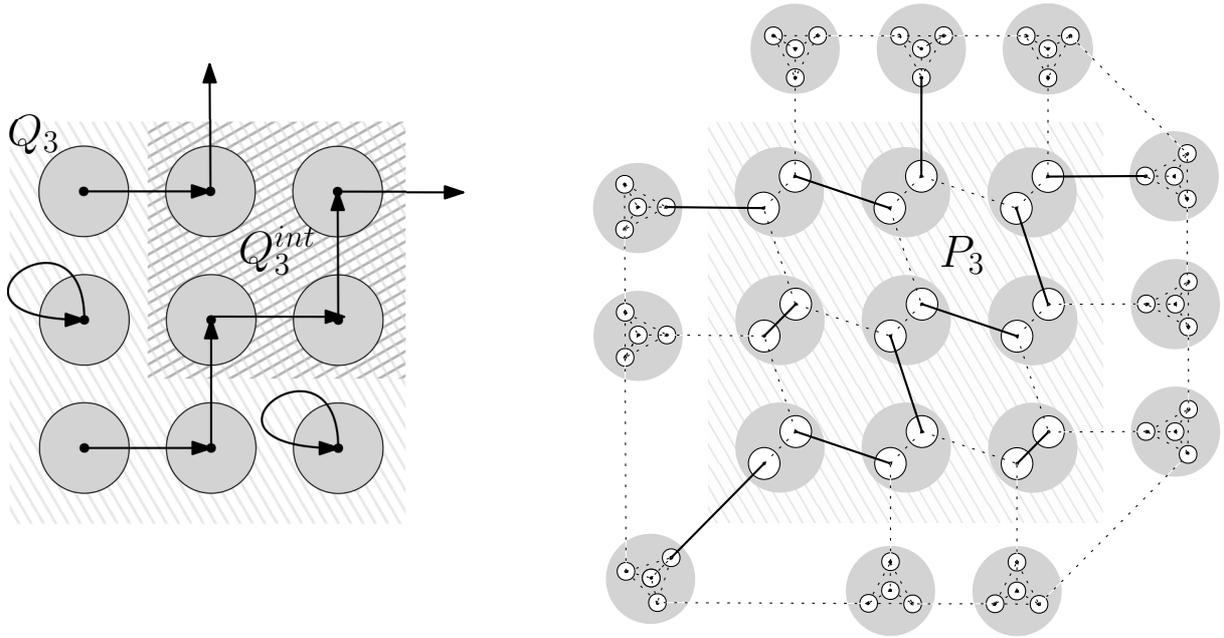}
  \caption{The correspondence between an element in $B_{3,3}(\Omega(A_L))$ and a perfect cover of $P_3$ in $G_3$. }
  \label{fig:PerToPC}
\end{figure}

\begin{lemma}
\label{lem:CovToPM}
Let $C\in \PC(P_n,G)$, then
\begin{enumerate}
\item Any two perfect matchings $M',M''\in \PM(G_n)$ containing $C$,  agree on the edges connecting the $T$-gadgets with other gadgets (when we refer to $M'$ and $M''$ as indicator functions).
\item There are exactly $3^{z(C)}$ perfect matchings of $G_n$ containing $M'$, where $z(C)$ is the number of $T$-gadgets with all of their incoming edges assigned with $0$ in perfect matchings containing $C$.
\end{enumerate}
\end{lemma}
\begin{proof}
\begin{enumerate}
\item Let $M$ be a perfect matching containing $C$. We will show that the assignment of edges connecting the $T$-gadgets with other gadgets in $M$ is uniquely determined by $C$. First, we note that the assignment of edges connecting $T$-gadgets and $IO$-gadgets is determined. For an edge connecting $T$ and $IO$ gadgets, if $e\in C\subset M$ we have that $M(e)=1$. Otherwise $M(e)=0$ as the vertex from $P_n$ in $e$ is already covered by another edge $e'\in C\subset M$.

Let $1\leq i_1 < i_2 < \cdots < i_p \leq 4n-1$ be the indices of the $T$-gadgets connected to $IO$-gadgets in $M$ (which are determined by $C$, as described above). We saw in the proof of Lemma \ref{lem:PerToPC} that each perfect cover of $P_n$ represents an injective function from the $n\times n$ square (denoted by $Q_n$) to itself with its image covering the $(n-1)\times (n-1)$ square remaining when we remove the lower and left edges (denoted by $Q_n^{int}$).  In this representation, vertices mapped outside of the square and vertices in the square with no pre-image are represented by edges connecting the $O$-vertices and $I$-vertices with $T$-gadgets respectively. In such an injective map, the number of  of vertices mapped outside of the square and the number of vertices in the square with no pre-image must be equal. Thus, the number of edges connecting the $O$-vertices and $I$-vertices with $T$-gadgets are equal. We conclude that $p$ is even, since it is the sum of those equal numbers.

For $1\leq l \leq 4n-2$, denote the edge connecting $T_{l}$ with $T_{l+1}$ by $e_l$. We claim that for all $1\leq l \leq 4n-2$, we have that $M(e_l)=1$ if there exists $1\leq j\leq \frac{p}{2}$ such that $i_{2j-1}\leq l < i_{2j}$ and  $M(e_l)=0$ otherwise. We will prove this claim by induction on $l$.

For $l=1$, note that $T_1$ is connected only two gadgets, $T_2$ and $IO_1$. By Lemma \ref{lem:K4}, the number of edges connecting $T_1$ with other gadgets assigned with $1$ must by even. Thus, $M(e_1)=1$ if and only if $i_1=1$. In case that $i_1=1$, we have $i_1 \leq 1 \leq i_2$ and  indeed $M(e_1)=1$. Otherwise, $1=l<i_1$ and $M(e_1)=0$. Assume that the claim is true for $1\leq l-1\leq 4n-3 $. We split into cases:
\begin{itemize}
\item If $l<i_1$, $T_l$ is not connected to an $IO$-gadget in $M$ and by the induction $M(e_{l-1})=0$. By Lemma \ref{lem:K4}, $M(e_l)=0$ as the number of edges in $M$ connecting to $T_1$ must be even.
\item If $l>i_p$, $T_l$ is not connected to an $IO$-gadget in $M$ and by the induction $M(e_{l-1})=0$. By Lemma \ref{lem:K4}, $M(e_l)=0$ as the number of edges in $M$ connecting to $T_1$ must be even.
\item If $l=i_{2j-1}$ for some $1\leq j \leq \frac{p}{2}$, $T_l$ is connected to an $IO$-gadget in $M$ and  by the induction assumption $M(e_{l-1})=0$. Again, by Lemma \ref{lem:K4}, $M(e_l)=1$ as the number of edges in $M$ connecting to $T_1$ must be even.
\item If $i_{2j-1}<l<i_{2j}$ for some $1\leq j \leq \frac{p}{2}$, $T_l$ is not connected to any $IO$-gadget and by the induction $M(e_{l-1})=1$. Hence, similarly to the previous case, $M(e_l)=1$.
\item If $l=i_{2j}$ for some $1\leq j \leq \frac{p}{2}-1$, $T_l$ is connected to an $IO$-gadget in $M$ and  by the induction assumption $M(e_{l-1})=1$. Thus $M(e_l)=1$.
\item Otherwise $i_{2j}<l<i_{2j+1}$ for some $1\leq j < \frac{p}{2}$, meaning that $T_l$ is not connected to any $IO$-gadget. By the induction step, $M(e_{l-1})=0$ and therefore $M(e_l)=0$ as well.
\end{itemize}   
\item  By the first part of this lemma, perfect matchings containing $C$ can only be different in the inner edges of the $T$-gadgets, this explains why $z(C)$ is well defined. By Lemma \ref{lem:K4}, for a $T$ gadgets with $2$ incoming edges assigned with 1, there is only one way to choose inner edges such that each of its vertices is covered by exactly one edge. For a $T$ gadgets with no incoming edges assigned with 1, there is three options to choose inner edges having any of its vertices covered by exactly one edge. Thus, we deduce that there are exactly $3^{z(C)}$ perfect matchings containing $C$.
\end{enumerate}
\end{proof}
\begin{theorem}
\label{th:PMcount}
For any $n\in\N$,
\[ \abs{B_{n,n}(\Omega(A_L))}=\PM(G_n,W), \]
where $W:E_n\to \R$ is the weighting function given by
\[W(e)= \begin{cases}
\frac{1}{3} & \text{if }e\in \mathset{\mathset{T_{k,i},T_{k,j}}: 1\leq k \leq 4n-1 \text{ and }1\leq i < j \leq 3 }\\
1 & \text{otherwise.}
\end{cases} \]
\end{theorem}
\begin{proof}
Let $C$ be a perfect cover of $P_n$ and $M\subset E_n$  be a perfect matching of $G_n$ containing $C$. In the proof of Lemma \ref{lem:K4} we see that for any gadget $T_k$ with no incoming edges assigned with $1$ in $M$, exactly one edge in  $\mathset{\mathset{T_{k,i},T_{k,j}}: 1\leq i < j \leq 3 }$ is assigned with $1$. It was also shown that for a gadget $T_k$,  with $2$ incoming edges assigned with $1$ in $M$, all of the edges in  $\mathset{\mathset{T_{k,i},T_{k,j}}: 1\leq i < j \leq 3 }$ assigned with $0$. Therefore we have,
\[ \prod_{e\in M}W(e)=\parenv{\frac{1}{3}}^{z(C)},  \]
where $z(M)$ is the number of $T$-gadgets with all of their incoming edges assigned with $0$ in perfect matchings containing $C$ (which is well defined by Lemma \ref{lem:CovToPM}). By Lemma \ref{lem:PerToPC}, any perfect cover $C\in \PC(P_n,G_n)$ has a unique $\pi\in B_{n,n}(\Omega(A_L))$ such that $C=C_\pi$. By Lemma \ref{lem:CovToPM},  $C=C_\pi$  is contained in exactly  $3^{z(C)}$ distinct perfect matchings of $G_n$. Clearly, every perfect matching of $G_n$ contains a perfect cover of $P_n$. Thus we have,
\begin{align*}
\PM(G_n,W)&=\sum_{M\in \PM(G_n)} \prod_{e\in M}W(e)\\
&=\sum_{C \in \PC(P_n,G_n)\text{ }}  \sum_{\underset{C\subseteq M}{M\in \PM(G_n)}} \prod_{e\in M}W(e)\\
&=\sum_{\pi\in B_{n,n}(\Omega(A_L))\text{ }}  \sum_{\underset{C_\pi \subseteq M}{M\in \PM(G_n)}} \prod_{e\in M}W(e)\\
&=\sum_{\pi\in B_{n,n}(\Omega(A_L))\text{ }}  \sum_{\underset{C_\pi \subseteq M}{M\in \PM(G_n)}} \parenv{\frac{1}{3}}^{z(C_\pi)}\\
&=\sum_{\pi\in B_{n,n}(\Omega(A_L))}  3^{z(C_\pi)} \cdot \parenv{\frac{1}{3}}^{z(C_\pi)}=\abs{B_{n,n}(\Omega(A_L)).}
\end{align*}
\end{proof}
\begin{theorem} There exists a polynomial-time algorithm for computing $\abs{B_{n,n}(\Omega(A_L))}$.
\label{th:Alg} 
\end{theorem}
\begin{proof}
Let $n\in \N$, we enumerate the vertices of $G_{n}$ by $\mathset{1,2,\dots,2n^2+4(4n-1)}$. Let, $W$ be the weighting function from Theorem \ref{th:PMcount}. By Proposition \ref{prop:PaffOr}, we may apply a polynomial-time algorithm to find a Pfaffian orientation for the $G_n$ (as it is planar). Denote it by $E\subseteq V^2$. Given such an orientation we consider the $(2n^2+4(4n-1))\times (2n^2+4(4n-1))$ weighted adjacency matrix, $A_G$ defined by  
\[ A_G(i,j)=\begin{cases}
W(i,j) & \text{If } (i,j)\in E'\\
W(i,j) & \text{If } -(j,i)\in E'\\
0 & \text{otherwise.} 
\end{cases} \]
We note that by its definition, $A_G$ is skew-symmetric and thus by Proposition \ref{prop:Kas} and Theorem \ref{th:PMcount} we obtain 
\[ \abs{B_{n,n}(\Omega(A_L))}=\PM(G_n,W)=\Pf(A_G)=\sqrt{\abs{\det(A_G)}}.\]
The complexity of computing $\det(A_G)$ is $O\parenv{\parenv{2n^2+4(4n-1)}^3}=O(n^6)$ integer operations.
\end{proof}

The algorithm described above, providing a method for counting the exact number of elements in $B_{n,n}(A_L)$, may be generalized for counting the number of perfect covers of sub-graphs of any planar graphs. That is, given a locally finite planar graph $G=(V,E)$ and a finite subset of vertices, $V'$, we may use a similar approach in order to count (in polynomial-time) the number of perfect covers of $V'$ inside $G$. This may be useful when we try to give an upper bound on the entropy of an SFT. 
\begin{theorem}
\label{th:GenAlg}
Let $G=(V,E)$ be a locally finite planar graph and $V'\subseteq V$ be a finite subset of vertices with even size that can be separated from $V\setminus V'$ by a simply connected domain (in some planar representation of $G$). Then, there exists a polynomial time algorithm for computing $\abs{\PC(V',G)}$.
\end{theorem}
\begin{proof}
The proof of this theorem is just a natural generalization of the proof of Theorem \ref{th:Alg}. Given $G$ and $V'$, let $S$ be the set of all vertices in $V\setminus V'$ connected to some vertex in $V'$. Since $G$ is locally finite and $V'$ is a finite set, $S$ is finite as well. Let $s_1,\dots,s_n$ be a clockwise order enumeration of $S$ (with respect to the planar representation of $G$ in which $V'$ and $V'\setminus V$ are separated by a simply connected domain). We construct a new weighted graph, $\hat{G}=(\hat{V},\hat{E},W)$, such that $\PC(V',G)=\PM(\hat{G},W)$. In the new graph, $\hat{G}$, vertices of $V'$ and edges between them remain as in the original graph. We replace any vertex $s_i\in S$ by a $T$-gadget (see Figure \ref{fig:L-gad}),$T_i$, and add edges connecting $T_{i,1}$ with all the vertices from $V'$ connected to $s_i$ in the original graph $G$. Finally, we add the edges connecting between the $T$-gadgets, 
\[ \mathset{T_{i,3},T_{i+1,2}}_{i=1}^{n-1},\]
(See Figure \ref{fig:GenralAlg}).
\begin{figure}
 \centering
  \includegraphics[width=150mm, scale=0.65]{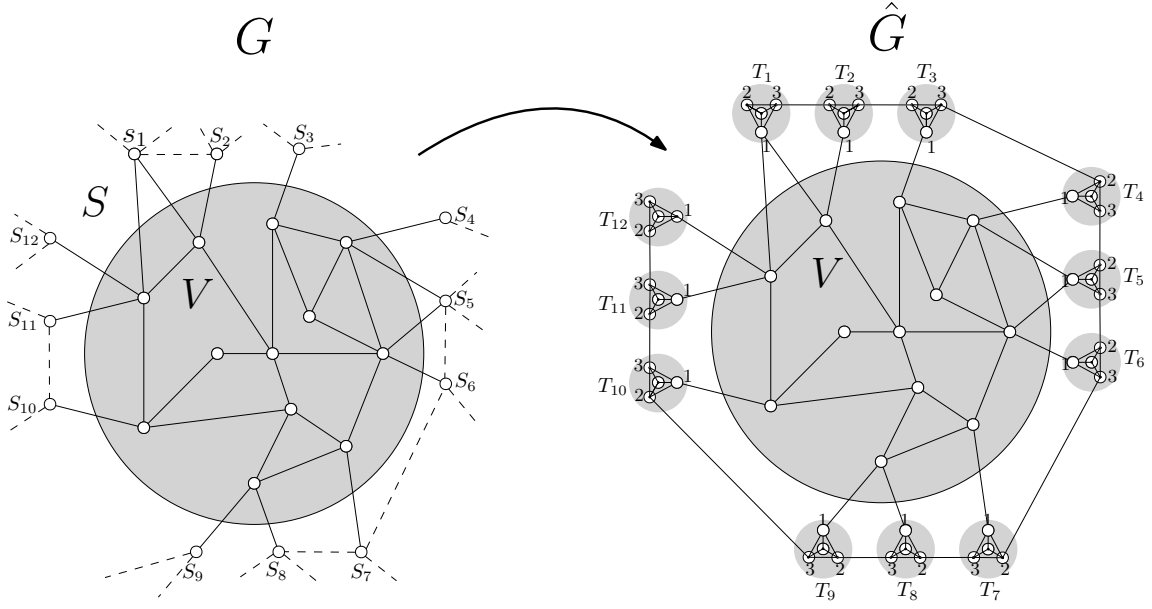}
  \caption{The construction of $\hat{G}$ from $G$.}
  \label{fig:GenralAlg}
\end{figure}
By construction of the $\hat{G}$, it is clear that perfect covers of $V'$ in $\hat{G}$ correspond bijectively with perfect covers of $V'$ in $G$ and therefore 
\[ \abs{\PC(V',\hat{G})}=\abs{\PC(V',G)}.\]
We  carefully repeat the steps of the proof of Lemma \ref{lem:CovToPM}, and claim that any perfect cover, $C\in \PC(V,\hat{G})$, is contained in exactly $3^{z(C)}$ perfect matchings of $\hat{G}$, where $z(C)$ is half of the number of edges in $C$ intersecting $S$, which is an even number since $|V'|$ is assumed to be even. In fact, this is the only assumption on $V'$ used in the proof of Lemma \ref{lem:CovToPM}. We set the weighting function,
\[W(e)= \begin{cases}
\frac{1}{3} & \text{if }e\in \mathset{\mathset{T_{k,i},T_{k,j}}: 1\leq k \leq n \text{ and }1\leq i < j \leq 3 }\\
1 & \text{otherwise.}
\end{cases} \]
We repeat the same argument as in Theorem \ref{th:PMcount} and obtain
\begin{align*}
\PM(V',\hat{G})&=\sum_{M\in \PM(\hat{G})}\prod_{e\in M}W(e)\\&=\sum_{C\in \PC(V',\hat{G})}\sum_{\underset{C\subseteq M}{M\in \PM(\hat{G})}}\prod_{e\in M}W(e)\\
&=\sum_{C\in \PC(V',\hat{G})}\sum_{\underset{C\subseteq M}{M\in \PM(\hat{G})}}\frac{1}{3^{z (C)}}\\
&=\sum_{C\in \PC(V',\hat{G})} 1= \abs{\PC(V',\hat{G})}.
\end{align*}
The assumption that $V'$ can be separated from $V\setminus V'$ by a simply connected domain ensures that the $T$ gadgets may be connected such that $\hat{G}$ stays planar.     
The final part of the proof, showing that $\PM(\hat{G})$ can be computed with complexity of $O(|V|^3)$ integer operations, are the same as in the proof of Theorem \ref{th:Alg}, relying on Kasteleyn's result.
\end{proof}	
	
	\section{Alternative Correspondence For Bipartite Graphs}
	\label{PM_alt}
In the first section of this chapter we described a general correspondence of restricted permutations by perfect matching. This correspondence proved to be useful for studying cases where the corresponding graph is a $\Z^2$-periodic bipartite planar graph. Unfortunately, this is usually not the case. In this section, we find an alternative correspondence of restricted movement permutations and perfect matchings, for the case where the original graph is bipartite. We use this correspondence in order to study  restricted permutations of the graph $G_{A_+}$ presented in Example \ref{ex:GenToZd}.
\begin{theorem}
\label{th:PerToPM}
Let $G=(V\uplus U,E)$ be a directed bipartite graph. There is an embedding of $\Omega(G)$ inside $\PM(G')\times \PM(G')$, where $G'=(V,E')$ is the undirected graph obtained by erasing the directions from the edges in $E$. Formally,
\[ E'=\mathset{\mathset{v,u}: (v,u)\in E}.\] 
That is, there exists an injective $\Phi:\Omega(G)\to \PM(G')^2$. If a group $H$ acts on $G$ by graph isomorphisms then it induces an action on $\PM(G')$ and the following diagram commutes: 
   \begin{align*}
\includegraphics[keepaspectratio=true,scale=1.15]{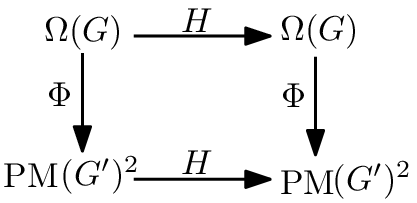}
\end{align*}
Furthermore, if $G$ is symmetric (that is, $(v,u)\in E$ implies $ (u,v)\in E$), $\Phi$ is a bijection.
\end{theorem}
\begin{proof}
Let $\pi\in \Omega(G)$ be a restricted permutation. Consider
\begin{align*}
&M_{\pi}^1\eqdef \mathset{\mathset{v,\pi(v)}:v\in V} \\
\end{align*}
and
\begin{align*}
&M_{\pi}^2\eqdef \mathset{\mathset{u,\pi(u)}:u\in U}.
\end{align*} 
Clearly, $M_{\pi}^1,M_{\pi}^2\subseteq E'$ are subsets of size $2$ since $\pi$ is restricted by $G$, and $G$ has no self loops. In particular $(x,\pi(x))\in E$ for all $x\in V\uplus U$. First, we verify that indeed $M_{\pi}^1$ and $M_{\pi}^2$ belong to $\PM(G')$. Let $v\in V$. Since $G$ is bipartite and $\pi$ is restricted by $G$, $\pi(v')\in U$ for any $v'\in V$ and in particular $\pi(v')\neq v$. Now, directly from the definition of $M_{\pi}^1$, it follows that $\mathset{v,\pi(v)}$ is the unique edge to cover $v$ in $M_{\pi}^1$. For $u \in U$, note that $v=\pi^{-1}(u) \in V$ and $\mathset{\pi^{-1}(u),u}\in E'$ as $\pi$ is restricted by $G$. We note that $\mathset{v,u}\in M_{\pi}^1$ as $\pi(v)=u$.

Assume to the contrary that $u$ is covered by another edge $\mathset{v',u}\in M_\pi^1$. It follows that $v'\in V$ as $G$ is bipartite and thus $\pi(v')=u$. This is a contradiction as $\pi$ is a permutation. The proof that $M_\pi^2\in \PM(G')$ is symmetric.

Define $\Phi(\pi)=(M_{\pi}^1,M_{\pi}^2)$. We now turn to to prove that $\Phi$ is a bijection. Let $\pi_1,\pi_2\in \Omega(G)$ be two distinct restricted permutations. Since $\pi_1\neq \pi_2$, there exists $x\in V\uplus U$ such that $\pi_1(x)\neq \pi_2(x)$, without loss of generality $x\in V$. From the definition of $M_{\pi_1}^1$ and  $M_{\pi_2}^1$, we have that 
\[ \mathset{x,\pi_1(x)}\in M_{\pi_1}^1\text{ and }   \mathset{x,\pi_2(x)}\in M_{\pi_2}^1. \]
Since $M_{\pi_1}^1$ and $M_{\pi_2}^1$ are perfect matchings of $G'$, $\mathset{x,\pi_2(x)}\notin M_{\pi_1}^1 $
as $x$ is already covered by the edge $\mathset{x,\pi_1(x)}$ in $M_{\pi_1}^1$. Thus, $ M_{\pi_1}^1\neq  M_{\pi_2}^1$ and in particular $\Phi(\pi_1)\neq \Phi(\pi_2)$. Let $H$ be a group acting on $G$ by graph isomorphisms. For any $v\in V$, $u\in U$, and $h\in H$ we have 
\begin{align*}
\mathset{u,v}\in E' &\iff (u,v)\in E \text{ or } (v,u)\in E\\ 
&\iff  (hu,hv)\in E \text{ or } (hv,hu)\in E\\ 
&\iff \mathset{hu,hv}\in E'.
\end{align*}
This shows that $H$ acts on $G'$ by isomorphisms and similarly to the proof of Theorem \ref{th:GenCan}, this action induces a group action of $G$ on $\PM(G')$ by 
\[ h(M)\eqdef \mathset{\mathset{hv,hu}~:~\mathset{v,u}\in M}.\]

In order to complete the first part of the theorem, it remains to show that for any $\pi\in \Omega(G)$ and $h\in H$ we have $\Phi(h(\pi))=h(\Phi(\pi))$. That is, $(hM_\pi^1,hM_\pi^2)=(M_{\pi^h}^1,M_{\pi^h}^2)$. Indeed,
\begin{align*}
hM_\pi^1&=\mathset{\mathset{hv,hu} ~:~ \mathset{v,u}\in M_\pi^1}\\
&=\mathset{\mathset{hv,h\pi(v)}~:~v\in V}\\
&=\mathset{\mathset{hv,h\pi(h^{-1}(hv))}~:~v\in V}\\
&=\mathset{\mathset{hv,(\pi^{h})(hv))}~:~v\in V}\\
&=\mathset{\mathset{u,(\pi^{h})(u))}~:~u\in hV}\\
\underset{hV=V}{\Rsh} &=\mathset{\mathset{u,(\pi^{h})(u))}~:~u\in V}=M_{\pi^{h}}^1.
\end{align*}
Symmetrically, we show that $hM_\pi^2=M_{\pi^h}^2$, which completes the first part of the proof.

Assume furthermore that $G$ is symmetric, we need to show that the map $\pi \to (M_\pi^1,M_\pi^2)$ is invertible. Given two perfect matchings, $M_1,M_2\in \PM (G')$, for any $x\in V\uplus U$ and $i \in \mathset{1,2}$, denote by $M_i(x)$ the unique vertex in $V\uplus U$ such that $\mathset{x,M_i(x)}\in M_i'$. Define $\pi_{M_1,M_2}:V\uplus U \to V\uplus U$ by 
\[ 
\pi_{M_1,M_2}(x)=\begin{cases}
M_1(x) & \text {if } x\in V\\
M_2(x) & \text {if } x\in U
\end{cases}.
\] 
Note that for any $x\in V\uplus U$,
\[ \mathset{x,\pi_{M_1,M_2}(x)} \in \mathset{\mathset{x,M_1(x)},\mathset{x,M_2(x)}} \subseteq E'.\]
Since $G$ is symmetric, $(x,M_1(x)),(x,M_2(x))\in E$ and therefore $(x,\pi_{M_1,M_2}(x))\in E$. That is, $\pi_{M_1,M_2}$ is restricted by $G$. Assume that $\pi_{M_1,M_2}(x)=\pi_{M_1,M_2}(x')$, since $G$ is bipartite and $\pi_{M_1,M_2}$ is restricted by $G$ we have that $x,x'\in V$ or $x,x'\in U$, without the loss of generality let us assume that $x,x'\in V$. From the definition of $\pi_{M_1,M_2}$, it follows that 
\[ M_1(x)=\pi_{M_1,M_2}(x)=\pi_{M_1,M_2}(x')=M_1(x'),\]
and therefore $x=x'$ (as $M_1$ is a perfect matching of $G$). This shows that $\pi_{M_1,M_2}$ is injective. Let $x\in U$, we note that $M_1(x)\in V$ and $M_1(M_1(x))=x$. From the definition $\pi_{M_1,M_2}$, it follows that 
\[ \pi_{M_1,M_2}(M_1(x))=x.\]
Similarly, if $x\in V$, $\pi_{M_1,M_2}(M_2(x))=x$ and $\pi_{M_1,M_2}$ is onto $V\uplus U$. Define $\Psi:\PM(G')^2\to \Omega(G)$ by 
\[ \Psi(M_1,M_2)=\pi_{M_1,M_2}.\]
It is easy to see that $\Psi\circ \Phi$ and $\Phi\circ \Psi$ are the identity functions on $\PM(G')^2$ and $\Omega(G)$ respectively.
\end{proof}
\begin{corollary}
\label{cor:undirected-PM}
For any $G$ undirected  bipartite graph, there is a bijection between $\Omega(G)$ and $\PM(G)^2$. 
\end{corollary}
\begin{proof}
Recall that we may identify it with a symmetric directed graph not containing self loops, $\tilde{G}$, such that $\Omega(G)=\Omega(\tilde{G})$. We note that $(\tilde{G})'=G$. Thus, by applying Theorem \ref{th:PerToPM} on $\tilde{G}$ we deduce that there is a bijection between $\Omega(G)$ and $\PM(G)^2$.
\end{proof}
\begin{example}
Consider the two-dimensional (undirected) honeycomb lattice from Example \ref{ex:Honeycomb}, denoted it by $L_H$. By Corollary \ref{cor:undirected-PM}, restricted permutations of the honeycomb lattice correspond with pairs of perfect matchings of $L_H$.  In Section \ref{A_L} we have shown that perfect matchings of the honeycomb lattice are in 1-1 correspondence with permutations of $\Z^2$ restricted by the set $A_L$. Combining the results, we conclude that restricted permutations of the honeycomb lattice correspond with pairs of $\Z^2$ permutations restricted by $A_L$. Formally, we define the action of $\Z^2$ on $\Omega(A_L)\times \Omega(A_L) $ by 
\[ n(\pi_1,\pi_2)\eqdef (\pi_1^n,\pi_2^n).\]
It is easy to verify that it is indeed a continuous group action on $\Omega(A_L)\times \Omega(A_L)$ (with the product topology). There exists a bijection $\Phi:\Omega(L_H)\to \Omega(A_L)\times \Omega(A_L)$ which commutes with the action of $\Z^2$. That is, for any $n\in \Z^2$ and $\pi\in \Omega(L_H)$,
\[ \Phi(\pi^n)=\parenv{(\Phi \pi)_1)^n,(\Phi \pi)_2)^n}.\] 
See Figure \ref{fig:HoneycombPerToPMtoAL} for an illustration of the correspondence of a restricted permutation of the honeycomb lattice, pairs of perfect matchings and permutations of $\Z^2$ restricted by $A_L$.
\begin{figure}
 \centering
  \includegraphics[width=150mm, scale=0.65]{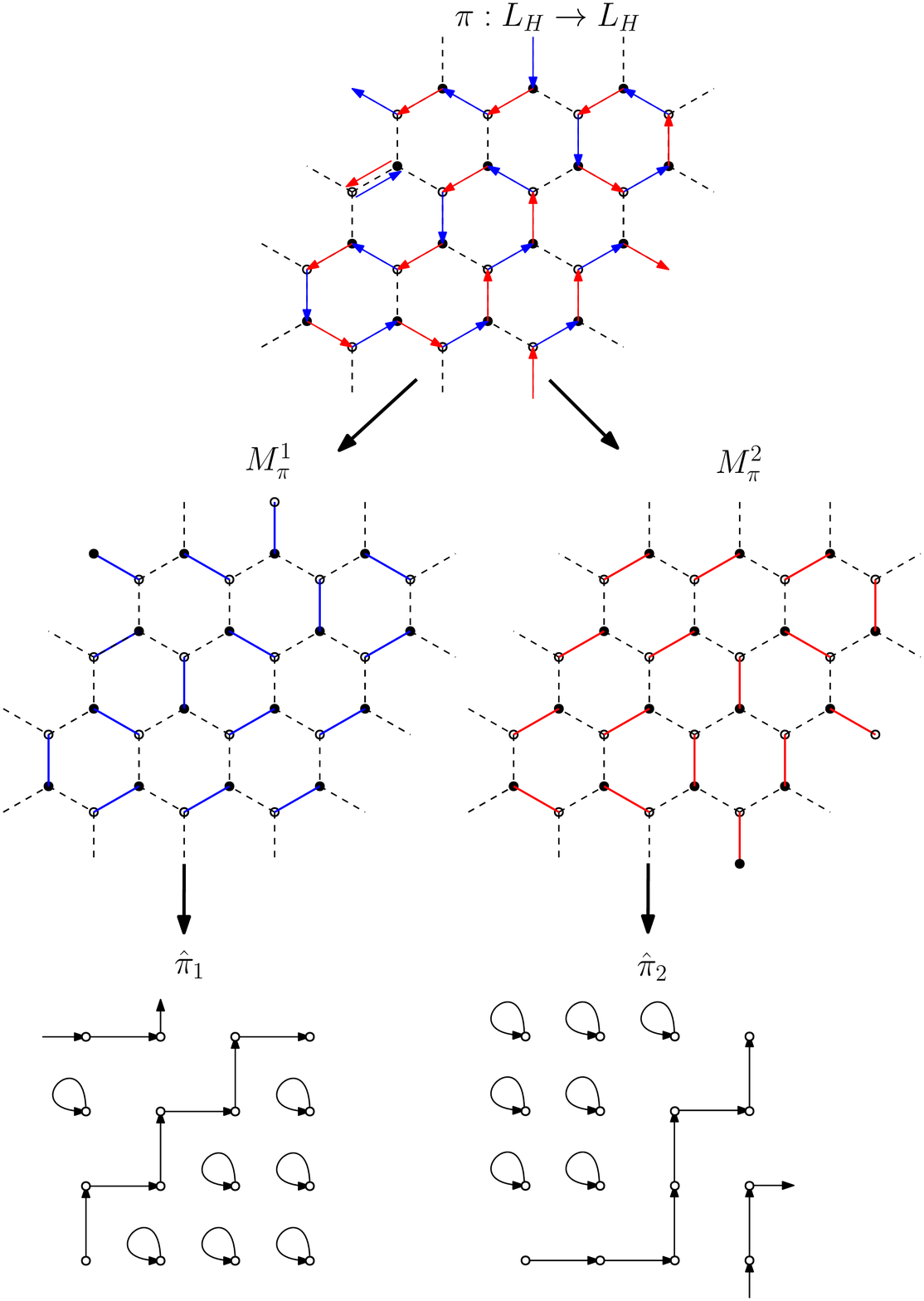}
  \caption{The correspondence between a restricted permutation of the honeycomb lattice, perfect matchings and permutations of $\Z^2$ restricted by $A_L$.}
  \label{fig:HoneycombPerToPMtoAL}
\end{figure}
\end{example}
		
	\subsection{Permutations of $\Z^2$ Restricted by $A_+$ }
		\label{A_+}
In this section, we consider the case of permutations of $\Z^2$ restricted by the set $A_+=\mathset{(0,\pm 1),(\pm 1,0)}$, presented in Example \ref{ex:GenToZd}. In that case, the corresponding graph described in Theorem   \ref{th:GenCan} (the general correspondence) is $\Z^2$-periodic bipartite graph, but in this $\Z^2$-periodic presentation it has intersecting edges (see Figure \ref{fig:Gen_A+}). Thus, we can not use the results from the theory of $\Z^2$-periodic bipartite planar graphs as in the case of $\Omega(A_L)$. Fortunately, the graph $G_{A_+}$ (see Figure \ref{fig:Graph_PL}) is $\Z^2$-periodic, bipartite, and planar (when we think of it as an undirected graph by removing the directions from the edges). Using the alternative correspondence to perfect matchings, we have that restricted permutation of $G_{A_+}$ correspond to pairs of perfect matchings of the square lattice in $\Z^2$, which we denote by $L_S$. 
\begin{figure}
 \centering
  \includegraphics[width=80mm, scale=0.5]{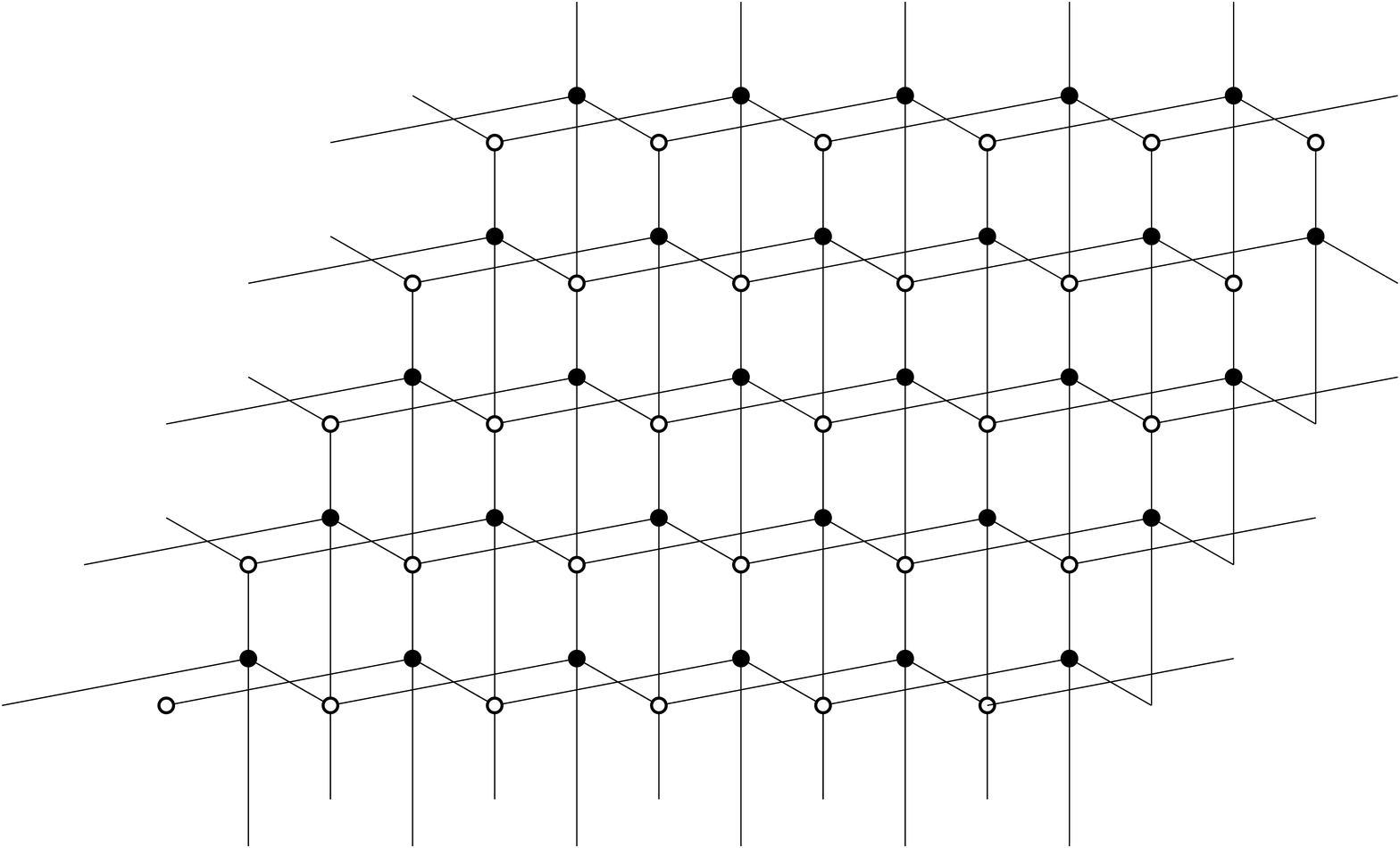}
  \caption{The corresponding graph for $G_{A_+}$ from Theorem \ref{th:GenCan}.  }
  \label{fig:Gen_A+}
\end{figure}

 The problem of finding the exponential growth rate of perfect matchings (also called Dimer coverings) of the square lattice , also known as the square lattice Dimer problem or Domino Tiling Problem, was studied thoroughly in the last century  (see \cite{Fis61, Kas61, ChoKenPro01, Fis66, TemFis60}). 
\begin{figure}
 \centering
  \includegraphics[width=80mm, scale=0.5]{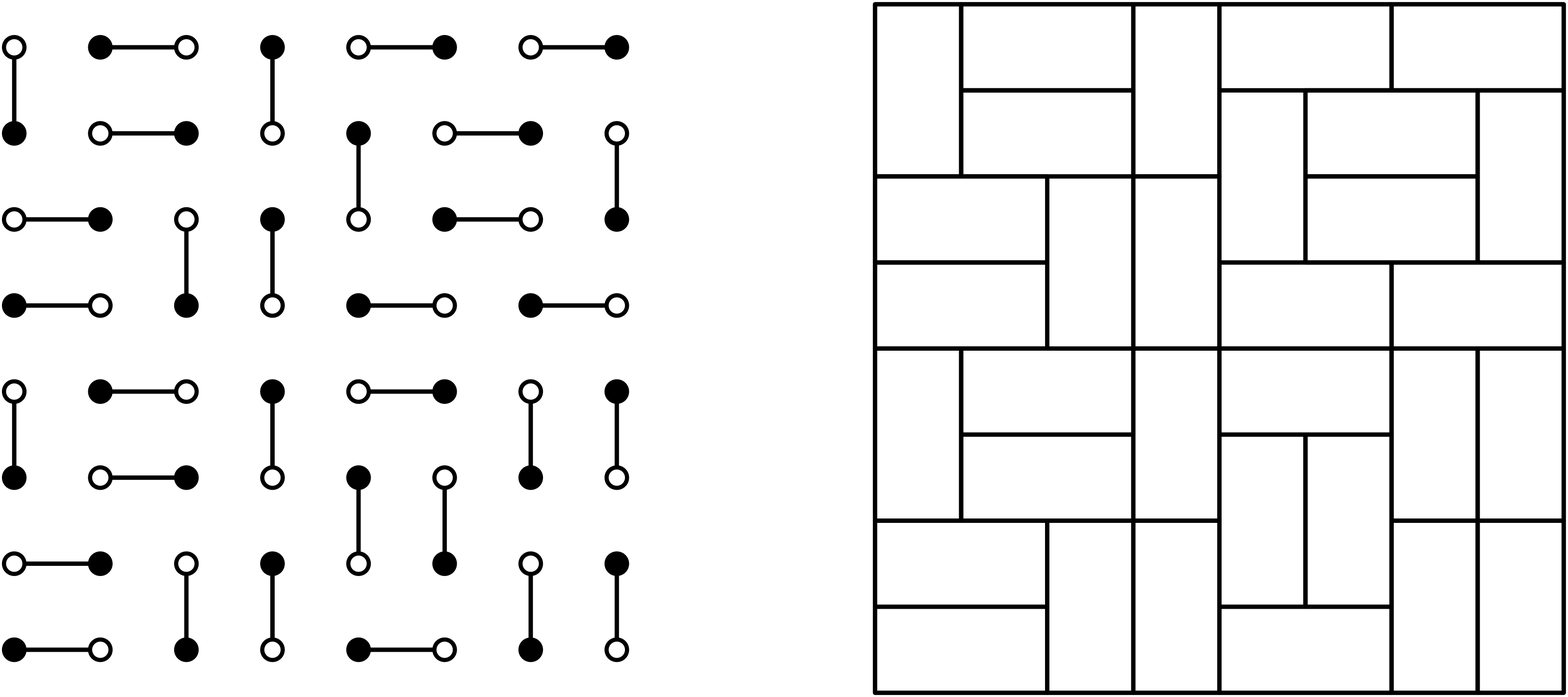}
  \caption{Dimer covering / tiling of the plane.}
  \label{fig:DimerEx}
\end{figure}

We start by showing that $\PM(L_S)$ can be formulated as a two-dimensional SFT, we later show that $\Omega(A_+)$ is conjugated to the cartesian product of this SFT with itself. We use some well known results regarding the square lattice Dimer model in order to find the (topological, periodic and closed) entropy. Finally, we discuss methods for counting patterns in polynomial-time.

Given $M\in \PM(L_S)$ and $n\in \Z^2$, there exists a unique element in $ \Z^2$ such that $\mathset{n,M(n)}\in M$, which we denote by $M(n)$. Furthermore, from the definition of the square lattice $M(n)\in n-A_+$. For all $n\in \Z^2$, we define 
\[\omega_M(n)\eqdef M(n)-n.\]
It is easy to verify that $\omega_M\in A_+^{\Z^2}$ is an injective embedding of $\PM(L_S)$ in $A_+^{\Z^2}$. Since $M$ is a perfect matching, it has the property that $M(M(n))=n$ for all $n\in \Z^2$. This property may be interpreted as
\begin{align*}
n=M(M(n))&\underset{\Downarrow}{=}M(n)+\omega_M(M(n))\\
\omega_M(n)=M(n)-n=-&\omega_M(M(n))=-\omega_M(n+\omega_M(n)).
\end{align*} 
Consider the set 
\[ \Omega_D\eqdef \mathset{\omega\in A_+^{\Z^2}: \forall n\in \Z^2, \omega(n)=-\omega(n+\omega(n))}. \]

We saw that for any $M\in \PM(L)$, $\omega_M\in \Omega_D$. It is not difficult to verify that any element $\omega\in \Omega_D$ defines a perfect matching of $S$ by 
\[ M_\omega=\mathset{\mathset{n,n+\omega(n)}:n\in \Z^2}.\] 
Furthermore, $M_{\omega_M}=M$, meaning that we can encode the elements of $\PM(L_S)$ by the elements in $\Omega_D$ bijectively.

In order to check whether an element $\omega\in A_+^{\Z^2}$ is in $\Omega_D$. it is sufficient to check that the condition $\omega(n)=-\omega(n+\omega(n))$ is satisfied in all of the coordinates $n\in \Z^2$. We note that it is sufficient to check that this condition is not violated in any $2\times 2$ sub-array. Thus, $\Omega_D$ is the SFT defined by the set of forbidden patterns $F_D$ defined as 
\[ F_D\eqdef \mathset{w\in A_+^{[2]\times [2]} :\begin{array}{l}
   \exists n\in [2]\times [2] \such w+w(n)\in [2]\times [2] \\ \text{ and } w(n)\neq-w(n+w(n))\end{array}}. \]

We consider the topological space $\Omega_D^2\eqdef \Omega_D\times \Omega_D$, equipped with the product topology (the product of the topology of $\Omega_D$ with itself). Claraly, $\Omega_D^2$ is a compact space (as $\Omega_D$ is compact). We have $\Z^2$ acting on $\Omega_D^2$ by 
\[ n(\omega_0,\omega_1)\eqdef (n\omega_0,n\omega_1),\]
which is a continuous operation as the action of $\Z^2$ on $\Omega_D$ is continuous. That is, the pair $(\Omega_D^2,\Z^2)$ forms a topological dynamical system. When we identify $\PM(L_S)\times \PM(L_S)$ with $\Omega_D^2$, by Theorem \ref{th:PerToPM}, there is a bijection $\Phi:\Omega(A_+)\to\Omega_D^2$ which commutes with the action of $\Z^2$.

By the construction of the bijection described in the proof  of Theorem \ref{th:PerToPM}, we observe that we can view $\Omega(A_+)$ as a subset of $A_+^{\Z^2}$ (see Chapter \ref{CHA:preliminaries}). $\Phi$ is given by $\Phi(\omega)=(\omega_0,\omega_1)$ where
\[ \omega_i(n,m)=\begin{cases}
\omega(n,m) & \text{if } (n+m) \bmod 2=i \\
-\omega\parenv{\pi_\omega^{-1}(n,m)} & \text{otherweise,}
\end{cases} \] 
and $\pi_\omega$ is the permutation identified with $\omega$ .
 Furthermore, this action is a homeomorphism as the pre-image and image of a cylinder sets are cylinder sets. This shows that $\Omega(A_+)$ and $\Omega_D^2$ are topologically conjugated. 
\begin{theorem}
\label{th:PlusEnt}
\begin{align*}
\topent(\Omega(A_+))=2 \topent \parenv{\Omega_D}
\end{align*}
\end{theorem}
\begin{proof}
Let $n_1,n_2 \in\mathbb{N}$, we will define $\phi: B_{n_1,n_2}(\Omega(A_+))\to B_{n_1,n_2}(\Omega_D)\times B_{n_1,n_2}(\Omega_D) $ as follow: for $v\in B_{n_1,n_2}(\Omega(A_+))$ let $\omega\in\Omega(A_+)$ such that $\omega\parenv{[n_1]\times[n_2]}=v$, define $\phi(v)=(v_0,v_1)$ as 
\begin{align*}
(v_0,v_1)=(\omega_0\parenv{[n_1]\times[n_2]},\omega_1\parenv{[n_1]\times[n_2]})
\end{align*}
where $(\omega_0,\omega_1)=\Phi(\omega)$ and $\Phi$ is the conjugation function $\Omega(A_+)\to \Omega_D^2$. First, we claim that $\phi$ is well defined. By the construction of $\Phi$, if for $\omega,\omega'\in \Omega(A_+)$ we have 
\[ \omega\parenv{[n_1]\times[n_2]}=\omega'\parenv{[n_1]\times[n_2]} \] 
then 
\[ (\omega_0\parenv{[n_1]\times[n_2]},\omega_1\parenv{[n_1]\times[n_2]})=(\omega_0'\parenv{[n_1]\times[n_2]},\omega_1'\parenv{[n_1]\times[n_2]}). \] 
This shows that the definition of $\phi$ does not depend on the choice of $\omega$.
Let $v, v'\in B_{n_1,n_2}(\Omega(A_+))$ such that $v\neq v'$, then there exists $i=(i_1,i_2)\in [n_1]\times[n_2]$ for which $v(i)\neq v'(i)$. If $(i_1+i_2) \bmod 2 =0$, from the definition of $f$ we have 
\begin{align*}
v_0(i)=v(i)\neq v'(i)=v_0'(i)
\end{align*}
Similarly we show that if If $(i_1+i_2) \bmod 2 =1$, $v_1(i)\neq v_1'(i)$ and $\phi(v)\neq \phi(v')$. This proves that $\phi$ is injective.

Let $v_0,v_1\in B_{n_1,n_2}(\Omega_D)$, there exists $\omega_0,\omega_1\in\Omega_D$ such that 
$\omega_0\parenv{[n_1]\times[n_2]}=v_0$ and $\omega_1\parenv{[n_1]\times[n_2]}=v_1$ and $\omega=\Phi^{-1}(\omega_0,\omega_1)$. Let  $v\eqdef \omega([n_1]\times[n_2])$. From the definition of $\phi$, it follows that $\phi(v)=(v_0,v_1)$. We have shown that $\phi$ is a bijection, and therefore 
\begin{align*}
\abs{ B_{n_1,n_2}(\Omega(A_+))}=\abs{ B_{n_1,n_2}(\Omega_D)\times  B_{n_1,n_2}(\Omega_D)}=\abs{ B_{n_1,n_2}(\Omega_D)}^2.
\end{align*}
Finally, 
\begin{align*}
\topent(\Omega(A_+))&=\limsup_{n_1,n_2\to\infty} \frac{\log\parenv{\abs{ B_{n_1,n_2}(\Omega(A_+))}} }{n_1\cdot n_2}\\
&=\limsup_{n_1,n_2\to\infty} \frac{\log\parenv{\abs{ B_{n_1,n_2}(\Omega_D)}^2} }{n_1\cdot n_2}\\
&=2\cdot \limsup_{n_1,n_2\to\infty} \frac{\log\parenv{\abs{ B_{n_1,n_2}(\Omega_D)}} }{n_1\cdot n_2}=2\topent(\Omega_D).
\end{align*}
\end{proof}

The equivalence between the double Dimer model and $\Omega(A_+)$ we have just proved may also be used in order to find the periodic and closed entropy of $\Omega(A_+)$.  For $n\in \N$, let $L_{S,n}=(V_n,E_n)$ be the $n\times n$ square sub-graph of the square lattice, that is 
\[ V_n\eqdef [n]\times [n], \ E_n\eqdef \mathset{\mathset{v_0,v_1}\in ([n]\times [n])^2 : \norm{v_0-v_1}_1=1 } .\]
Let $L_{S,n}^T=(V_n^T,E_n^T)$ be the $n\times n$ square lattice on the torus, 
\[V_n^T\eqdef [n]\times [n], \ E_n^T\eqdef \mathset{\mathset{v_0,v_1}\in ([n]\times [n])^2 : \norm{(v_0-v_1)\bmod (n,n)}_1=1 }. \]
In Kasteleyn's original work \cite{Kas61}, an exact formula for $\abs{\PM(L_{S,2n}^T)}$ was given, which later used to show that
\[ \lim_{n\to \infty} \frac{\log_2\abs{\PM(L_{S,2n}^T)}}{4n^2}=\frac{1}{4}\cdot \intop_0^1 \intop_0^1(4-2\cos(2\pi x) -2\cos (2\pi y))dxdy.\]
It is also shown in \cite{Kas61}, that the exponential growth rate of $\abs{\PM(L_{S,n})}$ is the same as  $\abs{\PM(L_{S,n}^T)}$, That is 
\[ \lim_{n\to \infty} \frac{\log_2\abs{\PM(L_{S,2n}^T)}}{4n^2}= \lim_{n\to \infty} \frac{\log_2\abs{\PM(L_{S,2n})}}{4n^2}.\]

We saw that a permutation of $\Z^2$ restricted by $A_+$ correspond to a pair of dimer coverings of $L_S$. For $n\in \N$, we consider the restriction of the bijection $\Phi: \Omega(A_+) \to \PM(L_S)\times \PM(L_S)$ to $\fix_{n\Z^2}(\Omega(A_+))$. In a similar fashion as in the proof of Theorem \ref{th:PlusEnt}, this restriction is a bijection between periodic permutations from $\fix_{n\Z^2}(\Omega(A_+))$ and pairs of periodic points in $\fix_{n\Z^2}(\Omega_D^2)=\fix_{n\Z^2}(\Omega_D)\times \fix_{n\Z^2}(\Omega_D)$, which in the perspective of perfect matings, represent elements in $\PM(L_n^T)$. That is, there is a bijection between $\fix_{n\Z^2}(\Omega(A_+))$ and $\PM(L_{S,n}^T)\times \PM(L_{S,n}^T)$. We recall that elements in $B_{n,n}^f(\Omega(A_+))$ (closed permutations of $[n]\times [n]$), represent a subset of $\fix_{n\Z^2}(\Omega(A_+))$. Hence, we similarly have a bijection between closed permutations of $[n]\times [n]$ and a subset of $\PM(L_{S,n}^T)\times \PM(L_{S,n}^T)$. This subset is exactly $\PM(L_{S,n})\times \PM(L_{S,n})$.

By a similar calculation as in the proof of Theorem \ref{th:PlusEnt} we obtain,
\[ \topent_c(\Omega(A_+))=\topent_p(\Omega(A_+))=2 \lim_{n\to \infty} \frac{\log_2\abs{\PM(L_{n}^T)}}{n^2} \] 
In their work, Meyerovitch and Chandgotia  \cite{MeyCha19} explain the well known result 
\[ \topent (\Omega_D)=\lim_{n\to \infty} \frac{\log_2\abs{\PM(L_{S,2n}^T)}}{4n^2}. \]
In their proof, they use a principle called reflection positivity, relying on the symmetry of the uniform measure on perfect matchings,  with respect to reflection along
some hyperplanes. 
We combine these results and obtain:
\begin{theorem}
\label{th:Sol+}
\[ \topent_c(\Omega(A_+))=\topent_p(\Omega(A_+))=\topent(\Omega(A_+))=\frac{1}{2}\cdot \intop_0^1 \intop_0^1(4-2\cos(2\pi x) -2\cos (2\pi y))dxdy. \]
\end{theorem}  

\begin{remark}
 The correspondence between Dimer coverings and restricted permutations described in this part of the work can be visualized by the general correspondence to perfect matchings of $G_{A_+}'$, described in Section \ref{PM_char}. We recall that $G_A'$ is composed from two copies of $\Z^2$. Given a permutation of $\Z^2$ restricted by $A_+$, we may consider the correspondent perfect matching in $G_A'$. If we draw the two copies of $\Z^2$ such that the odd vertices (i.e., vertices with odd sum of components) of one copy are drawn with the even vertices of the other copy, we get a visualization of its two corresponding dimer coverings described in Theorem \ref{th:PlusEnt} (See Figure \ref{fig:DimD})
 \begin{figure}
 \centering
  \includegraphics[scale=0.75]{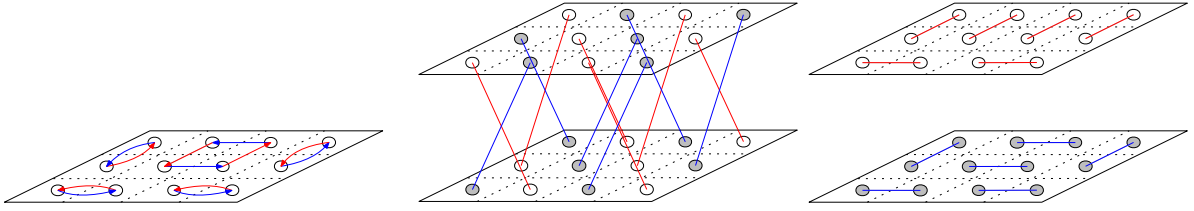}
  \caption{The correspondence between restricted permutation and Dimer coverings.}
  \label{fig:DimD}
\end{figure}
\end{remark}
\begin{remark}
Kasteleyn provided an exact formula for $\abs{\PM(L_{S,2n}^T)}$ \cite{Kas61}. This formula can be used in order to compute the exact number of toral permutations restricted by $A_+$, as $\abs{\fix_{n\Z^2}(\Omega(A_+))}=\abs{\PM(L_{S,2}^T)}^2$. In order to get a more complete picture, we want to be able to compute the exact number of patterns in $B_n(\Omega(A_+))$ and in $B_n^f(\Omega(A_+))$ (closed permutations of $[n]$ restricted by $A_+$), for any given $n\in\N^d$. We already know that 
\[ \abs{B_n^f(\Omega(A_+))}=\abs{\PM(L_{S,n})}^2. \]
Since $L_{S,n}$ is a finite planar graph, using Kasteleyn's method, $\abs{\PM(L_{S,n})}$ is computable in polynomal-time. Therefore, $\abs{B_n^f(\Omega(A_+))}$ is computable in polynomial-time as well. For the computation of $\abs{B_n(\Omega(A_+))}$, we recall that by Theorem \ref{th:PlusEnt}, $\abs{B_n(\Omega(A_+))}=\abs{B_n(\Omega_D)}^2$. Elements of $B_n(\Omega_D)$ represent perfect coverings of $[n]$ in $L_S$. By Theorem \ref{th:GenAlg}, $\abs{\PC([n],L_S)}$ is computable in polynomial-time, and therefore 
$\abs{B_n(\Omega(A_+))}=\abs{\PC([n],L_S)}^2$ as well.
\end{remark}

\chapter{Entropy}
\label{CHA:Entropy}
In this chapter, we investigate the entropy of dynamical systems defined by permutations with restricted movement. We focus on permutations of $\Z^d$ restricted by some finite set $A\subseteq \Z^d$. We start by proving some basic properties of the topological entropy for such SFTs and use them in order to find the topological entropy whenever $d=2$ and $\abs{A}=3$. Later, we further study the entropy of the important one-dimensional case where $A=[-k,k]$. This case was presented and studied by Schmidt and Strasser in \cite{SchStr17}. We use the results presented in \cite{SchStr17} in order to show equality between the closed, periodic and topological entropy. We discuss the topic of global and local admissibility of patterns, and use the results of this section in order bound the entropy in some specific case where $\abs{A}=5$.  In the last part, we review two related models of injective and surjective restricted functions of graphs.

	\section{Properties}
	\label{EntProperties}
We show that the entropy of $\Z^d$-permutations, restricted by some finite set $A$, is invariant under the operation of an injective affine transformation on $A$. Furthermore, we prove that conjugacy holds in the case of volume preserving affine transformation. That is, an affine transformation of the form $x\to Mx+b$, where $\det(M)=\pm 1$.

\begin{fact}
\label{fact:ShiftInv}
(\cite{SchStr17}, Proposition 1.1) Let $d\geq 1$, and $A\subseteq\mathbb{Z}^{d}$ be a finite set. For any $b\in \Z^d$, $\Omega(A)$ and $\Omega(A+b)$ are topologically conjugate (where $A+b$ denotes $\sigma_b(A)$). 
\end{fact}
\begin{proposition}
\label{prop:EntProp}
Let $d\geq 1$, and $A\subseteq\mathbb{Z}^{d}$ be a finite set.
For any group isomorphism $M\in \Aut(\Z^d)$, $\Omega(A)$ (with the usual action of $\Z^d$) is topologically conjugated to $\Omega(M(A))$ when an elements $a\in\Z^d$ acts by $\sigma_{Ma}$ and
\[ M(A) \eqdef \mathset{M(a)~:~a\in A}. \] 
\end{proposition}
\begin{proof} Given $\pi\in \Omega(A)$, we consider the function given by 
$\Phi(\pi)=\pi^M$. That is, $(\Phi(\pi))(n)\eqdef M\parenv{\pi(M^{-1}n)}$. Since $M\in \Aut(\Z^d)$, we have $M^{-1}n\in \Z^d$ for all $n\in \Z^d$ and $\Phi(\pi)$ is well defined. Furthermore, $M^{-1}\in \Aut(\Z^d)$ and therefore $\Phi(\pi)=M\circ \pi \circ M^{1}$ is also a permutation of $\Z^d$ (as a composition of bijective function). For $n\in \Z^d$, we compute
\begin{align*}
\Phi(\pi)(n)-n&=M\pi(M^{-1}n)-n\\
&=M\pi(M^{-1}n)-M(M^{-1}n)\\
&=M \underset{\in A}{\underbrace{\parenv{\pi(M^{-1}n)-M^{-1}n}}}\in M(A).
\end{align*}
This shows that $\Phi$ indeed maps $\Omega(A)$ to $\Omega(M A)$. When we identify the elements of $\Omega(A)$ and $\Omega(M( A))$ with elements of $A^{\Z^d}$ and $M (A)^{\Z^d}$, the calculation above shows that $\Phi$ acts by $(\Phi \omega)(n)=M \omega (M^{-1}n)$.

Clearly, $\Phi$ is invertible as its inverse is given by $\Phi^{-1}(\omega)=M^{-1}\circ \omega \circ M$, and it is an homeomorphism as  it is easy to verify that $\Phi$ and $\Phi^{-1}$ maps cylinder sets to cylinder sets. Now, we note that 
\begin{align*}
\sigma_a(\Phi(\omega))(n)&=\phi(\omega)(n+a)\\
&=M \omega(M^{-1}(n+a))\\
&=M (\sigma_{M^{-1}a}(\omega))(M^{-1}n)=\Phi(\sigma_{M^{-1}a}(\omega))(n)
\end{align*} 
This shows that $\Phi \circ \sigma_{ a}=\sigma_{Ma} \circ \Phi$ and completes the proof.
\end{proof}

Let $A\in\Z^d$ be some finite set, $M\in \SL(d,\Z)$ and $a\in \Z^d$, where $\SL(d,\Z)$ denotes the set of all $d\times d$ matrix with integer entries and determinant with absolute value of $1$. Combining the two parts of Proposition \ref{prop:EntProp} and Fact \ref{fact:ShiftInv}, we have that $\Omega(A)$ is topologically conjugated to $\Omega(M\cdot A+b)$ when an element $a\in\Z^d$ acts by $\sigma_{Ma}$ and $M\cdot A \eqdef \mathset{M\cdot a~:~a\in A}$. Since entropy is invariant under conjugacy, we have that $\topent(\Omega(A))=\topent(\Omega(A))=\topent(\Omega(M'\cdot A+b))$. For a matrix $M'\in \GL(d,\Z)$ (an invertible $d\times d$ matrix with integer entries), $\Omega(A)$ and $\Omega(M\cdot A+b)$ are not necessarily conjugate. However, we will now show that the equality of the entropies does hold anyway.

\begin{proposition}
\label{prop:MatEnt}
Let $d\geq 2$ and let $A\subseteq \Z^2$ be a finite set. For any $d\times d$ invertible integer matrix we have 
\[ \topent(\Omega(A))=\topent(\Omega(M\cdot A)).\]
\end{proposition}
\begin{proof}
If $\det(M)=\pm 1$, $M^{-1}$ is also an integer matrix and therefore by Proposition \ref{prop:EntProp} the desired equality holds. Otherwise, since $M$ is an integer matrix, the set $H\eqdef M\Z^d=\mathset{Mn~:~n\in \Z^d}$ is a subgroup of $\Z^d$ of index $k=|\det(M)|$. Let $P$ be the $d$-dimensional parallelepiped formed by the vectors $M e_1,\dots,M e_d$ where $e_1,\dots,e_d$ are the vectors of the standard basis. That is, 
\[ P\eqdef \mathset{\sum_{i=1}^d p_i\cdot Me_i~:~ p_1,p_2,\dots,p_d\in [0,1)}. \] 
Note that $\parenv{\biguplus_{v\in H}(P+v)\cap \Z^d}=\Z^d$ and theretofore any coset of $H$ has a unique vector in $P$. We enumerate the cosets of $H$ by $\mathset{H_i=H+v_i}_{i=1}^{k}$, where $v_i$  is the unique vector in $H_i \cap P$. We now define a map from $\Omega(M\cdot A)$ to $\Omega(A)^k$. Given $\omega\in \Omega(M\cdot A)\subseteq (M\cdot A)^{\Z^d}$, we define $\Phi(\omega)\eqdef (\omega_1,\omega_2,\dots, \omega_k)$, where $\omega_i\eqdef M^{-1}\circ \omega \circ \sigma_{v_i} \circ  M$, and $\sigma_{v_i}$ is the regular shift by $v_i$ in $\Z^d$. That is,
\[ \omega_i(n) =M^{-1} \omega(Mn+v_i). \]
Clearly, $\omega(Mn+v_i)\in M\cdot A$ and therefore $\omega_i(n) =M^{-1} \omega(Mn+v_i)\in A$ and $\omega_i\in A^{\Z^d}$ for all $1\leq i \leq k$.  Now we need to show that $\omega_j\in \Omega(A)$ for all $j$. That is, $\pi_{\omega_j}$ is indeed a permutation of $\Z^d$  (where as usual, $\pi_{\omega_j}$ is defined by $\pi_{\omega_j}(n)=n+\omega_j(n)$). 
\begin{itemize}
\item Injectivity - let $n,n'\in \Z^d$, since $\pi_\omega$ is a permutation of $\Z^d$ we have:
\begin{align*}
&\pi_{\omega_j}(n)=\pi_{\omega_j}(n')\\
&\iff n+\omega_j(n)=n'+\omega_j(n')\\
&\iff M( n+\omega_j(n))=M(n'+\omega_j(n'))\\
&\iff M(n)+v_j+ M(\omega_j(n))= M(n')+v_j+ M(\omega_j(n'))\\
&\iff M(n)+v_j+ M\parenv{M^{-1}(\omega(M(n)+v_j))}= M(n')+v_j+ M\parenv{M^{-1}(\omega(M(n')+v_j))}\\
&\iff \underset{\pi_\omega(M(n)+v_j)}{\underbrace{M(n)+v_j+\omega(M(n)+v_j)}}=\underset{\pi_\omega(M(n')+v_j)}{\underbrace{M(n')+v_j+\omega(M(n')+v_j})}\\
&\iff M(n')+v_j=M(n')+v_j\iff n=n'.
\end{align*}
\item Surjectivity - let $n\in \Z^d$. There exists $m\in \Z^d$ such that $\pi_{\omega}(m)=M(n)+ v_j$. Since $\pi_\omega$ is restricted by $M\cdot  A$, $m$ belongs to the same coset as $M(n)+v_j$, which is $H_j$.
Thus, $m$ is of the form $M(m')+v_j$ for some $m'\in \Z^d$. We have 
\begin{align*}
\pi_{\omega_j}(m')&=m'+\omega_j(m')\\
&=m'+M^{-1}(\omega(M(m')+v_j))\\
&=M^{-1}(\underset{m}{\underbrace{M(m' ) +v_j}} +\underset{\omega(m)}{\underbrace{\omega(M(m')+v_j)}})-	M^{-1}(v_j)\\
&=M^{-1}(\underset{\pi_\omega(m)}{\underbrace{m+\omega(m)}})- M^{-1}(v_j)\\
&=M^{-1}(M(n) + v_j)-M^{-1}(v_j)=M^{-1}(M(n))=n.
\end{align*}
\end{itemize}

We claim that $\Phi$ is invertible. For $n\in \Z^d$, we define $j_n$ to be the index of the coset for which $n\in H_{j_n}$. Given $\omega_1,\dots \omega_{k}\in \Omega(A)$ and $n\in \Z^d$, we define
\[ \omega(n)=M(\omega_{j_n}(M^{-1}(n-v_{j_n})))\in M\cdot A ,\]
and $\Psi(\omega_0,\dots,\omega_{k-1})=\omega$. We observe that for any $j$, $\omega_j$ defines the restriction of  $\pi_\omega$ to the coset $H_j$. We may repeat the same arguments used in the first part of the proof (in reversed order) to show that this restriction is a permutation of the coset $H_j$. Thus $\pi_\omega$ is a permutation of $\Z^d$ and $\omega \in \Omega(M\cdot A)$. It is easy to verify that $\Psi$ is exactly the inverse function of $\Psi$, and thus $\Psi$ and $\Phi$ are bijections.

For $n\in \N^d$, consider the set $F_n$ defined by
\[ F_n=\bigcup_{j=1}^k\parenv{M\cdot [n]+v_j}.\]
Denote by $C_n$ the set of patterns that are obtained by restricting  elements in $\Omega(M\cdot A)$ to $F_n$. That is,
\[ C_n\eqdef \mathset{\omega(F_n)~:~\omega \in \Omega(M\cdot A)}.\]
 We may use the map defined above in order to find a bijection between $C_n$ and $B_n(A)^k$. For a pattern $w\in C_n$ let $\omega\in \Omega(M\cdot A)$ such that $w=\omega(F_n)$, let $\Phi(\omega)=(\omega_1,\dots,\omega_k)$. Define $\phi(w)=(w_1,\dots,w_k)$ where $w_j$ is the restriction of $\omega$ to $[n]$. The fact that $\Phi$ is invertible suggests that this operation is invertible - given $w_1,\dots,w_k\in B_n(A)$ we may define 
\[ \psi(w_1,\dots,w_k)\eqdef \Psi(\omega_1,\dots,\omega_n )(F_n)\]
where $\omega_j\in \Omega(A)$ is such that $w_j=\omega_j([n])$. Clearly, $\psi$ is the inverse function of $\phi$.

It is easy to see that $F_n$ is in fact the intersection of a $d$-dimensional parallelepiped (which is convex) with $\Z^d$,  containing $k\cdot \abs{[n]}$ points. Thus, by Theorem $A$ in \cite{BalBolQua02}, 
\begin{align*}
\topent (\Omega(M\cdot A))&=\lim_{n\to\infty}\frac{\log_2\abs{\abs{C_n}}}{|F_n|}\\
&=\lim_{n\to\infty}\frac{\log_2\abs{B_{n}(A)^k}}{|F_n|}\\
&=\lim_{n\to\infty}\frac{k\log_2\abs{B_{n}(A)}}{k|[n]|}=\topent (\Omega(A)).
\end{align*}
\end{proof}

\begin{corollary}
\label{cor:AffInvEnt}
Let $A\subseteq \Z^d$ be a finite set. For $d\geq 2$, $b\in \Z^d$ and a matrix $M\in \GL(d,\Z)$,
\[ \topent (\Omega(A))=\topent (\Omega (M\cdot A+b) ).\] 
\end{corollary}
\begin{proof}
Follows directly from Propositions \ref{fact:ShiftInv} and \ref{prop:MatEnt}.
\end{proof}

\begin{definition}
\label{def:Aff}
Let $A=\mathset{x_1,x_2,\dots,x_n}$ be a finite set of points, contained in some vector space $V$ over a field $F$. The affine dimension of $A$, denoted by $\dim_{\text{aff}}(A)$, is defined to be the dimension of the vector space $V_A$, where 
\[ V_A\eqdef \mathset{\sum_{i=1}^n \alpha_i x_i~:~\alpha_1,\alpha_2,\dots, \alpha_d\in \R \such
\sum_{i=1}^n \alpha_i =0}.\]
We say that $A$ has full affine dimension if $\dim_{\text{aff}}(A)+1=\abs{A}$ and that the vectors composing $A$ are affinely independent.
\end{definition}

\begin{theorem}
\label{th:AffEqEnt}
Let $d\geq 2$ and $A,B\subseteq \Z^d$ be finite sets with full affine dimension such that $\abs{B}=\abs{A}=d'\leq d+1$.
Then, $\topent(\Omega(A))=\topent(\Omega(B))$. Furthermore, If $d=2$ and $d'=3$
\[ \topent(\Omega(A))=\topent(\Omega(B))=\frac{1}{4\pi^2}\intop_0^{2\pi}\intop_0^{2\pi}\log \abs{1+e^{ix}+e^{iy}}dxdy. \]
\end{theorem}

\begin{proof}
Let $a\in A$, and $A' \eqdef A-a$. It is easy to verify that the elements of $A'$ span the space $V_A$ from from Definition \ref{def:Aff}.  Let us enumerate the elements of $A'$ by $a_0,\dots,a_{d'-1}$, where $a_0=0$. By the assumption, $\dim(\Linspan(a_0,a_1,\dots,a_{d'-1}))=d'-1$ and therefore $a_1,\dots,a_{d'-1}$ are linearly independent. Thus, we can complete them to a basis of $\R^d$ with integer vectors $v_{d'},\dots, v_d$ in the case where $d'-1<d$. If $d'=d+1$, $\mathset{a_1,\dots,a_{d'-1}}$ is already a basis of $\R^d$.

Let $M$ be the unique $d\times d$ matrix that maps the ordered basis $(e_1,e_2, \dots, e_d)$  to the ordered basis $(a_1,a_2,\dots,a_{d'-1},v_{d'}, \dots, v_d)$ 
, where $e_i=(0,\dots,0,\overset{i}{1},0,\dots,0)$. Clearly $M$ has integer entries as the rows of $M$ are the vectors $a_1,a_2,\dots,a_{d'-1},v_{d'}, \dots, v_d$. Let $C_{d'}\eqdef \mathset{0,e_1,e_2,\dots,e_{d'-1}}$. from the construction of $M$ it follows that $M\cdot C_{d'}=A'=A-a$. Using Corollary \ref{cor:AffInvEnt} we obtain
 \[ \topent(\Omega(C_{d'}))=\topent(\Omega(M\cdot C_{d'}+a))=\topent(\Omega(A)).\]
If $d=2$ and $d'=3$ we note that $C_{3}=\mathset{(0,0),(0,1),(1,0)}$, which in the notation of Section \ref{A_L}, is the set $A_L$. By Theorem \ref{th:SolL}, 
\[  \topent(\Omega(A))=\topent(\Omega(A_L))=\frac{1}{4\pi^2}\intop_0^{2\pi}\intop_0^{2\pi}\log_2\abs{1+e^{ix}+e^{iy}}dxdy.\]
\end{proof}

	\section{The Entropy of $\Omega([-k,k])$ and Size of Balls in $\ell_\infty$ Metric on $S_n$ }
	\label{Closed_k}
In this part, we focus on one-dimensional restricted  permutations, that is, restricted permutations of $\Z$. We show the equivalence between closed restricted permutations and balls in $\ell_\infty$ metric on permutation spaces. We examine the relations between the closed and regular topological entropy.

Given a subset of of integer numbers, $F\subseteq \Z$, we consider the $\ell_\infty$ metric on  $S(F)$ (the set of permutations of $F$), given by 
\[ d_\infty(f,g)\eqdef \norm{f-g}_\infty=\sup_{n\in N}\abs{f(n)-g(n)},\]
and balls in $d_\infty$ metric, given by 
\[ B(f,k)\eqdef \mathset{g\in S(F) ~:~ d_\infty(f,g)\leq k}.\] 
As in \cite{DezHua98}, $d_\infty$ is a right invariant metric, that is $d_\infty(f,g)=d_\infty(fh,gh)$ for all $f,g,h\in S(F)$. Thus, for all $f,g\in S(F)$, 
\[d_\infty(f,g)=d_\infty(I,fg^{-1})=d_\infty(I,gf^{-1}), \]
where $I$ denotes the identity function on $F$. We conclude that the size of a ball in $\ell_\infty$ metric does not depend on the center of the ball. In particular, the size of a ball of radius $k$, denoted by $B(F,k)$, is given by 
\[ B(F,k)=\abs{B(I,k)}=\abs{\mathset{f\in S(F)~:~d_\infty(I,f)\leq k}}.\]

In our terminology, $B(I,k)$ is exactly the number of permutations of $F$, restricted by the set $[-k,k]$. Thus, for $F=[n]$, with the notation from Chapter \ref{CHA:preliminaries},  
\[ B([n],k)=\abs{\mathset{f\in S([n])~:~d_\infty(I,f)\leq k}}=\abs{B_n^f([-k,k])},\]
and the asymptotic ball size, defined to be $\limsup_{n\to\infty}\frac{\log (B([n],k))}{n}$, is exactly the closed entropy, $\topent_c (\Omega([-k,k]))$. The asymptotic and non asymptotic ball size in $\ell_\infty$ metric were studied in detail in \cite{Sch09,SchVon17,Klo08,Klo09,Lag62}. 

In this section, we will use the special structure of $\Omega([0,k])$, discovered by Schmidt and Strasser in \cite{SchStr17}, in order to show that the closed and regular topological entropy of $\Omega([-k,k])$ are equal. This will prove that the asymptotic ball size in the $\ell_\infty$ is given by the entropy of $\Omega([-k,k])$.

\begin{definition}
A one-dimensional SFT $\Omega \subseteq\Sigma^{\mathbb{Z}}$
is called irreducible  if for all $ n_{1},n_{2}\in\mathbb{N}$
and $a\in B_{n_{1}}\left(\Omega \right)$, $b\in B_{n_{2}}\left(\Omega \right)$
there exists some $n_{3}\in\mathbb{N}$ and $c\in B_{n_{3}}\left(\Omega \right)$
such that $abc\in B_{n_{1}+n_{2}+n_{3}}\left(\Omega \right)$,
where $abc$ is the concatenation of $a$, $b$ and $c$. Let $\Irre \left(\Omega \right)$
be the set of all irreducible SFTs included in $\Omega $.
An irreducible component of $\Omega $
is a maximal element in $\Irre \left(\Omega \right)$ with respect
to the inclusion order.
\end{definition}
\begin{fact} (\cite{LinMar85}, Theorem 4.4.4)
\label{fact:subEnt}
Any SFT $\Omega\subseteq\Sigma^{\mathbb{Z}} $ has a finite number of irreducible components $\Omega_{1},\Omega_{2},\dots,\Omega_{n}$ such that $\Omega=\bigcup_{i=1}^{n}\Omega_{i}$ and $\topent 
(\Omega)=\max_{1\leq i\leq n} \topent (\Omega_{i})$. 
\end{fact}
\begin{fact} 
\label{fact:conjIr}
Let $\Omega_{1}\subseteq\Sigma_{1}^{\mathbb{Z}}$,
$\Omega_{2}\subseteq\Sigma_{2}^{\mathbb{Z}}$ be topologically conjugate SFTs and $\phi:\Omega_{1}\to\Omega_{2}$ be a conjugacy map. If $\Omega\subseteq\Omega_{1}$ is an irreducible component of $\Omega_{1}$, then $\phi\left(\Omega\right)$ is an irreducible component of $\Omega_{2}$ and $h\left(\Omega\right)=h\left(\phi\left(\Omega\right)\right)$.
\end{fact} 

Fact \ref{fact:conjIr} is followed from the equivalence between irreducibility and topological transitivity, which is invariant under conjugacy. See Example 6.3.2. in \cite{LinMar85} for further details.

Schmidt and Strasser studied the decomposition of $\Omega([k+1])$ into irreducible components, and the properties of these irreducible components.
\begin{theorem} (\cite{SchStr17}, Section 2)
\label{th:SchIr}
\begin{enumerate}
\item For every $\omega\in\Omega\left(\left[k+1\right]\right)$ there exists
an integer $a(\omega) \in [k+1]$ such that 
\[ \abs{\sum_{n=m}^{m+N-1} \omega_{n}-Na\left(\omega\right)} <k^{2} \]
 for every $m\in\mathbb{Z}$, $N\in\mathbb{N}$. This integer can
be viewed as the average shift of $\mathbb{Z}$, imparted by the
permutation $\pi_{\omega}$. Moreover, $a\left(\omega\right)$ is
given by 
\[ a\left(\omega\right)=\abs{\left\{ j\in\left[-k,-1\right]\,~:~\,\pi_{\omega}\left(j\right)\geq0\right\}} .\]
\item The irreducible components of $\Omega([k+1])$ are $\left\{ \Omega\left(\left[k+1\right]\right)_{l}\right\} _{l\in\left[k+1\right]}$,
where 
\[\Omega([k+1])_l\eqdef \mathset{\omega\in \Omega([k+1]) \ ~:~ \ a(\omega)=l} .\]
\item For any $l\in [k+1]$
, the subshifts $\Omega\left(\left[k+1\right]\right)_{l}$
and $\Omega\left(\left[k+1\right]\right)_{k-l}$ are topologically conjugate.
\item For $0\leq l<\frac{k}{2}$, $\topent\left(\Omega\left(\left[k+1\right]\right)_{l}\right)\leq\topent \parenv{\Omega\left(\left[k+1\right]\right)_{l+1}}$.
\item $\left|\Omega\left(\left[k+1\right]\right)_{0}\right|$=$\left|\Omega\left(\left[k+1\right]\right)_{k}\right|=1$,
and for $l\in\left[1,k-1\right]$, the topological entropy $h\left(\Omega\left(\left[k+1\right]\right)_{l}\right)$
satisfies 
\[ \left(1-\frac{l}{k}\right) \log\left(l+1\right)\leq h\left(\Omega\left(\left[k+1\right]\right)_{l}\right)\leq\log\left(l+1\right).\]
\item For any $l\in [1,k-1]$, let $\binom{[k]}{l}\subseteq 2^{[k]}$ denote the set of all subsets of $[k]$ containing exactly $l$ elements. Define $M_{k+1,l}$ to be the $\binom{k}{l}\times \binom{k}{l}$ matrix with entries from $\mathset{0,1}$, satisfying $M(A,B)=1$ if and only if one of the following condition is satisfied:
\begin{itemize}
\item $0\notin A$ and $B=A-1\eqdef \mathset{a-1~:~a\in A}$.
\item $0\in A$ and $B=(A'-1)\cup \mathset{j}$ for some $j\in [k]\setminus A'$, where $A'\eqdef A\setminus \mathset{0}$.
\end{itemize}
$\Omega([k+1])_l$ is topologically conjugated to $X_{k+1,l}$ defined by 
\[ X_{k+1,l}\eqdef \mathset{(A_n)_{n\in \Z}\in \binom{k}{l}^{\Z}~:~M(A_n,A_{n+1}) \text{ for every }n\in \Z  }.\] 
\end{enumerate}
\end{theorem} 
We use the conjugacy of $\Omega([-k,k])$ and $\Omega([0,2k+1])$ (see Proposition \ref{prop:EntProp}) and convert their results to $\Omega([-k,k])$.
\begin{definition}
for $l\in\left[2k+1\right]$, we define
\[  \Omega\left(\left[-k,k\right]\right)_{l}\eqdef \left\{ \omega\in\Omega\left(\left[-k,k\right]\right)\,~:~\,f\left(\omega\right)=l\right\}
,\]
where
\[ f\left(\omega\right)\eqdef \abs{\left\{ j\in\left[-2k,-1\right]\,~:~\,\pi_{\omega}\left(j\right)\geq-k\right\}}. \]
\end{definition}
\begin{proposition}

\label{prop:pro2k} \
\begin{itemize}
\item For $l\in\left[2k+1\right]$, the sets $\left\{ \Omega\left(\left[-k,k\right]\right)_{l}\right\} _{l\in\left[2k+1\right]}$
are the irreducible components of $\Omega\left(\left[-k,k\right]\right)$.
\item For $l\in\left[2k+1\right]$, the subshifts $\Omega\left(\left[-k,k\right]\right)_{l}$
and $\Omega\left(\left[-k,k\right]\right)_{2k-l}$ are topologically
conjugate.
\item For $0\leq l<k$, $\topent \left(\Omega\left(\left[-k,k\right]\right)_{l}\right)\leq \topent\parenv{\Omega\left(\left[-k,k\right]\right)_{l+1}}$.
\item $\left|\Omega\left(\left[-k,k\right]\right)_{0}\right|$=$\left|\Omega\left(\left[k\right]\right)_{2k}\right|=1$,
and for $l\in\left[1,2k-1\right]$, the topological entropy $\topent \left(\Omega\left(\left[-k,k\right]\right)_{l}\right)$
satisfies
\begin{align*}
\left(1-\frac{l}{2k}\right) & \log\left(l+1\right)\leq \topent \left(\Omega\left(\left[-k,k\right]\right)_{l}\right)\leq\log\left(l+1\right).\\
\end{align*}
\item $\Omega([-k,k])$ is topologically conjugated to $X_{2k+1,l}$ from Proposition \ref{th:SchIr}.6.  
\end{itemize}
\end{proposition}

\begin{proof}
By Proposition \ref{prop:EntProp}, $\Omega\left(\left[2k+1\right]\right)$ and $\Omega\left(\left[-k,k\right]\right)$ are topologically conjugated.
It is easy to verify that the function $\Phi_{-k}:\Omega\left(\left[2k+1\right]\right)\to\Omega\left(\left[-k,k\right]\right)$ defined by $\Phi_{-k}(\omega)\eqdef \sigma_{-k}\circ \omega$ is a conjugacy map.
By the definitions of $a\left(\omega\right)$ and $f\left(\omega\right)$,
it is easy to see that $a\left(\omega\right)=f\left(\Phi_{-k}\left(\omega\right)\right)$
for all $\omega\in\Omega\left(\left[2k+1\right]\right)$. Thus, $\Phi_{-k}\left(\Omega\left(\left[2k+1\right]\right)_{l}\right)=\Omega\left(\left[-k,k\right]\right)_{l}$
for any $l\in\left[2k+1\right]$. Since $\left\{ \Omega\left(\left[2k+1\right]\right)_{l}\right\} _{l\in\left[2k+1\right]}$
are the irreducible components of $\Omega\left(\left[2k+1\right]\right)$
(Theorem \ref{th:SchIr}), and $\Phi_{-k}$ is an conjugacy, by Fact \ref{fact:conjIr}, it follows that 
\[ \left\{ \Phi_{-k}\left(\Omega\left(\left[2k+1\right]\right)_{l}\right)\right\} _{l\in\left[2k+1\right]}=\left\{ \Omega\left(\left[-k,k\right]\right)_{l}\right\} _{l\in\left[2k+1\right]} \]
are the irreducible components of $\Omega\left(\left[-k,k\right]\right)$.
The rest follows immediately from fact \ref{fact:conjIr}. 
\end{proof}

\begin{corollary}
\label{cor:AvMeaning}
For all $\omega\in \Omega\parenv{[-k,k]}$ and for all $m\in\mathbb{Z}$,
\[ \abs{\mathset{j\in[m,m+2k-1]~:~\pi_\omega (j) \geq m+k}}=f(\omega).\]
\end{corollary}
\begin{proof}
By proposition \ref{prop:pro2k}, $\omega$ is contained in the irreducible component $\Omega\parenv{[-k,k]}_{f(\omega)}$. Consider the shift left by $m+2k$ of $\omega$, denoted by $\sigma_{m+2k}(\omega)$. Since irreducible components are shift invariant, $\sigma_{m+2k}(\omega)\in \Omega\parenv{[-k,k]}_{f(\omega)}$ as well. Hence,
\begin{align*}
\abs{\mathset{j\in[m,m+2k-1]~:~\pi_\omega (j) \geq m+k}}&=\abs{\mathset{j\in[-2k,-1]~:~\pi_{\sigma_{m+2k}(\omega)} (j) \geq -k}}\\
&=f(\sigma_{m+2k}(\omega))=f(\omega).
\end{align*}
\end{proof}

\begin{theorem}
\label{th:eqd1}
For any $n\in\mathbb{N}$ such that $n>2k$, 
\[ \left|B_{n}\left(\Omega\left(\left[-k,k\right]\right)_{k}\right)\right|\leq\left|B_{n+2k}^{f}\left(\left[-k,k\right]\right)\right|,\]
Where $B_{n}\parenv{\Omega([-k,k])_k}$ is the set of patterns of length $n$ which appear in elements of $\Omega\left(\left[-k,k\right]\right)_k$.
\end{theorem}

\begin{proof}
We want to find an injection $B_{n}\left(\Omega\left[-k,k\right]_{k}\right)\to B_{n+2k}^{f}\left(\left[-k,k\right]\right)$.
Let $a=\left(a_{0},\dots,a_{n-1}\right)\in B_{n}\left(\Omega\left(\left[-k,k\right]\right)_{k}\right)$,
by the definition of $B_{n}\left(\Omega\left(\left[-k,k\right]\right)_{k}\right)$, there exists $\omega\in\Omega\left(\left[-k,k\right]\right)_{k}$ for which $a=\left(\omega_{0},\omega_{2},\dots,\omega_{n-1}\right)$.
Our goal now is to define a permutation $\pi_a:\left[n+2k\right]\to\left[n+2k\right]$ which is restricted by $[-k,k]$, using $a$.\\
\textsl{Step 1} - Defining $\pi_a':[k,n+k-1]\to [n+2k]$, which we later extend to a permutation of $[n+2k]$.  For $m\in\left[k,n+k-1\right]$ define
\[ \pi_a'(m)= m+a_{m-k}= \pi_{\omega}\left(m-k\right)+k.\]
Clearly, $\pi_a':[k,n+k-1]\to [n+2k]$ 
defined by so far is injective and restricted by $[-k,k]$ since $\pi_{\omega}$ is a permutation of $\mathbb{Z}$, and in particular its restriction to $[n]$ is injective.\\
\textsl{Step 2} - Extending $\pi_a'$  to $\pi_a'':[n+k,n+2k-1]\to [n+2k]$. Let $H\subseteq [n+2k]$ be the set of all indices in $[n+2k]$ which are uncovered by the image of $\pi_a'$. Formally,
\[ H  \eqdef \left\{ j\in\left[n+2k\right]\,~:~\,\forall i\in\left[k,n+k-1\right],\pi_a'\left(i\right)\neq j\right\}=[n+2k]\setminus \ima(\pi_a'). \]
Clearly, 
\[|H|=|[n+2k]|-|\ima(\pi_a')|=|[n+2k]|-|[k,n+k-1]|=2k\]
since $\pi_a$ is injective and $\ima(\pi_a')\subseteq [n+2k]$. Consider the sets
\begin{align*}
S_r & \eqdef \left\{ j\in\left[n-k,n+k-1\right]\,~:~\,\pi_a'\left(j\right)\geq n\right\}=(\pi_a')^{-1}([n,n+2k-1]) \\
 & \text{and}\\
D_r & \eqdef \left\{ j\in\left[n,n+2k-1\right]\,~:~\,\exists i\in\left[k,n+k-1\right]\,\text{s.t. }\pi_a\left(i\right)=j\right\} \\
 & =\text{Im}\left(\pi_a\right)\cap\left[n,n+2k-1\right]
\end{align*}
Recall that $\omega\in B_{n}\left(\Omega\left(\left[-k,k\right]\right)_{k}\right)$
thus, by Corollary \ref{cor:AvMeaning}, $\left|S_r\right|=k$. Since $\pi_a'$ is restricted by $[-k,k]$ ($\pi_a'\left(m\right)-m\in\left[-k,k\right]$ for all $m\in\left[k,n+k-1\right]$), we have $(\pi_a')^{-1}\parenv{D_r}=S_r$. By injectivity of $\pi_a'$ we have 
\[\abs{S_r}=\abs{D_r}=k.\] 
Let  $H_r$ be the set of  indices in $[n,n+2k-1]$ which are uncovered by the image of $\pi_a'$, that is 
\[ H_r=H\cap\left[n,n+2k-1\right].\] 
We note that $H_r=\left[n,n+2k-1\right]\setminus D_{r}$, and $D_r\cap H_r=\emptyset$.
Therefore , 
\[ \left|H_r\right|=|[n,n+2k-1]|- |D_r|=2k-k=k.\]
Let $\left\{ j_{0},j_{1},\dots,j_{k-1}\right\} $ be the elements of $H_r$ ordered such that 
$n\leq j_{0}<j_{1}<\cdots<j_{k-1}\leq n+2k-1$ . It is easy to check
that for all $n+k\leq i\leq n+2k-1$ we have that $\left|i-j_{i-n-k}\right|\leq k$, as in the worst case $H_r=[n,n+k-1]$ and equality holds.
We can now define $\pi_a'':[k,n+2k-1]\to [n+2k]$ by
\[ \pi_a''\left(i\right)=\begin{cases}
j_{i-n-k} & \text{ if } i\in [n+k,n+2k-1]\\
\pi_a'(i) & \text{ otherwise}
\end{cases}.\] 
By the construction, it is clear that the extended $\pi_a''$ defined
so far on $\left[k,n+2k-1\right]$ is injective and restricted by $[-k,k]$. See Figure \ref{fig:ClOpK} for a visualization of step 2.\\ 
\textsl{Step 3} Extending $\pi_a''$ to $\pi_a:[n+2k]\to [n+2k]$: in a similar way to Step 2, we denote, $H_l=H_{\sigma}\cap\left[2k\right]$
and we claim that $H=H_l\uplus H_r$.  If it wouldn't be true, we would have $H_l\uplus H_r\subsetneq H $, and there exists $i\in [n+k,n-k]$ which is not covered by the image of $\pi_a'$. That would be a contradiction to the fact that $\pi_\omega$ used to define $\pi_a
$ is a permutation of $\Z$ which is restricted by $[-k,k]$. We obtain  
\[ |H_l|=|H|-|H_r|=2k-k=k.\]

Let $H_l=\left\{ l_{0},l_{1},\dots,l_{k-1}\right\} $ where
$0\leq l_{0}<l_{1}<\cdots<l_{k-1}\leq2k-1$. By similar arguments as in the previous step, for all $0\leq i\leq k-1$, we have $\left|i-l_{i}\right|\leq k$.
We define 
\[ \pi_a\left(i\right)=\begin{cases}
j_{i-n-k} & \text{ if } i\in [k]\\
\pi_a''(i) & \text{ otherwise}
\end{cases}.\] 
Now we have $\pi_a$ which is completely defined on $\left[n+2k\right]$
and it is injective, so it is a permutation. The map $a\to\pi_a$ is
obviously 1-1 as $a\left(l\right)=\pi_a\left(l+k\right)-\left(l+k\right)$
for all $l\in\left[n\right]$ so $a$ can be completely restored from
$\pi_a$. 
\begin{figure}
 \centering
  \
  \includegraphics[width=150mm, scale=0.6]{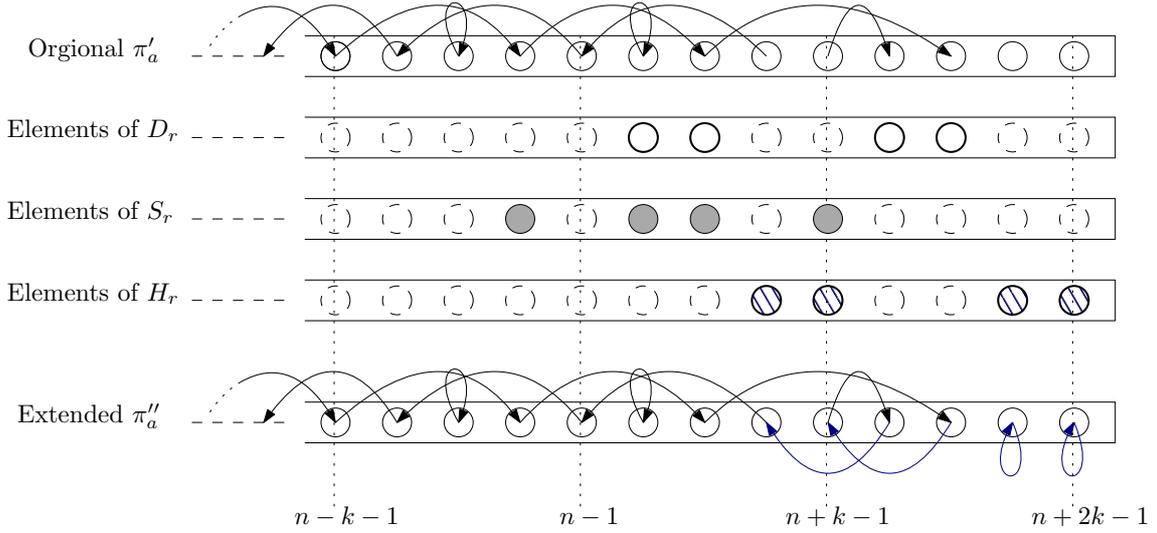}
  \caption{A demonstration of the construction of $\pi_a''$ from $\pi_a'$ where $k=4$. }
  \label{fig:ClOpK}
\end{figure}
\end{proof}

\begin{fact} (\cite{LinMar85}, Theorem 4.4.4)
\label{fact:SpectralRad}
Let $\Sigma$ be some finite set, $X\subseteq \Sigma^{\Z}$ be an irreducible SFT, and $M$ be a $\abs{\Sigma}\times \abs{\Sigma}$ matrix with  entries form $\mathset{0,1}$. If 
\[ \Omega \eqdef \mathset{(a_n)_{n\in \Z}\in \Sigma^{\Z}~:~M(a_n,a_{n+1}) \text{ for every }n\in \Z  },\]
then the topological entropy is given by $\topent (\Omega) =\log(\lambda_M)$ where $\lambda_M$ is the spectral radius of $M$. 
\end{fact} 

\begin{corollary}
\[ \topent_c(\Omega([-k,k]))=\topent(\Omega([-k,k]))=\log(\lambda_{M_{2k+1,k}}). \]
\end{corollary}

\begin{proof}
since $\topent_c (\Omega([-k,k])) \leq \topent(\Omega([-k,k]))$ (see Chapter \ref{CHA:preliminaries}), is it sufficient to show that $\topent_c(\Omega([-k,k])) \geq\topent(\Omega([-k,k]))$.
By Proposition \ref{prop:pro2k}, $\left\{ \Omega\left(\left[-k,k\right]\right)_{l}\right\} _{l\in\left[2k\right]}$, the irreducible components of $\Omega\left(\left[-k,k\right]\right)$
satisfy
\begin{align*}
\topent\left(\Omega\left(\left[-k,k\right]\right)_{0}\right) & \leq \topent\left(\Omega\left(\left[-k,k\right]\right)_{1}\right)\leq\cdots\leq \topent\left(\Omega\left(\left[-k,k\right]\right)_{k}\right).\\
\end{align*}
and 
\begin{align*}
\topent\left(\Omega\left(\left[-k,k\right]\right)_{l}\right) =  \topent\left(\Omega\left(\left[-k,k\right]\right)_{2k-l}\right).
\end{align*}
Thus, by Theorem \ref{fact:subEnt}, $h\left(\Omega\left(\left[-k,k\right]\right)\right)=\max_{l\in\left[2k+1\right]}h\left(\Omega\left(\left[-k,k\right]\right)_{l}\right)=h\left(\Omega\left(\left[-k,k\right]\right)_{k}\right)$.
Now, we use Fact \ref{th:eqd1} to deduce:
\begin{align*}
\topent_c(\Omega([-k,k])) & =\limsup_{n\to\infty}\frac{\log\left(\left|B_{n}^{f}\left(\left[-k,k\right]\right)\right|\right)}{n}\\
 & \geq\limsup_{n\to\infty}\frac{\log\left(\left|B_{n-2k}\left(\Omega\left(\left[-k,k\right]\right)_{k}\right)\right|\right)}{n}\\
 & =\limsup_{n\to\infty}\underset{\to1}{\underbrace{\left(\frac{n-2k}{n}\right)}}\cdot\underset{\to \topent \parenv{\Omega\parenv{[-k,k]}_k} }{\underbrace{\frac{\log\left(\left|B_{n-2k}\left(\Omega\left(\left[-k,k\right]\right)_{k}\right)\right|\right)}{n-2k}}}\\
 & =\topent\left(\Omega\left(\left[-k,k\right]\right)_{k}\right)\\
 & =\topent\left(\Omega\left(\left[-k,k\right]\right)\right).
\end{align*}
On the other hand, the conjugacy of $\Omega([-k,k])_k$ and $X_{2k+1,k}$ and Fact \ref{fact:SpectralRad} implies that
\[ \topent_c(\Omega([-k,k]))=\topent(\Omega([-k,k])_k=\topent(X_{2k+1,k})=\log(\lambda_{M_{2k+1,k}}). \] 
\end{proof}
\begin{example}
For $k=3$, the asymptomatic ball size in $\ell_\infty$ distance,
\[ \limsup_{n\to \infty}\frac{\log(B([n],3))}{n}= \topent_c(\Omega([-k,k]))=\log(\lambda_{M_{7,3}}),\]
where $\lambda_{M_{7,3}}$ is the largest root of the polynomial
\begin{align*}
p(x)&= 1    -x    -3x^2    -5x^3    -9x^4   -17x^5   -21x^6    +20x^7    +28x^8     +8x^9    -4x^{10}    \\
&+12x^{11}    +16x^{12}    -4x^{13}   -13x^{14}   -11x^{15}     +3x^{16}     +x^{17}    -3x^{18}     +x^{19}     +x^{20}.
\end{align*}
\end{example}
		
			\section{Local and Global Admissibility}
	\label{LocGlob}
Given an $SFT$, $\Omega\subseteq \Sigma^{\Z^d}$, a finite set $U\subseteq \Sigma^d$, and a pattern $v\in \Sigma^U$, a natural question is weather this pattern is globally admissible, i.e., whether there exists $\omega\in \Omega$ such that the restriction of $\omega$ to $U$ is the pattern $v$. Generally, this question does not have a simple answer. It is proved in \cite{Rob71} that in the general case, it is not decidable whether a finite pattern is globally admissible, i.e., there is no algorithm that can decide  whether a finite pattern is globally admissible or not.

If $\Omega$ is defined by the set of forbidden patterns $F$,  a necessary condition for global admissibility is local admissibility. We say that a pattern $v\in \Sigma^U$ is locally admissible if it does not contain any of the forbidden patterns in $F$. That is, for any forbidden pattern $p \in F \cap \Sigma^{U'}$ and $n\in \Z^d$ such that $U'\subseteq \sigma_n(U)$, $\parenv{\sigma_n(v)}(U')\neq p$. 
Clearly, if a pattern is globally admissible, it is also locally admissible. However, local admissibility does not imply global admissibility. See Example \ref{Ex:LocAd} for a pattern which is locally admissible but not globally admissible in the context of restricted permutations.

In the context of restricted permutations, for a finite restricting set $A\subseteq \Z^d$, a pattern $v\in A^U$ is identified with a function $f_v: U\to U+A$, defined by $f_v(n)=n+v(n)$. The pattern $v$ is globally admissible if it is the restriction of some $\omega\in \Omega(A)$. For such $\omega\in \Omega$, we have that $f_v$ is the restriction of the permutation $\pi_\omega\in \Omega(A)$ to the set $U$. Thus, global admissibility of $v$ is equivalent to the existence of a permutation $\pi\in S(\Z^d)$, restricted by $A$, extending $f_v$.

In Proposition \ref{prop:NecessExt} we present a description of the conditions for local admissibility. In Proposition \ref{prop:compli2} we show that local admissibility of rectangular patterns is sufficient for global admissibility in two cases of restricting sets. We use these results for bounding entropy in Section \ref{Bounds} and for counting rectangular patterns in in Section \ref{A_L}.

\begin{definition}
Let $A,U\subseteq \Z^d$ be some finite sets. The boundary of $U$ with respect to $A$, denoted by $\partial (U,A)$, is defined to be the set of all indices $u\in U$ for which $u-A\not\subseteq U$. The interior of $U$ with respect to $A$ is defined to be $\Int(U,A)\eqdef U\setminus \partial (U,A)$
\end{definition}

\begin{proposition}
\label{prop:NecessExt}
Let $A\subseteq \Z^d$ be a finite non-empty restricting set and $U\subseteq \Z^d$ be some set. If a pattern $v\in A^U$ is globally admissible, then $f_v:U\to U+A$ defined by $f(n)=n+v(n)$ is injective and $\Int(U,A)\subseteq \ima(f_v)$. 
\end{proposition}

\begin{proof}
Assume that $v$ is globally admissible, and let $\omega\in \Omega(A)$ be such that $\omega([n])=v$. We note that $f_v$ is the restriction of $\pi_\omega$ to $U$, where $\pi_\omega:\Z^2\to \Z^2$ is the permutation defined by $\pi_\omega(m)=m+\omega(m)$. Clearly, $f_v$ is injective since $\pi_\omega$ is injective (as a permutation). Since $\pi_\omega$ is surjective, $\Int(U,A)\subseteq \ima( \pi_\omega)$. On the other hand, by the definition if $\Int(U,A)$ 
\[ \pi_\omega^{-1}(\Int(U,A)) \subseteq \Int(U,A)- A \subseteq U,\]
as $\pi_\omega$ is restricted by $A$.  Thus, 
\[ \Int(U,A)\subseteq \pi_\omega(U)=\ima(f_v). \]
\end{proof}

The local admissibility conditions presented in Proposition \ref{prop:NecessExt} are necessary for global admissibility. The following example shows that they are not sufficient.

\begin{example}
\label{Ex:LocAd}
Consider the restricting set $A_+=\mathset{(0,\pm 1),(\pm 1,0)}$ reviewed in Section \ref{A_+} and the set $U\eqdef [3]\times [5]\setminus \mathset{(1,2)}$. It is easy to see that restricted function $f:U\to U+A_+$ presented  in Figure \ref{fig:LocAdmiss} is injective. Furthermore, $\Int(U,A_+)$ is an empty set and therefore it is contained in the image of $f$ in a trivial way. So $f$ is locally admissible. Assume to the contrary that there exists $\pi\in\Omega(A_+))$ extending $f$. We note that for both $(1,1)$ and $(1,3)$, the only possible pre-image is $(1,2)$. Hence, $\pi$ cannot be surjective which is a contradiction. This shows that $f$ is not globally admissible.

\begin{figure}
 \centering
  \
  \includegraphics[width=115mm, scale=0.1]{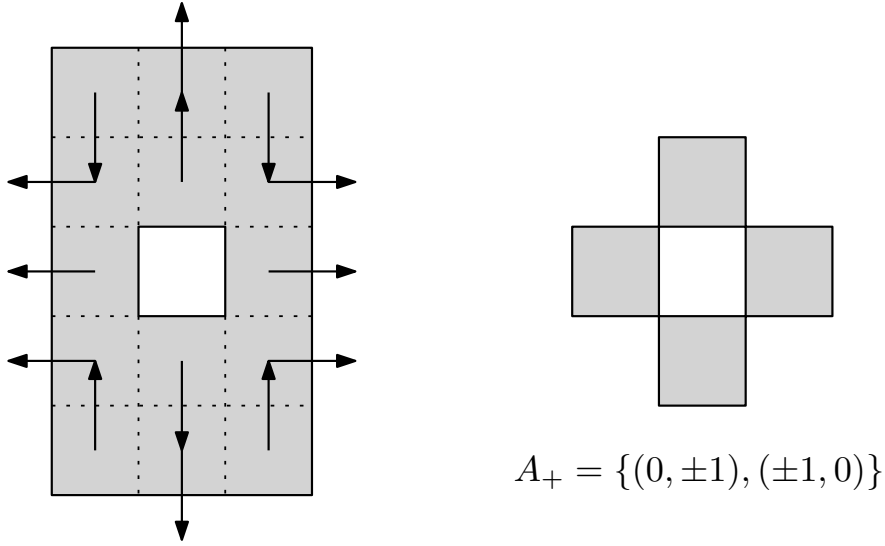}
  \caption{A locally admissible pattern which is not globally admissible where the restricting set is $A_+$.  }
  \label{fig:LocAdmiss}
\end{figure}
\end{example}

\begin{proposition}
\label{prop:compli2}
For $A\in \mathset{A_\oplus, A_L}$ (where $A_L$ and $A_\oplus$ are defined in Chapter \ref{CHA:preliminaries} and Section \ref{Bounds} respectively), let $(n_1,n_2)=n\in (\N^2\setminus [3]\times [3])$ and $v\in A^{[n]}$ be a rectangular pattern. Then $v$ is globally admissible  if and only if $f_v:U\to U+A$ defined by $f_v(n)=n+v(n)$ is injective and $\Int(U,A)\subseteq \ima(f_v)$. In the notation of Chapter \ref{CHA:preliminaries}, that is, $v\in B_n(\Omega(A)))$ if and only if $f_v$ is injective and  $\Int(U,A)\subseteq \ima(f_v)$.
\end{proposition}

\begin{proof}
We will show the proof for $A=A_\oplus$, the proof in the case that $A=A_L$ follows a similar idea.
The first direction (global admissibility implies injective $f_v$ and $\Int(U,A)\subseteq \ima(f_v)$) is is true by Proposition \ref{prop:NecessExt}. 

For the other direction, we first note that $\Int([n],A_\oplus)=[1,n_1-2]\times [1,n_2-2]$. 
Assume that $\pi_v$ is injective and $[1,n_1-2]\times [1,n_2-2]\subseteq \ima(\pi_v)$, we need to find a restricted permutation $\pi\in (\Omega(A_\oplus))$ such that the restriction $\pi([n])$ is $\pi_v$. Consider the sets 
\[ O_v=\ima(\pi_v)\setminus([n]),\]
and 
\[ H_v=[n]\setminus \ima(\pi_v). \]
We observe that $v\in A_\oplus^{[n]}$, which implies that $\ima(\pi_v)\subseteq [n]+A_\oplus$. Thus, $O_v\subseteq ([n]+A_\oplus)\setminus [n]$. We denote $([n]+A_\oplus)\setminus [n]$ by $\partial^{A_\oplus} [n]$. By the assumption, $\pi_v$ satisfies the second condition. Hence, $H_v\subseteq \partial([n],A_\oplus)=[n]\setminus [[1,n_1-2]\times [1,n_2-2]]$. We observe that for $m\in \partial^{A_\oplus} [n] $, there exists a unique vector $e_m\in {\mathset{(0,\pm 1), (\pm 1,0)}}$ such that $m+e_m\in [n]$. Similarly, for $m\in \partial ([n],A_\oplus)$, there exists a unique vector $e_m\in {\mathset{(0,\pm 1), (\pm 1,0)}}$ such that $m+e_m\notin [n]$. We now define $\pi$ on $\Z^2\setminus [n]$. 
\begin{itemize}
\item For $m\in O_v$ such that $m+e_m\in H_v$, define $\pi(m)=m+e_m$.
\item For $m\in O_v$ such that $m+e_m\notin H_v$, we define $\pi(m)=m-e_m$. Furthermore, for all $k\in \N$ we define $\pi(m-k\cdot e_m)= m-(k+1)\cdot e_m$.
\item For $m\in H_v$ such that $\pi(m+e_m)\notin O_v$, we define $\pi(m+k\cdot e_m)=m+(k-1)\cdot e_m$ for all $k\in \N$.
\item For any index $m\in \Z^2\setminus [n]$ not defined in the first three items, we define $\pi(m)=m$.
\end{itemize} 
Clearly, $\pi$ is restricted by $A_\oplus$ as $e_m\in A_\oplus$ for any $m\in H_v \cup O_v$. It may be verified that $\pi$ is bijective.
\begin{figure}
 \centering
  \
  \includegraphics[width=130mm, scale=0.1]{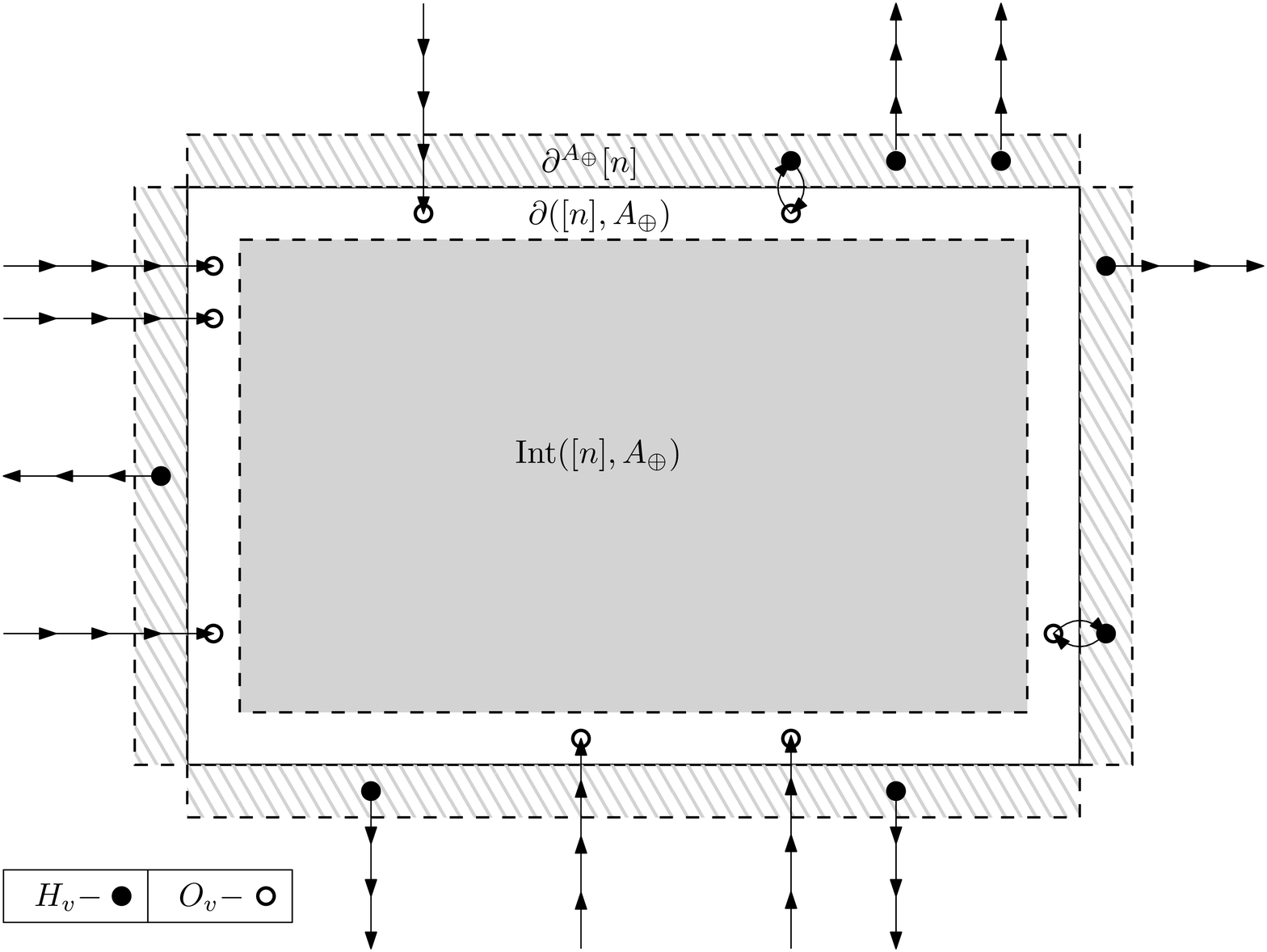}
  \caption{The extension of $\pi_v$ to a $\pi\in \Omega(A_\oplus)$. }
  \label{fig:Ext}
\end{figure}
\end{proof}
See Figure \ref{fig:Ext} for a demonstration of the procedure of defining $\pi$.

			\section{Entropy Bounds}
	\label{Bounds}
In Chapter \ref{CHA:PM} we found analytical expressions of the entropy of restricted permutations of $\Z^2$ in two cases, by using the theory of perfect matching of bipartite $\Z^2$-periodic planar graph. Such a calculation of the entropy is possible for a restricting set, $A\subseteq \Z^2$, if $G_A$ (defined in Chapter \ref{CHA:preliminaries}) or the corresponding graph from Theorem \ref{th:GenCan}, $G_A'$, are bipartite $\Z^2$-periodic and planar. Unfortunately, this is usually not the case. In this section, we will bound the entropy for such example. Having the case where $|A|=3$ solved (see Theorem \ref{th:AffEqEnt}), and a solution for one case where $\abs{A}=4$ (see Section \ref{A_+}), we will focus on an elementary example where $\abs{A}=5$.

Consider the set $A_\oplus\eqdef \mathset{(0,0),(0,\pm 1),(\pm 1, 0)}$. Note that the corresponding graph from Theorem \ref{th:GenCan} it has a $\Z^2$-periodic representation that has intersecting edges (see Figure \ref{fig:Gen_AO+}). Furthermore, the graph $G_{A_\oplus}$ contains self loops (as $(0,0)\in A_\oplus$), and therefore we cannot use the alternative correspondence from  \ref{th:PerToPM}.   
\begin{figure}
 \centering
  \
  \includegraphics[width=150mm, scale=0.5]{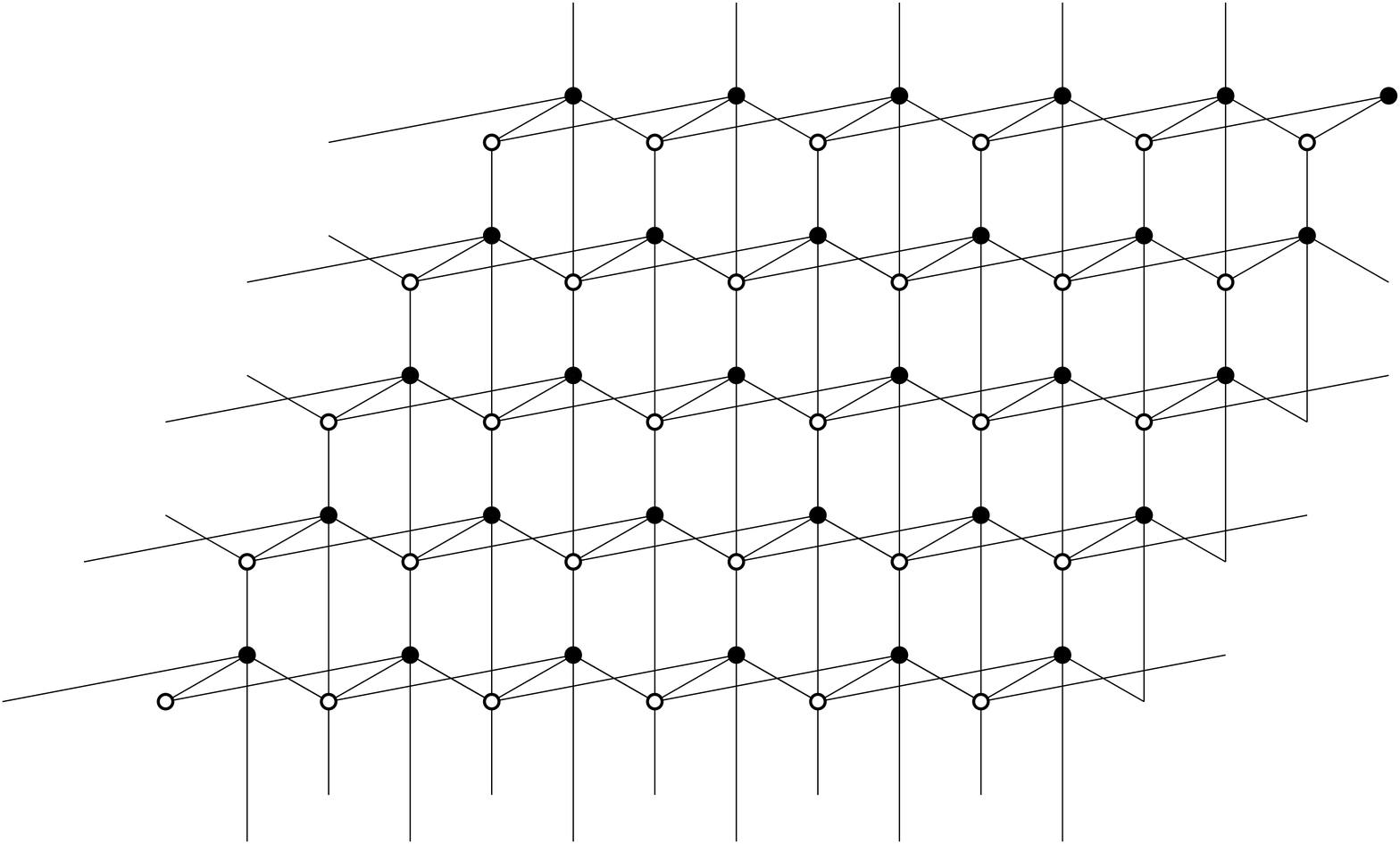}
  \caption{The corresponding graph for $G_{A_\oplus}$ from Theorem \ref{th:GenCan}. }
  \label{fig:Gen_AO+}
\end{figure}
We will bound $\topent(\Omega(A_\oplus))$ by the entropy of one-dimensional SFTs. 

\begin{definition}
Let $G=(V,E)$ be some finite directed graph. The adjacency matrix of $G$ is defined to be $|V|\times |V|$ matrix, $M$, with entries from $\mathset{0,1}$ such $M_G(v,u)=1$ if and only if $(v,u)\in E$. The vertex shift of $G$ is defined to be
\[ X_G\eqdef \mathset{(v_n)_{n\in \Z}\subseteq V^{\Z}: (v_n,v_n+1)\in E \text { for all }n\in \Z}.\] 
\end{definition}

\begin{fact} \label{fact:VerShift}(\cite{LinMar85}, Proposition 2.3.9, Theorem 4.4.4) Any vertex shift, $X_G$ is a one dimensional SFT and its entropy is given by 
\[ \topent(X_G)=\lambda_{M_G},\]
where $\lambda_{M_G}$ is the spectral radius of $M_G$.
\end{fact}

For $m\in \N$, the horizontal infinite stripe of width $m$ is defined to be 
\[ S_m\eqdef \mathset{n=(n_1,n_2)\in \Z^2 : n_2\in [m]}=\Z\times [m]\subseteq \Z^2. \] 
Denote by $\Omega_{\oplus,m}$ the set of all permutations of $S_m$ restricted by $A_\oplus$. In the usual manner, we identify it as a subset of $A_\oplus^{S_m}$. We aim to show that $\Omega_{\oplus,m}$, with the one dimensional shift operation of $\Z$ given by $\sigma_m(\omega)(n)=\omega(n+(0,m))$, is a one-dimensional SFT over $(A_\oplus^m)^\Z$. Furthermore, we will show that $\frac{\topent(\Omega_{\oplus,m})}{m}\leq \topent(\Omega(A_\oplus))$.
\begin{proposition}
\label{prop:StripeSFT1} 
$\Omega_{\oplus,m}$ is a one-dimensional SFT for any $m\in \N$.
\end{proposition}
\begin{proof}
In order to show that $\Omega_{\oplus,m}$ is an SFT, we will find a  finite directed graph such that the elements of $\Omega_{\oplus,m}$  (when considered as elements $(A_\oplus^m)^\Z$)  are exactly the set of bi-infinite paths on that graph. This will show that $\Omega_{\oplus,m}$ is a vertex shift, and by Fact \ref{fact:VerShift}, it is an SFT.  Given $k$ column vectors of length $m$, $v_0,\dots, v_{k-1}\in A_\oplus^m$, we can identify them with a function $\varphi_{v_,\dots,v_{k-1}}:([k]\times [m]\to[k]\times [m])+A_\oplus$ by $\varphi_{v_0,\dots,v_{k-1}}(n_1,n_2)=(n_1,n_2)+v_{n_1}(n_2)$. Consider $G_m=(V,E)$ where 
\[ V=\mathset{(v,u)\in A_\oplus^m\times A_\oplus^m:\ima(\varphi_{v,u})\subseteq S_m \text{ and }\varphi_{v,u}\text{ is injective }}\]
and 
\[ E=\mathset{((v,u),(u,w))\in V^2:\varphi_{v,u,w}\text{ is injective and }\mathset{1}\times [m]\subseteq \ima(\varphi_{v,u,w})}.\]
States in $G$ are pairs of vectors corresponding with an injective function, mapping the $[2]\times[m]$ to the horizontal stripe of width $m$. Edges are just triples of vectors corresponding with an injective function such that its image covers the middle column of $[3]\times [m]$.

Now, we will show that bi-infinite paths in $G_m$ encodes bijectively element in $\Omega_{\oplus,m}$. For an infinite sequence $w=(v_n,u_n)_{n\in \Z}\subseteq V^{\Z}$, define $\pi_w:S_m\to S_m$ by
\[ \pi_w(n_1,n_2)\eqdef (n_1,0)+\varphi_{v_{n_1},u_{n_1}}(0,n_2)=(n_1,n_2)+v_{n_1}(n_2). \] 
\begin{itemize}
\item  From the construction of $w$ we can see that \[\pi_w(n_1,n_2)=\varphi_{v_{n_1},u_{n_1}}\in S_m+(n_1,0)=S_m\]. 
\item $\pi_w$ is injective: Assume to the contrary that $\pi_w(n_1,n_2)=\pi_w(l_1,l_2)$ for $(n_1,n_2)\neq (l_1,l_2)$. The movements of $\pi_w$ are restricted in $A_\oplus$ by its definition, thus  $(n_1,n_2)-(l_1,l_2)\in A_\oplus - A_\oplus\subseteq [-2,2]^2 $, in particular $\abs{n_1-l_1}\leq 2$, without loss of generality, $n_1\leq l_1$. If $l_1=n_1+1$ we have that $(v_{n_1},v_{l_1})\in V$ and by the construction of $w$,
\begin{align*}
\overset{\varphi_{v_{n_1},v_{l_1}}(0,n_2)+(n_1,0)}{\overbrace{(n_1,n_2)+v_{n_1}(n_2)}}=\pi_w(n_1,n_2)&\underset{\Downarrow}{=}\pi_w(l_1,l_2)=\overset{\varphi_{v_{n_1},v_{l_1}}(1,l_2)+(n_1,0)}{\overbrace{(l_1,l_2)+v_{l_1}(l_2)}}\\
\varphi_{v_{n_1},v_{l_1}}(0,n_2)&=\varphi_{v_{n_1},v_{l_1}}(1,l_2). 
\end{align*} 
Meaning that $\varphi_{v_{n_1},v_{l_1}}$ is not injective, which is a contradiction.  If $n_1=l_1$ we will similarly obtain $\varphi_{v_{n_1},v_{n_1+1}}(0,n_2)=\varphi_{v_{n_1},v_{n_1+1}}(0,l_2)$ in contradiction. The remaining case is where $l_1=n_1+2$. In this case we note that $((v_{n_1},v_{n_1+1}),(v_{n_1+1},v_{l_1}))\in E$ and we have 
\begin{align*}
\varphi_{v_{n_1},v_{n_1+1},v_{l_1}}(0,n_2)+(n_1,0)=\pi_w(n_1,n_2)&\underset{\Downarrow}{=}\pi_w(l_1,l_2)=\varphi_{v_{n_1},v_{n_1+1},v_{l_1}}(2,l_2)+(n_1,0)\\
\varphi_{v_{n_1},v_{n_1+1},v_{l_1}}(0,n_2)&=\varphi_{v_{n_1},v_{n_1+1},v_{l_1}}(2,l_2). 
\end{align*} 
Meaning that $\varphi_{v_{n_1},v_{n_1+1},v_{l_1}}$ is not injective, which is a contradiction.
\item $\pi_w$ is onto $S_m$: It is sufficient to show that for any $n\in \Z$, the column $\mathset{n}\times [m]$ is contained in $\ima(\pi_w)$. By the construction, the restriction of $\pi_w$ to coordinates $(i,j)\in \mathset{{n-1,n,n+1}}\times [m]$ is given by 
\[ \pi_w(i,j) = \varphi_{v_{n-1},v_n,v_{n+1}}(i-n+1,j)+(n-1,0)\]
and
\begin{align*}
\mathset{n}\times [m]=\mathset{1}\times [m]+(n-1,0)\subseteq \varphi_{v_{n-1},v_n,v_{n+1}}([2]\times[m])+(n-1,0),
\end{align*}
as $((v_{n-1},v_n)(v_n,v_{n+1}))\in E$.
We note that 
\[  \varphi_{v_{n-1},v_n,v_{n+1}}([2]\times[m])+(n-1,0) =\pi_w([n-1,n+1]\times [m]), \] 
and we have
\[\mathset{n}\times [m]\subseteq \pi_w([n-1,n+1]\times [m]).\] 
\end{itemize}

We have so far shown that the map $(v_n,u_n)_{n\in \Z}\to w \in \Omega_{\oplus,m}$ is well defined. Clearly it is bijective, as its inverse is given by $w\to (v'_n,u'_n)_{n\in \Z} $ where $v'_n$ and $u'_n$   are just the columns indexed by $n$ and $n+1$  in $w$ respectively. 
\end{proof}
\begin{theorem}
\label{th:PlusLo}
For any $m\in \N$, 
\[ \topent(\Omega(A_\oplus))\geq \frac{\topent(\Omega_{\oplus,m})}{m}.\]
\end{theorem}
\begin{proof}
If for all $k,n\in \N$ we can show that $\abs{B_n(\Omega_{\oplus,m})}^k\geq \abs{B_{n,m\cdot k}(\Omega(A_\oplus)}$, by Fact \ref{prop:fekete},  we will obtain
\begin{align*}
\topent(\Omega(A_\oplus))&=\lim_{n_1,n_2\to\infty } \frac{\log_2\abs{B_{n_1,n_2}(\Omega(A_\oplus)))}}{n_1n_2}\\
&=\lim_{k,n\to\infty}\frac{\log_2\abs{B_{n,m\cdot k}(\Omega(A_\oplus))}}{k m n}\\
&\geq\lim_{k,n\to\infty}\frac{\log_2\abs{B_n(\Omega_{\oplus,m})}^k}{k m  n}\\
&=\lim_{k,n\to\infty}\frac{\log_2\abs{B_n(\Omega_{\oplus,m})}^k}{ m  n}=\frac{\topent(\Omega_{\oplus,m})}{m}.
\end{align*}

It remains to show that for given $k,n\in \N$, $\abs{B_n(\Omega_{\oplus,m})}^k\geq \abs{B_{n,m\cdot k}(\Omega(A_\oplus)}$. Let $(w_0,w_2,\dots,w_{k-1})\in B_n(\Omega_{\oplus,m})^k$. By the definition of $B_n(\Omega_{\oplus,m})$ each $w_i$ is a path of length $n$ in $G_m'$, where $G_m'$ is the graph generating $\Omega_{\oplus,m}$, defined in the proof of Proposition \ref{prop:StripeSFT1}. Denote it by $\parenv{v_k^{(i)},u_k^{(i)}}_{k=1}^n$. As we saw in the proof of Proposition \ref{prop:StripeSFT1}, for all $i=1,2,\dots,k$, the path repenting $w_i$ is the restriction of for some permutation $f_i\in \Omega_{A_\oplus,m}$ to the rectangle $[n]\times [m]$.

Now we construct a permutation of $\Z^2$, $\pi\in \Omega(A_\oplus)$ using $f_1,\dots, f_k$. Note that $\bigcup_{l\in \Z}\parenv{S_m+(0,l\cdot m)}=\Z^2$. Hence if we define $\pi$ on $\parenv{S_m+(0,l\cdot m)}_{l\in \Z}$ and show that the restriction of $\pi$ to $S_m+(0,l\cdot m)$ is a permutation of $S_m+(0,l\cdot m)$ we can conclude that $\pi \in \Omega(A_\oplus)$. Given $l\in\Z$ and $n=(n_1,n_2)\in S_m$ define  
\[ \pi((0,l\cdot m)+n)\eqdef f_{l\bmod k} (n)+(0,l\cdot m).\] 
Clearly, the restriction of $\pi$ to $S_m+(0,l\cdot m)$ is indeed a permutation of $S_m+(0,l\cdot m)$ since it is only a shift by $(0,l\cdot m)$ of $f_{l\bmod k} $ which is a permutation of $S_m$. Note that $\pi$ is restricted by $A_\oplus$ as it has the same displacements as $f_0,f_1,\dots,f_{k-1}$, which are restricted by $A_\oplus$. That is, $\pi\in \Omega(A_\oplus)$. We note that the restriction of $\omega_\pi$ to $[n]\times [m\cdot k]$ is exactly $ \begin{pmatrix}
w_{k-1} & \dots & w_1 & w_0 \end{pmatrix}^t $, which is the $n\times (m \cdot k)$ array defined by
\[ \begin{pmatrix}
w_{k-1} & \dots & w_1 & w_0 \end{pmatrix} ^t\eqdef  \begin{pmatrix}
w_{k-1}\\
\vdots\\
w_1\\
w_0
\end{pmatrix}.  \]
Thus, the embedding 
\[ (w_0,w_1,\dots,w_{k-1})\longrightarrow(w_{k-1},\dots,w_1,w_0)^t \in B_{n,m\cdot k}(\Omega(A_\oplus))\] 
is well defined. Obviously it is injective as $(w_0,w_1,\dots,w_{k-1})$ can be reconstructed from $(w_{k-1},\dots,w_1,w_0)^t $. This completes the proof.
\end{proof}

We use similar method of approximating the entropy by the entropy one-dimensional stripe like SFTs in order to derive an upper bound. In the part, we will use the same notation as in the proof of the lower bound. 
For $m\in \N$, let $\Omega_{\oplus,m}'$ be the set of all injective functions $f:S_m\to S_m+A_\oplus$, restricted by $A_\oplus$ such that $\Z\times [1,m-2] \subseteq \ima(f)$. That is, the set of all $A_\oplus$-restricted  functions from the stripe, which are injective and having image which covers the interior of the stripe $S_m$.  As before, we identify such function with elements in $A_\oplus^{S_m}$ in the usual manner. 
\begin{proposition}
\label{prop:1dSFT2}
$\Omega_{\oplus,m}'$ is a one dimensional SFT for all $m\in \N$.
\end{proposition}
\begin{proof}
The proof of this proposition is very similar to the proof of Proposition \ref{prop:StripeSFT1}. Consider the graph $G_m'=(V,E)$ defined by 
\[ V=\mathset{(v,u)\in A_\oplus^m\times A_\oplus^m: \varphi_{v,u}\text{ is injective }}\]
and 
\[ E=\mathset{((v,u),(u,w))\in V^2:\varphi_{v,u,w}\text{ is injective and }\mathset{1}\times [1,m-2]\subseteq \ima(\varphi_{v,u,w})},\]
where $\varphi_{v,u,w}$ and $\varphi_{v,u}$ are the functions defined by the vectors $v,u$ and $w$ the functions, as decried in the previous section. We show that any bi-infinite path in $G_m'$ corresponds bijectively to a function in $\pi\in \Omega_{\oplus,m}'$. Let $(v_n,u_n)_{n\in \Z}$ be such path. By the definition of $E$, $u_n=v_{n-1}$ for all $n$. We define $w\in A_\oplus^{S_m}$ by $w(n_1,n_2)=v_{n_1}(n_2)$ for all $n_1\in \Z$ and $n_2\in [m]$. First, let us show that $\pi_w:S_m\to S_m+A_\oplus$ identified with $w$ by $\pi_w(n_1,n_2)=(n_1,n_2)+w(n_1,n_2)$ is indeed an element in $\Omega_{\oplus,m}'$.

Injectivity is proved exactly the same as in the proof of Proposition \ref{prop:StripeSFT1}.
It remains to show that the image of $\pi_w$ covers the interior of $S_m$. It is sufficient to show that for any $n\in \Z$, the column $\mathset{n}\times [1,m-2]$ is contained in $\ima(\pi_w)$. By the construction, the restriction of $\pi_w$ to coordinates $(i,j)\in \mathset{{n-1,n,n+1}}\times [1,m-1]$ is given by 
\[ \pi_w(i,j) = \varphi_{v_{n-1},v_n,v_{n+1}}(i-n+1,j)+(n-1,0)\]
and  
\[ \mathset{n}\times [1,m-2]=\mathset{1}\times [1,m-2]+(n-1,0)\subseteq \varphi_{v_{n-1},v_n,v_{n+1}}([2]\times[m])+(n-1,0), \]
as $((v_{n-1},v_n)(v_n,v_{n+1}))\in E$.
From the definition of $\pi_w$ we have
\[ \pi_w([n-1,n+1]\times [m])=\varphi_{v_{n-1},v_n,v_{n+1}}([2]\times[m])+(n-1,0),\]
and therefore
\[ \mathset{n}\times [1,m-2]\subseteq \varphi_{v_{n-1},v_n,v_{n+1}}([2]\times[m])+(n-1,0)= \pi_w([n-1,n+1]\times [m]).\] 

 We have proved by now that the map $(v_n,u_n)_{n\in \Z}\to w\in \Omega'_{+,m}$ is well defined. Clearly it is bijective, as its inverse is given by $w\to (v'_n,u'_n)_{n\in \Z} $ where $v'_n$ and $u'_n$   are just the columns indexed by $n$ and $n+1$  in $w$ respectively.
\end{proof}
\begin{theorem}
\label{th:PlusUp}
For any $m\in \N$, 
\[ \topent(\Omega(A_\oplus))\leq \frac{\topent(\Omega'_{\oplus,m})}{m}.\]
\end{theorem}
\begin{proof}
Repeating the calculation from the proof of Theorem \ref{th:PlusLo}, if we show that for all $n,k\in \N$,  $\abs{B_{n,m\cdot k}(\Omega(A_\oplus))}\leq \abs{B_n(\Omega_{\oplus,m}')}^k$ we conclude
\[ \topent(\Omega(A_\oplus))=\lim_{k,n\to\infty}\frac{\log \abs{B_{n,m\cdot k}(\Omega(A_\oplus))}}{nmk}\leq \frac{\topent(\Omega'_{\oplus,m})}{m}. \]
Let $w\in B_{n,m\cdot k}(\Omega(A_\oplus))$. By the definition of $B_{n,m\cdot k}(\Omega(A_\oplus))$, there exists $\omega\in \Omega(A_\oplus))$ (representing a permutation of $\Z^2$, denoted by $\pi_\omega$ ),   such that its restriction to $[n]\times [k\cdot m]$ is $w$. Now we define functions $f_0,f_1,\dots,f_{k-1}:S_m\to S_m+A_\oplus$ by 
\[ f_l(i,j)\eqdef (i,j)+\omega\parenv{(i,j)+(0,l\cdot m)}=\pi_w(i,j+l\cdot m)-(0,l\cdot m) \] 
where the equality on the is followed by the definition $ \pi_\omega(i',j')\eqdef (i',j')+\omega(i',j')$.
We claim that $f_l\in \Omega_{\oplus,m}'$ for all $l$. Clearly, $f_l$ is restricted by $A_\oplus$ by its definition and the fact that $\omega$ is an element in $A_\oplus^{\Z^2}$. We observe that $f_l$ is obtained by shifting $\pi_\omega$ by $(0,l\cdot m)$ and restricting it to $S_m$. Thus, injectivity is followed immediately from the injectivity of $\pi_\omega$. Followed by this observation, we note that 
\[ f_l(S_m)=\pi_\omega(S_m+(0,l\cdot m))-(0,l\cdot m).\]
Since $\pi_\omega$ is restricted by $A_\oplus\subset [-1,1]^2$, and it is a onto $\Z^2$, we have that 
\begin{align*}
\Z \times [1,m-2] +(0,l\cdot m)&\subseteq \pi_\omega (\parenv{\Z \times [1,m-2] +(0,l\cdot m)}- \overset{[-1,1]^2}{\overset{\cup}{A_\oplus}})\\ 
&\subseteq \pi_\omega \parenv{\parenv{\Z \times [1,m-2] +(0,l\cdot m)}- [-1,1]^2}\\ 
&= \pi_\omega \parenv{\overset{S_m}{\overbrace{\Z \times [0,m-1]}} +(0,l\cdot m)}\\
&= \pi_\omega\parenv{S_m+(0,l\cdot m)}=f_l(S_m)+(0,l\cdot m).
\end{align*}
This shows that $f_l\in \Omega_{\oplus,m}'$. When we consider $f_l$ as an element in $A_\oplus^{S_m}$ we note that 
\[ f_l\parenv{[n]\times [m]}=\omega\parenv{[n]\times [m]+(0,l\cdot m)}=w\parenv{[n]\times [m]+(0,l\cdot m)}. \]
As shown in the proof of Proposition \ref{prop:1dSFT2}, $f_l\parenv{[n]\times [m]}$ represents a path of length $n$ in the graph describing $\Omega_{\oplus,m}'$, which is a word in $B_n(\Omega_{\oplus,m}')$. 
Thus, the map 
\[ w\longrightarrow \phi(w) \eqdef \begin{pmatrix}
w\parenv{[n]\times [m]}\\
w\parenv{[n]\times [m]+(0, m)}\\
w\parenv{[n]\times [m]+(0,2 m)}\\
\vdots\\
w\parenv{[n]\times [m]+(0,(k-1)  m)}
\end{pmatrix}\in B_n(\Omega_{\oplus,m}')^k\]
is well defined. It is also injective as $w$ can be trivially reconstructed from $\phi(w)$.
\end{proof}

The lower and bounds provided in Theorem \ref{th:PlusLo} and \ref{th:PlusUp}  may be computed for any $m\in \N$, as by fact \ref{fact:VerShift}, $\topent(\Omega_{\oplus,m})=\log(\lambda_{M_{G_m}})$ and $\topent(\Omega_{\oplus,m}')=\log(\lambda_{M_{G'_m}})$, where $\lambda_{M_{G_m}}$ is the spectral radius of the adjacency matrix of $G_m$ (and similarly for $\lambda_{M_{G'_m}}$).  The Achilles' heel of these bounds, is in the complexity of computing them. The dimension of the adjacency matrix $M_{G_m}$ is the number of vertices in $G_m$, which increase proportionately to $5^{cm}$ for some $c>0$. For example, for $m=5$, $M_{G_5}$ is a $66572\times 66572$ matrix, and $\lambda_{M_{G_4}}$ is not computable using standard computational power. We have the same problem with the upper bound. Computing the lower bound for $m=4$ and the lower bound for $m=3$ we obtain
\[ 1.01904 \leq \topent(\Omega(A_{\oplus})\leq 1.63029.\] 

			\section{The Entropy of Injective and Surjective Functions}
	\label{InjSurj}
So far, we have explored restricted permutations of graphs which are bijective functions. In this part of the work we examine the related models of restricted injective and surjective functions on graphs. We will show that under similar assumptions as in the case of permutations, the spaces of restricted injective and surjective  functions also have the structure of topological dynamical system. Finally, we examine the entropy of restricted injective / surjective functions on $\Z^d$, compared to the entropy of restricted permutations.

Let $G=(V,E)$ be some locally finite and countable directed graph. Similarly to  Chapter \ref{CHA:preliminaries}, a function $f:V\to V$ is said to be restricted by $G$ if $(v,f(v))\in E$ for all $v\in V$. We define the spaces of restricted injective and surjective functions of $G$ to be 
\[ \Omega_I(G)\eqdef \mathset{\varphi:V\to V: \varphi \text{ is injective and restricted by }G},\]
and
\[ \Omega_S(G)\eqdef \mathset{\varphi:V\to V: \varphi \text{ is surjective and restricted by }G}\]
respectively.

The spaces $\Omega_I(G)$ and $\Omega_S(G)$ are compact topological spaces, when equipped with the product topology (when $V$ has the discrete topology).  If $H$ is a group acting on $G$ by graph isomorphisms, it induces a homeomorphic group action on $\Omega_I(G)$ and  $\Omega_S(G)$ by conjugation. This is proven in a similar fashion as in the case of restricted permutations (see Chapter \ref{CHA:preliminaries}). Therefore, we will leave the details to the reader.

Throughout most of our work, we focused on  $\Z^d$-permutations, restricted by some finite set $A\subseteq \Z^d$. That is, permutations of $G=(\Z^d,E_A)$, where $E_A=\mathset{(n,n+a):n\in \Z^d, a\in A}$ and $\Z^d$ is acting on itself by translations.  In that case, the dynamical system $(\Omega(G_A))$, when considered as a subset of $A^{\Z^d}$,  is an SFT. That is also true in the case of $\Omega_I(G_A)$ and $\Omega_S(G_A)$. Similarly as in the case of permutations, we use the shorter notation of $\Omega_I(A)$ for $\Omega_I(G_A)$ and  $\Omega_S(A)$ for $\Omega_S(G_A)$.

\begin{proposition}
For any finite non-empty $A\subseteq \Z^d$, $\Omega_I(A)$ and $\Omega_S(A)$ are SFTs, when we identify a restricted function  $\varphi$ with an elements $A^{\Z^d}$ by $\omega_\varphi(n)=\varphi(n)-n$. 
\end{proposition}
\begin{proof}
A function  $\varphi:\Z^d\to \Z^d$ is injective if and only if the pre-image of any singleton is empty or a singleton. Note that if $\varphi$ is restricted by $A$, then $\varphi^{-1}(\mathset{m})\subseteq m-A$ for any $m\in \Z^d$. Thus, in order to check if a restricted function is injective, it is sufficient to check a local condition. Consider the set of patterns
\[ F_I\eqdef \mathset {w\in A^{-A}: \abs{\mathset{ n\in -A:  w(n)+n=0}}>1}\subseteq A^{\fin(Z^d)}.\]
We observe that $\abs{\varphi^{-1}(\mathset{m})}\leq 1$ if and only if $(\sigma_{m}\omega_\varphi)(-A)\notin F_I$. Thus, $\Omega_I(A)$ is the SFT defined by the set of forbidden patterns $F_I$. Similarly it is proven that  $\Omega_S(A)$ is an SFT, defined by the set of forbidden patterns 
\[F_S\eqdef  \mathset {w\in A^{-A}: \forall n\in -A, w(n)+n\neq 0}\subseteq A^{\fin(Z^d)}. \]
\end{proof}

For any finite non-empty set $A\subseteq\Z^d$, we note that $\Omega(A)$ is exactly the intersection of $\Omega_I(A)$ and $\Omega_S(A)$. Furthermore, we observe that $\Omega(A)$ is strictly contained in both  $\Omega_I(A)$ and $\Omega_S(A)$. Thus, $\topent(\Omega(A))\leq \min\mathset{\topent(\Omega_I(A)),\topent(\Omega_S(A))}$. This give rise to the natural question, whether the entropy of restricted permutations can be strictly smaller. Theorem \ref{th:InjSurj} provides a negative solution to the above question.

\begin{definition}
Given a measurable space $(X,\cF)$ and a group $H$ acting on $X$ by measurable transformations. A probability measure $\mu:\cF\to [0,1]$ is said to be invariant under the action of $H$ if for any measurable set $A\subseteq \cF$ and $h\in H$ we have 
\[\mu\parenv{h^{-1}(A)}=\mu(A). \]
 we define $M_H(X)$ to be the set of all probability measures on $X$ which are invariant under the action of $H$. A measure $\mu \in M_H(X)$ is called ergodic if it assigns invariant sets with $0$ or $1$. That is, $h^{-1}A=A$ for all $h\in H$  implies that $\mu(A)\in \mathset{0,1}$.
\end{definition}

\begin{proposition}
\label{prop:SuppMeasure}
Consider the measurable spaces $(\Omega_I(A),\cB_I)$ and  $(\Omega_I(A),\cB_S)$ with $\Z^d$ acting by shifts, where $ \cB_I$ and $\cB_S$ are the Borel $\Sigma$-algebras on $\Omega_I(A)$ and $\Omega_S(A)$ respectively. For any $\mu \in M_{\Z^d}(\Omega_I(A))$ and $\nu\in M_{\Z^d}(\Omega_S(A))$, 
\[ \Pro_\mu[\Omega(A)]=\Pro_\nu[\Omega(A)]=1. \] 
\end{proposition}

\begin{proof}
First, we prove the theorem for ergodic measures in $M_{\Z^d}(\Omega_I(A))$.  Recall that each element $\omega\in A^{\Z^d}$ is identified with a restricted function $\Z^d \to \Z^d$  by $f_\omega(n)\eqdef n+\omega(n)$. For $n\in \Z^d$ consider the function $P_n:A^{\Z^d}\to \N\cup \mathset{0}$ which assigns each $\omega\in A^{\Z^d}$ the number of pre-images of $n$. That is,
\[ P_n(\omega)\eqdef \abs{f_\omega^{-1}(\mathset{n})}=\abs{\mathset{m\in \Z^d: m+\omega(m)=n}}.\]  
Since $A$ is finite, there exists $M\in \N$ such that $A\subseteq [-M,M]^d$. For $(n_1,n_2,\dots,n_d)=n\in \N^d\setminus [-2M,2M]^d $ and $\omega\in A^{\Z^d}$ we examine the average of the functions $(P_m)_{m\in [n]}$, denoted by $A_n(\omega)$,
\begin{align*}
A_n(\omega)=\frac{\sum_{m\in [n]}P_m(\omega)}{\abs{[n]}}=\frac{\sum_{m\in [n]} \abs{f_\omega^{-1}(\mathset{m})}}{\abs{[n]}}=\frac{ \abs{f_\omega^{-1}([n])}}{\abs{[n]}}.
\end{align*}
Since $A$ is bounded in $[-M,M]^d$ and the movements of $f_\omega$ are restricted by $A$, we have  
\[[n_1-2M]\times \cdots \times [n_d-2M] \subseteq  f_\omega^{-1}([n]) \subseteq [n_1+2M]\times \cdots \times [n_d+2M].\] 
Thus, if we choose $n_k\eqdef (k,k,\dots,k)$ for $k\in \N$, for sufficiently large $k$, 
\[\frac{(k-2M)^d}{k^d}\leq A_{n_k}(\omega)\leq \frac{(k+2M)^d}{k^d}, \]
and in particular, $\lim_{k\to \infty} A_{n_k}(\omega)=1$.

Consider the measurable space $(\Omega_I,\mathcal{B_I})$, where $\mathcal{B_I}$ is the Borel $\Sigma$-algebra on $\Omega_I$. We note that for any $n\in \Z^d$ and $\omega\in \Omega_I$, 
\[ P_n(\omega)=P_0\circ \sigma_n(\omega)=P_0(n\omega)\] 
where $\sigma_n$ is the shift operation on $\Z^d$ and $n\omega$ denotes the group action of $n$ on $\omega$.

Consider the sequence of cubes, $([n_k])_{k=1}^{\infty}$. It is easy to check that it is a Følner sequence. That is, for any $m\in \Z^d$, 
\[ \lim_{k\to \infty}\frac{\abs{[n_k] \triangle \sigma_m([n_k])}}{\abs{[n_k]}}=0. \] 
By the Pointwise Ergodic Theorem \cite{Lin01}, for an ergodic $\mu \in M_{\Z^d}(\Omega_I(A))$ the sequence $(A_{n_k})_k$ converge almost everywhere to $\E_\mu[P_0]$ (the expectation of $P_0$ with respect to the measure $\mu$). On the other hand, we saw that the sequence $(A_{n_k})_k$ converge pointwise to the constant function $1$. We conclude that $\E_\mu[P_0]=1$.

We observe that $P_0$ can take only the values $0$ and $1$ on $\Omega_I(A)$, as any function defined by an element in $\Omega_I(A)$ is injective. Hence, 
\[ 1=\E_\mu[P_0]=1\cdot \Pro_\mu[P_0=1]+0\cdot \Pro_\mu[P_0=0]=\Pro_\mu[P_0=1].\]
Since $\mu$ is invariant under the action of $\Z^d$, for all  $n\in \Z^d$, 
\[ \Pro_\mu[P_0=1]=\Pro_\mu[P_0\circ \sigma_n=1]=\Pro_\mu[P_n=1].\]

We note that for $\omega\in \Omega_I(A)$, we have that $\omega\in \Omega(A)$ (i.e., $\omega$ represents a permutation), if and only if $f_\omega$ is also surjective. That is, any $n\in \N$ has a unique pre-image, which in the notation of this proof, is equivalent to $P_n(\omega)=1$ for all $n\in \Z^d$. We conclude that 
\[ \Pro_\mu[\Omega(A)]=\Pro_\mu\sparenv{\bigcap_{n\in \Z^d}\mathset{P_n=1}}=1.\] 

Now we turn to prove the claim for general $\mu\in \Omega_I(A)$. If $\mu$ is a convex combination of ergodic measures, then the claim follows immediately from the first case. For a general $\mu\in M_{\Z^d}(\Omega_I(A))$, By the ergodic decomposition theorem (\cite{EinWar13}, Theorem 4.8), $\mu$ is in the closed convex hull of the ergodic measures. That is, there is a sequence of measures $(\mu_n)_n\subseteq M_{\Z^d}(\Omega_I(A))$ which converge to  $\mu$ in the weak-* topology, and $\mu_n$ is a convex combination of a ergodic measures for each $n$. We obtain,
\[ \mu(\Omega(A))=\lim_{n\to\infty}\mu_n(\Omega(A))=1. \]

The proof for $\Omega_S(A)$ is very similar. Considering the restriction of the functions $(P_n)_{n\in \Z^d}$ to $\Omega_S(A)$, we observe that they can only take values greater or equal $1$ as for any element $\omega\in \Omega_S(A)$, as $f_\omega$ is surjective. By similar arguments as in the previous case, for any ergodic invariant probability measure $\nu\in M_{\Z^d}(\Omega_S(A))$, we have $\E_\nu[P_n]=1$ for all $n\in \Z^d$. Since $P_n\geq 1$, we conclude that $\Pro_\nu [P_n=1]=1$ for all $n$ and therefore $\Pro_\nu [\Omega(A)]=1$. For a non ergodic measure, we continue in a similar fashion as in the proof of the claim for $\Omega_I(A)$. 
\end{proof}

\begin{theorem}
\label{th:InjSurj}
For any finite non-empty set  $A\subseteq\Z^d$, 
\[ \topent(\Omega(A))=\topent(\Omega_I(A))=\topent(\Omega_S(A)). \] 
\end{theorem}

\begin{proof}
By the variational principle \cite{Mis67}, 
\[ \topent(\Omega_I(A))=\max_{\mu\in M_{\Z^d}(\Omega_I(A))}H(\mu), \] 
where $H(\mu)$ is the measure theoretical entropy of $\mu$. Let $\mu_0\in M_{\Z^d}(\Omega_I(A))$ be a measure such that $ \topent(\Omega_I(A))=H(\mu_0)$. Proposition \ref{prop:SuppMeasure} suggest that $\Pro_{\mu_0}[\Omega(A)]=1$, therefore the restriction of $\mu_0$ to the subspace $\Omega(A)$ is an invariant probability measure on $\Omega(A)$. Using the  variational principle once again, 
\[ \topent(\Omega(A))=\max_{\mu\in M_{\Z^d}(\Omega(A))}H(\mu)\geq  H(\mu_0)=\topent(\Omega_I(A)).\] 
The other direction of inequality is trivial as $\Omega(A)\subseteq \Omega_I(A)$, so equality holds. 
The proof for $\Omega_S(A)$ is exactly the same.

\end{proof}


\chapter{Conclusion and Future Work}
\label{CHA:conclusion}

In this work, we studied restricted permutations of sets having a geometrical form. We suggested a generalization of the model of $\Z^d$-restricted movement permutations, presented in \cite{SchStr17},  to restricted permutations of graphs. We showed that in some settings, the space of restricted permutations has the structure of a dynamical system. This generalization lead us to the observation on the natural correspondence  between restricted permutations and perfect matchings. Using this correspondence and the theory of perfect matchings of $\Z^2$-periodic planar graphs, we gained a deep understanding of restricted $\Z^2$-permutations in some specific elementary cases.

In the second part of the work, we focused on the entropy of restricted $\Z^d$-permutations. We proved an important invariance property, which was later used for finding the entropy $\Z^2$-permutations, restricted by sets consisting of $3$ elements. We studied the entropy in the related models of restricted injective and surjective functions. Finally, We discussed the relation between global and local admissibility (mostly for rectangular patterns). We used the findings of this part in order to derive upper and lower  bounds on the entropy in another elementary two-dimensional case, in which we could not use the theory of perfect matchings.

The most precise results on restricted permutations, presented in our work, were achieved using the correspondence to perfect matchings of $\Z^2$-periodic planar graphs. Unfortunately, the corresponding graphs are usually non-planar. Very little is known to us in those cases. The difficulty in studying  non-planar cases can be viewed in the entropy bounds given in Section \ref{Bounds}. Despite significant efforts and many bounding techniques attempted, there is still a major gap between our upper and lower bound.

Estimating the entropy of multidimensional SFTs is known to be a difficult problem in general, and we have witnessed that the case of restricted permutation is not an exception. This motivates us to question whether we can use the theory of perfect matchings for estimating the entropy in non-planar cases as well. Results from \cite{Tes00} concerning the counting of perfect matchings of non-planar graphs indicate that the answer might be positive. 

It seems that using similar methods as in \cite{KenOkoShe06}, combined with the results in \cite{Tes00}, we may lower bound the entropy of $\Z^d$ restricting permutations by the Mahler measure of a polynomial in $d$ variables. The Mahler measure of a complex polynomial in $d$ variables, $p(x)$, is defined to be 
\[ M(p)\eqdef \frac{1}{(2\pi)^d}\intop_{[0,2\pi]^d}\log\parenv{\abs{p(e^{i\theta_1},\dots,e^{i\theta_d}})} d\bar{\theta}. \]
Relations between Mahler measures of polynomials and the entropy of certain families of SFTs, already reviewed in \cite{LinSch18}, are also demonstrated in our work. Theorem \ref{th:SolL} and Theorem \ref{th:Sol+} suggest that the entropy of $\Z^2$-permutations restricted by $A_+$ and $A_L$ are the Mahler measures of polynomials with $2$ variables. The general connection between the entropy of restricted $\Z^d$-permutations and Mahler measures is an open subject for future research.

Another important issue we discussed in details, was the relations between closed, periodic and regular (topological) entropy. From  the perspective of coding theory, this subject has great significance, as coding related applications use permutations of finite sets. We revealed the nature of those relations in the two-dimensional cases  of permutations restricted by $A_L$ and $A_+$, and in the one-dimensional case of permutations restricted by $[-k,k]$. However, in the general case we have no information, and we leave it for future work.



\bibliographystyle{IEEEtran}

\end{spacing}


\end{document}